\documentclass[11pt]{article}

\usepackage{amssymb}
\usepackage{graphicx}
\usepackage{color}
\usepackage{amsmath,amsthm,amsfonts,epsfig,setspace}
\numberwithin{equation}{section}

\newcommand{\ds}{\displaystyle}
\def\nm{\noalign{\medskip}}
\newtheorem{thm}{Theorem}[section]
\newtheorem{rmk}{Remark}[section]
\newtheorem{cor}{Corollary}[section]
\newtheorem{definition}{Definition}
\newtheorem{lem}{Lemma}[section]

\newtheorem{prop}{Propsition}[section]

\newtheorem{cond}{Condition}
 \def\p{\partial}
\def \Vh0{\stackrel{\circ}{V}_h} 
   
\def\Om{\Omega}  \def\om{\omega}
   
\def\l{\label}  \def\f{\frac} \def\df{\dfrac} 

   \def\eps{\varepsilon}

\def\e{\eta}

\def\l|{\left|}
\def\r|{\right|}

\DeclareMathOperator*{\argmax}{arg\,max}

\newcommand{\R}{\mathbb{R}}
\newcommand{\N}{\mathbb{N}}

\newcommand{\lc}
{\mathrel{\raise2pt\hbox{${\mathop<\limits_{\raise1pt\hbox
{\mbox{$\sim$}}}}$}}}

\newcommand{\gc}
{\mathrel{\raise2pt\hbox{${\mathop>\limits_{\raise1pt\hbox{\mbox{$\sim$}}}}$}}}

\newcommand{\ec}
{\mathrel{\raise2pt\hbox{${\mathop=\limits_{\raise1pt\hbox{\mbox{$\sim$}}}}$}}}

\def\be{\begin{equation}} \def\ee{\end{equation}}

\def\bea{\begin{eqnarray}}  \def\eea{\end{eqnarray}}

\def\beas{\begin{eqnarray*}} \def\eeas{\end{eqnarray*}}

\def\bn{\begin{enumerate}} \def\en{\end{enumerate}}

\def\bd{\begin{description}} \def\ed{\end{description}}

\title{Mathematical analysis of plasmonic resonances for nanoparticles: the full Maxwell equations\thanks{\footnotesize This work was supported  by the ERC Advanced Grant Project MULTIMOD--267184.}}
\date{}

\author{
Habib Ammari\thanks{\footnotesize Department of Mathematics, 
ETH Z\"urich, 
R\"amistrasse 101, CH-8092 Z\"urich, Switzerland (habib.ammari@math.ethz.ch, sanghyeon.yu@sam.math.ethz.ch). }
\and   Matias Ruiz\thanks{\footnotesize Department of Mathematics and Applications,
Ecole Normale Sup\'erieure, 45 Rue d'Ulm, 75005 Paris, France
(matias.ruiz@ens.fr).} 
\and Sanghyeon Yu\footnotemark[2]
\and  
Hai Zhang\thanks{\footnotesize 
Department of Mathematics, 
 HKUST,  Clear Water Bay, Kowloon, Hong Kong (haizhang@ust.hk).}
}

\begin{document}
\maketitle

\begin{abstract}
In this paper we use the full Maxwell equations for light propagation in order to analyze plasmonic resonances for nanoparticles. We mathematically define the notion of plasmonic resonance and analyze its shift and broadening with respect to changes in size, shape, and arrangement of the nanoparticles, using the layer potential techniques associated with the full Maxwell equations. We present an effective medium theory for resonant plasmonic systems and derive a condition on the volume fraction under  which the Maxwell-Garnett theory is valid at plasmonic resonances.  
\end{abstract}

\medskip

\bigskip

\noindent {\footnotesize Mathematics Subject Classification
(MSC2000): 35R30, 35C20.}



\section{Introduction} \label{sec-intro}

The aim of this paper is to analyze plasmon resonant nanoparticles.  
Plasmon resonant nanoparticles have unique capabilities of enhancing the brightness and directivity of light, confining strong electromagnetic fields, and outcoupling of light into advantageous directions \cite{SC10}. Recent advances in nanofabrication techniques have made it possible to construct complex nanostructures such as arrays using plasmonic nanoparticles as components. A thriving interest for optical studies of plasmon resonant nanoparticles is due to their recently proposed use as labels in molecular biology \cite{plasmon4}. New types of cancer diagnostic nanoparticles are constantly being developed. Nanoparticles are also being used in thermotherapy as nanometric heat-generators that can be activated remotely by external electromagnetic fields \cite{baffou2010mapping}.    
Plasmon resonances in nanoparticles can be treated at the quasi-static limit as an eigenvalue problem for the Neumann-Poincar\'e integral operator \cite{pierre, Gri12, plasmon1,plasmon3}. At this limit, they are size-independent. However, as the particle size increases, they are determined from scattering and absorption blow up and become size-dependent. This was experimentally observed, for instance, in \cite{tocho}. 

The objective of this paper is twofold: (i) To  analytically
investigate the plasmonic resonances of a single nanoparticle and analyze the shift and broadening of the plasmon resonance with changes in size and shape of the nanoparticles using the full Maxwell equations; (ii)  To derive a Maxwell-Garnett type theory for approximating the plasmonic resonances of a periodic arrangement of nanoparticles. The paper generalizes to the full Maxwell equations the results obtained in \cite{matias, hyeonbae} where the Helmholtz equation was used to model light propagation. It provides the first mathematical study of the shift in plasmon resonance using the full Maxwell equations. On the other hand, it rigorously shows the validity of the Maxwell-Garnett theory for arbitrary-shaped nanoparticles at plasmonic resonances.  
The paper is organized as follows. In section \ref{sec-Preliminaries max} we first review commonly used function spaces. Then we introduce layer potentials associated with the Laplace operator and recall their mapping properties. 
In section  \ref{sec-Layer potential fomulation max} we first derive a layer potential formulation for the scattering problem and then we obtain a first-order correction to plasmonic resonances in terms of the size of the nanoparticle.  
This will enable us to  analyze the shift and broadening of the plasmon resonance with changes in size and shape of the nanoparticles.  The resonance condition is determined from absorption and scattering blow up and depends on the shape, size and electromagnetic parameters of both the nanoparticle and the surrounding material. Surprisingly, it turns out that in this case not only the spectrum of the Neumann-Poincar\'e operator plays a role in the resonance of the nanoparticles,
but also its negative. We explain how in the quasi-static limit, only the spectrum of the Neumann-Poincar\'e operator can be excited. However, when the particle size increases and deviates from the dipole approximation, the resonances become size-dependent. Moreover, a part of the spectrum of negative of the Neumann-Poincar\'e operator can be excited as in higher-order terms in the expansion of the electric field versus the size of the particle. In section \ref{sec-quastatic max} we establish the quasi-static limit for the electromagnetic fields and derive a formula for the enhancement of the extinction cross-section. It is not clear for what kind of geometries in $\R^3$ the spectrum of the Neumann-Poincar\'e operator has symmetries, that is, if $\lambda \in \sigma (\mathcal{K}_D^*)$ so does $-\lambda$. In section \ref{sec-explicit shpere max} we provide calculations for the case of spherical nanoparticles wherein these symmetries are not present and we explicitly compute the shift in the spectrum of the Neumann-Poincar\'e operator and the extinction cross-section.  In section \ref{sectshell} we consider the case of a spherical shell and apply degenerate perturbation theory since the eigenvalues associated with the corresponding Neumann-Poincar\'e operator are not simple.  It is also worth mentioning that the spectrum of the associated Neumann-Poincar\'e operator is symmetric around zero. 
In section \ref{sec-Anisitrop max} we analyze the anisotropic quasi-static problem in terms of layer potentials and define the plasmonic resonances for anisotropic nanoparticles. Formulas for a small anisotropic perturbation of resonances of the isotropic formulas are derived. 
Finally, section \ref{sec-MaxGar max} is devoted to establish a Maxwell-Garnett type theory for approximating the plasmonic resonances of a periodic arrangement of arbitrary-shaped nanoparticles. The Maxwell-Garnett theory provides a simple model for calculating the macroscopic optical properties of materials with a dilute inclusion of spherical nanoparticles \cite{book2}. It is widely used to assign effective properties to systems of nanoparticles. We rigorously obtain effective properties of a periodic arrangement of arbitrary-shaped nanoparticles and  derive a condition on the volume fraction of the nanoparticles  that insures the validity of the Maxwell-Garnett theory for predicting the effective optical properties of systems of embedded in a dielectric host material at the plasmonic resonances.

\section{Preliminaries} \label{sec-Preliminaries max}
Let us first fix some notation, definitions and recall some useful results for the rest of this paper.
\begin{itemize}
\item For a simply connected domain $D\Subset\mathbb{R}^3$, $\nu$ denotes the outward normal to $\p D$ and $\f{\p}{\p \nu}$ the outward normal derivative;
\item $\varphi \big\vert_{\pm}(x) = \lim_{t\rightarrow 0^+}\varphi(x\pm t\nu)$;
\item $Id$ denotes the identity operator;
\item $\nabla\times$ denotes the curl operator for a vector field in $\R^3$;
\item For any functional space $E(\p D)$ defined on $\p D$, $E_0(\p D)$ denotes its zero mean subspace.
\end{itemize}
Here and throughout this paper, we assume that $D$ is simply connected and of class $\mathcal{C}^{1,\alpha}$ for $0<\alpha<1$.

Let $H^s(\p D)$ denote the usual Sobolev space of order $s$ on $\p D$ and $$H^{s}_T(\p D) = \left\{\varphi \in \big(H^{s}(\p D)\big)^3, \nu\cdot\varphi=0  \right\}.$$ Let $\nabla_{\p D}$, $\nabla_{\p D}\cdot$ and $\Delta_{\p D}$ denote the surface gradient, surface divergence and Laplace-Beltrami operator respectively and define the vectorial and scalar surface curl by $\vec{\text{curl}}_{\p D}\varphi = -\nu\times\nabla_{\p D}\varphi$ for $\varphi \in H^{\f{1}{2}}(\p D)$ and $\text{curl}_{\p D}\varphi = -\nu \cdot(\nabla_{\p D} \times \varphi)$ for $\varphi \in H^{-\f{1}{2}}_T(\p D)$, respectively.\\
Remind that
\beas
\nabla_{\p D}\cdot \nabla_{\p D} &=& \Delta_{\p D},\\
\text{curl}_{\p D}\vec{\text{curl}}_{\p D} &=& -\Delta_{\p D},\\
\nabla_{\p D}\cdot\vec{\text{curl}}_{\p D} &=& 0,\\
\text{curl}_{\p D}\nabla_{\p D} &=& 0.
\eeas

We introduce the following functional space:
\beas
H^{-\f{1}{2}}_T(\text{div},\p D) &=& \left\{ \varphi \in H^{-\f{1}{2}}_T(\p D), \nabla_{\p D}\cdot\varphi\in H^{-\f{1}{2}}(\p D) \right\}.
\eeas

Let $G$ be the Green function for the Helmholtz operator $\Delta + k^2$ satisfying the Sommerfeld radiation condition in dimension three
$$
\bigg| \frac{\partial G}{\partial |x|}   - i k G \bigg| \leq C |x|^{-2}
$$ 
for some constant $C$ as $|x| \rightarrow + \infty$, uniformly in $x/|x|$. 
 
The Green function $G$ is given by
\begin{equation} \label{green}
G(x,y,k)= -\frac{e^{i k |x-y|}}{4\pi |x-y|}.
\end{equation}
Define the following boundary integral operators
\bea \label{eq-SLPvect}
\vec{\mathcal{S}}_{D}^{k} [\varphi]: H^{-\f{1}{2}}_T(\p D) &\longrightarrow & H^{\f{1}{2}}_T(\p D) \\
\varphi &\longmapsto & \vec{\mathcal{S}}_{D}^{k} [\varphi](x) = \int_{\p D} G(x, y, k) \varphi(y) d\sigma(y),  \quad x \in \R^3; \nonumber
\eea
\bea \label{eq-SLP}
\mathcal{S}_{D}^{k} [\varphi]: H^{-\f{1}{2}}(\p D) &\longrightarrow & H^{\f{1}{2}}(\p D) \\
\varphi &\longmapsto & \mathcal{S}_{D}^{k} [\varphi](x) = \int_{\p D} G(x, y, k) \varphi(y) d\sigma(y),  \quad x \in \R^3; \nonumber
\eea
\bea \label{eq-NP}
\mathcal{K}_{D}^* [\varphi]: H^{-\f{1}{2}}(\p D) &\longrightarrow & H^{-\f{1}{2}}(\p D) \\
\varphi &\longmapsto & \mathcal{K}_{D}^* [\varphi](x) = \int_{\p D } \f{\p G(x, y, 0)}{ \p \nu(x)} \varphi(y) d\sigma(y) ,   \quad x \in \p D; \nonumber
\eea
\bea \label{eq-NPvect}
\mathcal{M}_{D}^{k} [\varphi]: H^{-\f{1}{2}}_T(\textnormal{div},\p D) &\longrightarrow & H^{-\f{1}{2}}_T(\textnormal{div},\p D) \\
\varphi &\longmapsto & \mathcal{M}_{D}^{k} [\varphi](x) = \int_{\p D} \nu(x)\times\nabla_x\times G(x, y, k)\varphi(y) d\sigma(y),  \quad x \in \p D; \nonumber
\eea
\bea \label{eq-L}
\mathcal{L}_{D}^{k} [\varphi]: H^{-\f{1}{2}}_T(\text{div},\p D) &\longrightarrow & H^{-\f{1}{2}}_T(\text{div},\p D) \\
\varphi &\longmapsto & \mathcal{L}_{D}^{k} [\varphi](x) = \nu(x)\times \bigg(k^2\vec{\mathcal{S}}_{D}^{k}[\varphi](x) + \nabla\mathcal{S}_D^k[\nabla_{\p D}\cdot\varphi](x)\bigg) ,   \quad x \in \p D. \nonumber
\eea

Throughout this paper, we denote  $\vec{\mathcal{S}}_{D}^{0},\mathcal{S}_{D}^{0},\mathcal{M}_{D}^{0}$ by $\vec{\mathcal{S}}_{D},\mathcal{S}_{D},\mathcal{M}_{D}$, respectively. We also denote $\mathcal{K}_{D}$ by the $(\cdot,\cdot)_{-\f{1}{2},\f{1}{2}}$-adjoint of $\mathcal{K}_{D}^*$, where $(\cdot,\cdot)_{-\f{1}{2},\f{1}{2}}$ is the  duality pairing between $H^{-\f{1}{2}}(\p D)$.

We recall now some useful results on the operator
 $\mathcal{K}_D^*$ \cite{book3, kang1, kang3, shapiro}. 
\begin{lem} \label{lem-Kstar_properties max}
\begin{enumerate}
\item[(i)] The following Calder\'on identity holds:
$\mathcal{K}_D \mathcal{S}_{D}= \mathcal{S}_{D}\mathcal{K}_D^*$;
\item[(ii)]
The operator $\mathcal{K}_D^*$ is compact self-adjoint in the Hilbert space $H^{-\f{1}{2}}(\p D)$ equipped with the following
inner product
\be \label{innerproduct}
(u, v)_{\mathcal{H}^*}= - (u, \mathcal{S}_{D}[v])_{-\f{1}{2},\f{1}{2}},
\ee
which is equivalent to $(\cdot, \cdot)_{-\f{1}{2}, \f{1}{2}}$;
\item[(iii)]
Let $(\lambda_j,\varphi_j) $, $j=0, 1, 2, \ldots$ be the eigenvalue and normalized eigenfunction pair of $\mathcal{K}_D^*$ in $\mathcal{H}^*(\p D)$. Then, $\lambda_j \in (-\f{1}{2}, \f{1}{2}]$, $\lambda_j \neq 1/2$ for $j\geq 1$, $\lambda_j \rightarrow 0$ as $j \rightarrow \infty$ and $\varphi_j\in\mathcal{H}^*_0 (\partial D)$ for $j \geq 1$, where $\mathcal{H}^*_0 (\partial D)$ is the zero mean subspace of $\mathcal{H}^* (\partial D)$;
\item[(iv)]
The following representation formula holds: for any $\psi \in H^{-1/2}(\p D)$,
$$
\mathcal{K}_D^* [\psi]
= \sum_{j=0}^{\infty} \lambda_j (\psi, \varphi_j)_{\mathcal{H}^*} \otimes \varphi_j;
$$
\item[(v)]
 The following trace formula holds: for any $\psi\in \mathcal{H}^*(\p D)$, 
$$
(\pm\f{1}{2}Id+\mathcal{K}_D^*)[\varphi] = \f{\p \mathcal{S}_D[\varphi]}{\p \nu}\Big\vert_{\pm}.
$$
\item[(vi)] Let $\mathcal{H}(\p D)$ be the space $H^{\f{1}{2}}(\p D)$ equipped with the following equivalent inner product
\be
(u, v)_{\mathcal{H}}= -(\mathcal{S}_{D}^{-1}[u], v)_{-\f{1}{2},\f{1}{2}}.
\ee
Then, 
$\mathcal{S}_{D}$ is an isometry between $\mathcal{H}^*(\p D)$ and $\mathcal{H}(\p D)$.
\end{enumerate}
\end{lem}

The following result holds.
\begin{lem} \label{lem-Helm_decomposition} The following Helmholtz decomposition holds \cite{buffa2002traces}:
\beas
H^{-\f{1}{2}}_T(\textnormal{div},\p D) = \nabla_{\p D}H^{\f{3}{2}}(\p D) \oplus \vec{\textnormal{curl}}_{\p D}H^{\f{1}{2}}(\p D).
\eeas
\end{lem}
\begin{rmk}
The Laplace-Beltrami operator $\Delta_{\p D}: H^{\f{3}{2}}_0(\p D) \rightarrow H^{-\f{1}{2}}_0(\p D)$ is invertible. Here $H^{\f{3}{2}}_0(\p D)$ and $H^{-\f{1}{2}}_0(\p D)$ are the zero mean subspaces of $H^{\f{3}{2}}(\p D)$ and $H^{-\f{1}{2}}(\p D)$ respectively.
\end{rmk} 

The following results on the operator $\mathcal{M}_D$ are of great importance.
\begin{lem}
$\mathcal{M}_D:H^{-\f{1}{2}}_T(\textnormal{div},\p D) \longrightarrow  H^{-\f{1}{2}}_T(\textnormal{div},\p D)$ is a compact operator.
\end{lem}

\begin{lem} \label{lem-PropertiesDecomp max} The following identities hold \cite{pierre, griesmaier2008asymptotic}:
\beas 
\mathcal{M}_{D}[\vec{\textnormal{curl}}_{\p D}\varphi] &=& \vec{\textnormal{curl}}_{\p D}\mathcal{K}_D[\varphi], \quad \forall \varphi \in H^{\f{1}{2}}(\p D), \\
\mathcal{M}_{D}[\nabla_{\p D}\varphi] &=& -\nabla_{\p D}\Delta^{-1}_{\p D}\mathcal{K}_D^{*}[\Delta_{\p D}\varphi] + \vec{\textnormal{curl}}_{\p D}\mathcal{R}_D[\varphi], \quad \forall \varphi \in H^{\f{3}{2}}(\p D),
\eeas
where $\mathcal{R}_D = -\Delta^{-1}_{\p D}\textnormal{curl}_{\p D}\mathcal{M}_D\nabla_{\p D}$.
\end{lem}

\section{Layer potential formulation for the scattering problem} \label{sec-Layer potential fomulation max}
We consider the scattering problem of a time-harmonic electromagnetic wave incident on a plasmonic nanoparticle. The homogeneous medium is characterized by electric permittivity $\varepsilon_m$ and magnetic permeability $\mu_m$, while the particle occupying a bounded and simply connected domain $D\Subset\mathbb{R}^3$ of class $\mathcal{C}^{1,\alpha}$ for $0<\alpha<1$ is characterized by electric permittivity $\varepsilon_c$ and magnetic permeability $\mu_c$, both of which depend on the frequency. Define
$$
k_m = \omega \sqrt{\eps_m \mu_m}, \quad k_c = \omega \sqrt{\eps_c \mu_c},
$$
and
$$
\eps_D= \eps_m \chi(\R^3 \backslash \bar{D}) + \eps_c \chi({D}), \quad
\mu_D= \eps_m \chi(\R^3 \backslash \bar{D})+ \eps_c \chi({D}),
$$
where $\chi$ denotes the characteristic function.

For a given incident plane wave $(E^i,H^i)$, solution to the Maxwell equations in free space
\beas
\nabla \times E^i &=& i\om\mu_m H^i \quad \text{in }\mathbb{R}^3,\\
\nabla \times H^i &=& -i\om\eps_m E^i \quad \text{in }\mathbb{R}^3,
\eeas
the scattering problem can be modeled by the following system of equations
\bea \label{eq-Maxwell} 
\nabla \times E &=& i\om\mu_D H \quad \text{in }\mathbb{R}^3\backslash \p D, \nonumber\\
\nabla \times H &=& -i\om\eps_D E \quad \text{in }\mathbb{R}^3\backslash \p D,\\ 
\nu\times E\big\vert_+ -\nu\times E\big\vert_- &=& \nu\times H\big\vert_+ -\nu\times H\big\vert_- = 0 \quad \text{on } \p D, \nonumber
\eea
subject to the Silver-M\"uller radiation condition:
\beas
\lim _{|x|\rightarrow \infty} |x| \big(\sqrt{\mu_m}(H-H^i)(x)\times\f{x}{|x|}-\sqrt{\eps_m}(E-E^i) (x)\big) = 0
\eeas
 uniformly in $x/|x|$. 
Here and throughout the paper, the subscripts $\pm$ indicate, as said before,  the limits from outside and inside $D$, respectively.

Using the boundary integral operators \eqref{eq-SLPvect} and \eqref{eq-NPvect}, the solution to \eqref{eq-Maxwell} can be represented as \cite{Torres}
\be \label{eq-Maxwell_Layer}
E(x) = \left\{ \begin{array}{ll}
E^i(x) + \mu_m\nabla\times\vec{\mathcal{S}}_{D}^{k_m}[\psi](x) + \nabla\times\nabla\times\vec{\mathcal{S}}_{D}^{k_m}[\phi](x) & \quad x\in \mathbb{R}^3\backslash\bar{D},\\
\mu_c\nabla\times\vec{\mathcal{S}}_{D}^{k_c}[\psi](x) + \nabla\times\nabla\times\vec{\mathcal{S}}_{D}^{k_c}[\phi](x) & \quad x\in D,
\end{array}
\right.
\ee
and
\be
H(x) = -\f{i}{\om\mu_D}(\nabla\times E)(x) \quad x\in \mathbb{R}^3\backslash \p D,
\ee
where the pair $(\psi,\phi)\in \big(H^{-\f{1}{2}}_T(\textnormal{div}, \p D)\big)^2$ is the unique solution to
\be \label{eq-Maxwell_System}
\left( \begin{array}{cc}
\df{\mu_c + \mu_m}{2}Id + \mu_c\mathcal{M}_D^{k_c}-\mu_m\mathcal{M}_D^{k_m} & \mathcal{L}_D^{k_c}-\mathcal{L}_D^{k_m}\\
\mathcal{L}_D^{k_c}-\mathcal{L}_D^{k_m} & \left(\df{k_c^2}{2\mu_c}+\df{k_m^2}{2\mu_m}\right)Id + \df{k_c^2}{\mu_c}\mathcal{M}_D^{k_c}-\df{k_m^2}{\mu_m}\mathcal{M}_D^{k_m}  
\end{array} \right)\left(\begin{array}{c}
\psi\\
\phi
\end{array}\right) = \left(\begin{array}{c}
\nu\times E^i\\
i\om\nu\times H^i\end{array}\right)\Bigg\vert_{\p D}.
\ee

Let $D = z + \delta B$ where $B$ contains the origin and $|B| = O(1)$. For any $x\in \p D$, let $\widetilde{x} = \f{x-z}{\delta}\in \p B$ and define for each function $f$ defined on $\p D$, a corresponding function defined on $B$ as follows
\be \label{defeta}
\eta(f)(\widetilde{x}) = f(z + \delta \widetilde{x}).
\ee
Throughout this paper, for two Banach spaces $X$ and $Y$, by  $\mathcal{L}(X,Y)$ we denote the set of bounded linear operators from $X$ into $Y$. We  will also denote by $\mathcal{L}(X)$ the set $\mathcal{L}(X,X)$.
\begin{lem} \label{lem-asymptotics M max} 
For $\varphi \in H^{-\f{1}{2}}_T(\textnormal{div}, \p D)$, the following asymptotic expansion holds
\beas
\mathcal{M}_D^{k}[\varphi](x) = \mathcal{M}_B[\eta(\varphi)](\tilde{x}) + \sum_{j=2}^{\infty}\delta^{j}\mathcal{M}_{B,j}^k[\eta(\varphi)](\tilde{x}),
\eeas
where
\beas
\mathcal{M}_{B,j}^k[\eta(\varphi)](\widetilde{x}) = \int_{\p B} \f{-(ik)^j}{4\pi j!} \nu(\tilde{x})\times\nabla_{\tilde{x}}\times |\tilde{x}-\tilde{y}|^{j-1}\eta(\varphi)(\tilde{y}) d\sigma(\tilde{y}).
\eeas
Moreover, $\|\mathcal{M}_{B,j}^k\|_{\mathcal{L}\big(H^{-\f{1}{2}}_T(\textnormal{div}, \p B)\big)}$ is uniformly bounded with respect to $j$. In particular, the convergence holds in $\mathcal{L}\big(H^{-\f{1}{2}}_T(\textnormal{div}, \p B)\big)$ and $\mathcal{M}_D^{ k}$ is analytical in $\delta$.
\end{lem}
\begin{proof}
We can see, after a change of variables, that
\beas
\mathcal{M}_D^{k}[\varphi](x) = \int_{\p B} \nu(\tilde{x})\times\nabla_{\tilde{x}}\times G(\tilde{x}, \tilde{y}, \delta k)\eta(\varphi)(\tilde{y}) d\sigma(\tilde{y}).
\eeas
A Taylor expansion of $G(\tilde{x},\tilde{y},\delta k)$ yields
\beas
G(\tilde{x},\tilde{y}, \delta k) = - \sum_{j =0}^{\infty} \f{(i\delta k|\tilde{x}-\tilde{y}|)^j}{j! 4 \pi |\tilde{x}-\tilde{y}|}
= - \f{1}{4 \pi |\tilde{x}-\tilde{y}|} + \sum_{j =1}^{\infty} \delta^j \f{(i k)^j}{4 \pi j!}|\tilde{x}-\tilde{y}|^{j-1},
\eeas
hence
\beas
\mathcal{M}_D^{k}[\varphi](x) = \mathcal{M}_B[\eta(\varphi)](\tilde{x}) + \sum_{j=2}^{\infty}\delta^{j}\int_{\p B} \f{-(ik)^j}{4\pi j!} \nu(\tilde{x})\times\nabla_{\tilde{x}}\times |\tilde{x}-\tilde{y}|^{j-1}\eta(\varphi)(\tilde{y}) d\sigma(\tilde{y}).
\eeas
where it is clear from the regularity of $|\tilde{x}-\tilde{y}|^{j-1}$, $j\geq2$, that $\|\mathcal{M}_{B,j}^k[\eta(\varphi)]\|_{H^{-\f{1}{2}}_T(\textnormal{div}, \p B)}$ is uniformly bounded with respect to $j$, therefore, 
$\|\mathcal{M}_{B,j}^k\|_{\mathcal{L}\big(H^{-\f{1}{2}}_T(\textnormal{div}, \p B)\big)}$ is uniformly bounded with respect to $j$ as well.
\end{proof}

\begin{lem} \label{lem-asymptotics L max}  
For $\varphi \in H^{-\f{1}{2}}_T(\textnormal{div}, \p D)$, the following asymptotic expansion holds
\beas
(\mathcal{L}_{D}^{k_c}-\mathcal{L}_{D}^{k_m}) [\varphi](x) = \sum_{j=1}^{\infty}\delta^{j}\om \mathcal{L}_{B,j}[\eta(\varphi)](\tilde{x}),
\eeas
where
\beas
\mathcal{L}_{B,j}[\eta(\varphi)](\tilde{x}) = C_j\nu(\tilde{x})\times \Big(\int_{\p B}|\tilde{x}-\tilde{y}|^{j-2}\eta(\varphi)(\tilde{y})d\sigma(\tilde{y}) - \int_{\p B}\f{|\tilde{x}-\tilde{y}|^{j-2}(\tilde{x}-\tilde{y})}{j+1}\nabla_{\p B}\cdot\eta(\varphi)(\tilde{y})d\sigma(\tilde{y}) \Big),
\eeas
and
\beas
C_j = \f{i^{j}(k_c^{j+1}-k_m^{j+1})}{\om 4\pi (j-1)!}.
\eeas
Moreover, $\|\mathcal{L}_{B,j}\|_{\mathcal{L}\big(H^{-\f{1}{2}}_T(\textnormal{div}, \p B)\big)}$ is uniformly bounded with respect to $j$. In particular, the convergence holds in $\mathcal{L}\big(H^{-\f{1}{2}}_T(\textnormal{div}, \p B)\big)$ and $\mathcal{L}_D^{ k}$ is analytical in $\delta$.
\end{lem}
\begin{proof}
The proof is similar to that of Lemma \ref{lem-asymptotics M max}.
\end{proof}

Using Lemma \ref{lem-asymptotics M max} and Lemma \ref{lem-asymptotics L max}, we can write the system of equations \eqref{eq-Maxwell_System} as follows:
\bea \label{eq-Maxwell_System_delta}
\mathcal{W}_B(\delta)\left(\begin{array}{c}
\eta(\psi)\\
\om\eta(\phi)
\end{array}\right) =
 \left(\begin{array}{c}
\df{\eta(\nu\times E^i)}{\mu_m-\mu_c}\\
\df{\eta(i\nu\times H^i)}{\eps_m-\eps_c}\end{array}\right)\Bigg\vert_{\p B},
\eea
where
\be \label{eq-W max}
\mathcal{W}_B(\delta)=\left( \begin{array}{cc}
\lambda_{\mu}Id - \mathcal{M}_B+\delta^2\df{\mu_m\mathcal{M}_{B,2}^{k_m}-\mu_c\mathcal{M}_{B,2}^{k_c}}{\mu_m-\mu_c} + O(\delta^3) & \df{1}{\mu_m-\mu_c}(\delta\mathcal{L}_{B,1} + \delta^2\mathcal{L}_{B,2}) + O(\delta^3)\\
\df{1}{\eps_m-\eps_c}(\delta\mathcal{L}_{B,1} + \delta^2\mathcal{L}_{B,2}) + O(\delta^3) & \lambda_{\eps}Id - \mathcal{M}_B+\delta^2\df{\eps_m\mathcal{M}_{B,2}^{k_m}-\eps_c\mathcal{M}_{B,2}^{k_c}}{\eps_m-\eps_c}  + O(\delta^3)
\end{array} \right),
\ee
and
\be \label{eq-lbda_eps,mu max}
\lambda_{\mu} = \df{\mu_c + \mu_m}{2(\mu_m - \mu_c)}, \quad \lambda_{\eps} = \df{\eps_c + \eps_m}{2(\eps_m - \eps_c)}.
\ee
It is clear that
\beas
\mathcal{W}_B(0) = \mathcal{W}_{B,0} =  \left( \begin{array}{cc}
\lambda_{\mu}Id - \mathcal{M}_B & 0 \\
0 & \lambda_{\eps}Id - \mathcal{M}_B
\end{array} \right).
\eeas
Moreover,
\beas
\mathcal{W}_B(\delta) = \mathcal{W}_{B,0} + \delta\mathcal{W}_{B,1} + \delta^2\mathcal{W}_{B,2} + O(\delta^3),
\eeas
in the sense that
\beas 
\|\mathcal{W}_B(\delta) - \mathcal{W}_{B,0} - \delta\mathcal{W}_{B,1} - \delta^2\mathcal{W}_{B,2} \|\leq C\delta^3,
\eeas
for a constant $C$ independent of $\delta$. Here $\|A\| = \sup_{i,j} \|A_{i,j}\|_{H^{-\f{1}{2}}_T(\textnormal{div},\p B)}$ for any operator-valued matrix $A$ with entries $A_{i,j}$.\\ 
We are now interested in finding $\mathcal{W}_B^{-1}(\delta)$. For this purpose, we first consider solving the problem
\be \label{eq-valpropM0 max}
\left(\lambda Id - \mathcal{M}_B\right)[\psi] = \varphi
\ee
for $(\psi,\varphi)\in \big(H^{-\f{1}{2}}_T(\textnormal{div},\p B)\big)^2$ and $\lambda \not\in \sigma(\mathcal{M}_B)$, where $\sigma(\mathcal{M}_B)$ is the spectrum of $\mathcal{M}_B$.\\
Using the Helmholtz decomposition of $H^{-\f{1}{2}}_T(\textnormal{div},\p B)$ in Lemma \ref{lem-Helm_decomposition},  we can reduce  \eqref{eq-valpropM0 max} to an equivalent system of equations involving some well known operators. 
\begin{definition} \label{def-HelmDecompNotation max}
For $u\in H^{-\f{1}{2}}_T(\textnormal{div},\p B)$, we denote by $u^{(1)}$ and $u^{(2)}$ any two functions in $H^{\f{3}{2}}_0(\p B)$ and $H^{\f{1}{2}}(\p B)$, respectively, such that
\beas
u = \nabla_{\p B}u^{(1)} +  \vec{\textnormal{curl}}_{\p B}u^{(2)}.
\eeas
\end{definition}
Note that $u^{(1)}$ is uniquely defined and $u^{(2)}$ is defined up to the sum of a constant function.
\begin{lem} \label{eq-equivalence0 max}
Assume $\lambda \neq \f{1}{2}$, then problem \eqref{eq-valpropM0 max} is equivalent to 
\be \label{eq-equivalence0 system max} 
(\lambda Id - \widetilde{\mathcal{M}}_B)\left(\begin{array}{c}
\psi^{(1)}\\
\psi^{(2)}
\end{array}\right) =
\left(\begin{array}{c}
\varphi^{(1)}\\
\varphi^{(2)}
\end{array}\right),
\ee
where $(\varphi^{(1)},\varphi^{(2)})\in H^{\f{3}{2}}_0(\p B)\times H^{\f{1}{2}}(\p B)$ and 
\beas
\widetilde{\mathcal{M}}_B = \left( \begin{array}{cc}
-\Delta_{\p B}^{-1}\mathcal{K}_B^*\Delta_{\p B} & 0\\
\mathcal{R}_B &  \mathcal{K}_B 
\end{array} \right).
\eeas
\end{lem}
\begin{proof}
Let $(\psi^{(1)},\psi^{(2)})\in H^{\f{3}{2}}_0(\p B)\times H^{\f{1}{2}}(\p B)$ be a solution (if there is any) to \eqref{eq-equivalence0 system max} where $(\varphi^{(1)},\varphi^{(2)})\in H^{\f{3}{2}}_0(\p B)\times H^{\f{1}{2}}(\p B)$ satisfies 
\beas
\varphi =  \nabla_{\p B}\varphi^{(1)} + \vec{\textnormal{curl}}_{\p B}\varphi^{(2)}.
\eeas
We have
\bea
\big(\lambda Id+\Delta^{-1}_{\p B}\mathcal{K}_B^{*}\Delta_{\p B}\big)[\psi^{(1)}] &=& \varphi^{(1)}  \label{eq-equivalenceM 1 max} \\
\lambda\psi^{(2)} - \mathcal{R}_B[\psi^{(1)}] - \mathcal{K}_B[\psi^{(2)}] &=& \varphi^{(2)}. \label{eq-equivalenceM 2 max}
\eea
Taking $\nabla_{\p B}$ in \eqref{eq-equivalenceM 1 max},  $\vec{\textnormal{curl}}_{\p B}$ in \eqref{eq-equivalenceM 2 max}, adding up and using the identities of Lemma \ref{lem-PropertiesDecomp max} yields
\beas
\left(\lambda Id - \mathcal{M}_B\right)[\nabla_{\p B}\psi^{(1)} + \vec{\textnormal{curl}}_{\p B}\psi^{(2)}] = \nabla_{\p B}\varphi^{(1)} + \vec{\textnormal{curl}}_{\p B}\varphi^{(2)}.
\eeas
Therefore
\beas
\psi = \nabla_{\p B}\psi^{(1)} + \vec{\textnormal{curl}}_{\p B}\psi^{(2)},
\eeas
is a solution of \eqref{eq-valpropM0 max}.\\
Conversely, let $\psi$ be the solution to \eqref{eq-valpropM0 max}. There exist  $(\psi^{(1)},\psi^{(2)})\in H^{\f{3}{2}}_0(\p B)\times H^{\f{1}{2}}(\p B)$ and  $(\varphi^{(1)},\varphi^{(2)})\in H^{\f{3}{2}}_0(\p B)\times H^{\f{1}{2}}(\p B)$ such that
\beas
\psi &=&  \nabla_{\p B}\psi^{(1)} + \vec{\textnormal{curl}}_{\p B}\psi^{(2)},\\
\varphi &=&  \nabla_{\p B}\varphi^{(1)} + \vec{\textnormal{curl}}_{\p B}\varphi^{(2)}.
\eeas
and we have
\be \label{eq-lambda-M helm max}
\left(\lambda Id - \mathcal{M}_B\right)[\nabla_{\p B}\psi^{(1)} + \vec{\textnormal{curl}}_{\p B}\psi^{(2)}] = \nabla_{\p B}\varphi^{(1)} + \vec{\textnormal{curl}}_{\p B}\varphi^{(2)}.
\ee
Taking $\nabla_{\p B}\cdot$ in the above equation and using the identities of Lemma \ref{lem-PropertiesDecomp max} yields 
\beas
\Delta_{\p B}\big(\lambda Id+\Delta^{-1}_{\p B}\mathcal{K}_B^{*}\Delta_{\p B}\big)[\psi^{(1)}] = \Delta_{\p B}\varphi^{(1)}.
\eeas
Since $(\psi^{(1)},\varphi^{(1)})\in (H^{\f{3}{2}}_0(\p B))^2$ we get
\beas
\big(\lambda Id+\Delta^{-1}_{\p B}\mathcal{K}_B^{*}\Delta_{\p B}\big)[\psi^{(1)}] = \varphi^{(1)}.
\eeas
Taking $\textnormal{curl}_{\p B}$ in \eqref{eq-lambda-M helm max} and using the identities of Lemma \ref{lem-PropertiesDecomp max} yields
\beas
\Delta_{\p B}(\lambda\psi^{(2)} - \mathcal{R}_B[\psi^{(1)}] - \mathcal{K}_B[\psi^{(2)}]) = \Delta_{\p B}\varphi^{(2)}.
\eeas
Therefore there exists a constant $c$ such that
\beas
\lambda\psi^{(2)} - \mathcal{R}_B[\psi^{(1)}] - \mathcal{K}_B[\psi^{(2)}] = \varphi^{(2)}+c\chi(\p B).
\eeas
Since $\mathcal{K}_{B}(\chi(\p B))=\df{1}{2}\chi(\p B)$ we have
\beas
\lambda\Big(\psi^{(2)}-\f{c}{\lambda-1/2}\Big) - \mathcal{R}_B[\psi^{(1)}] - \mathcal{K}_B\Big[\psi^{(2)}-\f{c}{\lambda-1/2}\Big] = \varphi^{(2)}.
\eeas
Hence, $\Big(\psi^{(1)},\psi^{(2)}-\df{c}{\lambda-1/2}\Big)\in H^{\f{3}{2}}_0(\p B)\times H^{\f{1}{2}}(\p B)$ is a solution to \eqref{eq-equivalence0 system max}
\end{proof}

Let us now analyze the spectral properties of $\widetilde{\mathcal{M}}_B$ in 
\begin{equation} \label{hb}
H(\p B) := H^{\f{3}{2}}_0(\p B)\times H^{\f{1}{2}}(\p B)
\end{equation} equipped with the inner product
\beas
(u,v)_{H(\p B)} = (\Delta_{\p B}u^{(1)},\Delta_{\p B}v^{(1)})_{\mathcal{H}^*} + (u^{(2)},v^{(2)})_{\mathcal{H}},  
\eeas
which is equivalent to $H^{\f{3}{2}}_0(\p B)\times H^{\f{1}{2}}(\p B)$.\\ 
By abuse of notation we call $u^{(1)}$ and $u^{(2)}$ the first and second components of any $u\in H(\p B)$.\\
We will assume for simplicity the following condition.
\begin{cond} \label{cond-eigKstar max}
The eigenvalues of $\mathcal{K}^*_B$ are simple.
\end{cond} 
Recall that $\mathcal{K}^*_B$ and $\mathcal{K}_B$ are compact and self-adjoint in $\mathcal{H}^*(\p B)$ and $\mathcal{H}(\p B)$, respectively. Since $\mathcal{K}_B$ is the $(\cdot,\cdot)_{-\f{1}{2},\f{1}{2}}$ adjoint of $\mathcal{K}^*_B$, we have $\sigma(\mathcal{K}_B) = \sigma(\mathcal{K}^*_B)$, where $\sigma(\mathcal{K}_B)$ (resp. $\sigma(\mathcal{K}^*_B)$) is the (discrete) spectrum of $\mathcal{K}_B$ (resp. $\mathcal{K}^*_B$).\\
Define
\bea \label{def-partition spectrum max}
\sigma_1 &=& \sigma(-\mathcal{K}^*_B)\backslash \Big(\sigma(\mathcal{K}_B)\cup\{-\f{1}{2}\}\Big), \nonumber \\
\sigma_2 &=& \sigma(\mathcal{K}_B)\backslash \sigma(-\mathcal{K}^*_B),  \\
\sigma_3 &=& \sigma(-\mathcal{K}^*_B) \cap \sigma(\mathcal{K}_B). \nonumber
\eea
Let $\lambda_{j,1}\in \sigma_1$, $j=1,2\dots$ and let $\varphi_{j,1}$ be an associated  normalized eigenfunction of $\mathcal{K}^*_B$ as defined in Lemma \ref{lem-Kstar_properties max}. Note that $\varphi_{j,1}\in H^{-\f{1}{2}}_0(\p B)$ for $j\geq 1$. Then,
\beas
\psi_{j,1} = 
\left(\begin{array}{c}
\Delta_{\p B}^{-1}\varphi_{j,1}\\
(\lambda_{j,1} Id - \mathcal{K}_B)^{-1}\mathcal{R}_B[\Delta_{\p B}^{-1}\varphi_{j,1}]
\end{array}\right)
\eeas
satisfies
\beas
 \widetilde{\mathcal{M}}_B[\psi_{j,1}] = \lambda_{j,1}\psi_{j,1}.  
\eeas
Let $\lambda_{j,2}\in \sigma_2$ and let $\varphi_{j,2}$ be an associated normalized eigenfunction of $\mathcal{K}_B$. Then,
\beas
\psi_{j,2} = 
\left(\begin{array}{c}
0\\
\varphi_{j,2}
\end{array}\right)
\eeas
satisfies
\beas
\widetilde{\mathcal{M}}_B[\psi_{j,2}] = \lambda_{j,2}\psi_{j,2}.
\eeas
Now, assume that Condition \ref{cond-eigKstar max} holds. Let $\lambda_{j,3}\in \sigma_3$, let $\varphi_{j,3}^{(1)}$ be the associated normalized eigenfunction of $\mathcal{K}^*_B$ and let $\varphi_{j,3}^{(2)}$ be the associated normalized eigenfunction of $\mathcal{K}_B$. Then,
\beas
\psi_{j,3} = 
\left(\begin{array}{c}
0\\
\varphi_{j,3}^{(2)}
\end{array}\right)
\eeas
satisfies
\beas
\widetilde{\mathcal{M}}_B[\psi_{j,3}] = \lambda_{j,3}\psi_{j,3},
\eeas
and $\lambda_{j,3}$ has a first-order generalized eigenfunction given by
\be \label{eq-constant c max}
\psi_{j,3,g} = \left(\begin{array}{c}
c\Delta_{\p B}^{-1}\varphi_{j,3}^{(1)}\\
(\lambda_{j,3} Id - \mathcal{K}_B)^{-1}\mathcal{P}_{\textnormal{span}\{\varphi_{j,3}^{(2)}\}^{\bot}}\mathcal{R}_B[c\Delta_{\p B}^{-1}\varphi_{j,3}^{(1)}]
\end{array}\right)
\ee
for a constant $c$ such that $\mathcal{P}_{\textnormal{span}\{\varphi_{j,3}^{(2)}\}}\mathcal{R}_B[c\Delta_{\p B}^{-1}\varphi_{j,3}^{(1)}] = -\varphi_{j,3}^{(2)}$. Here,  $\textnormal{span}\{\varphi_{j,3}^{(2)}\}$ is the vector space spanned by $\varphi_{j,3}^{(2)}$, 
$\textnormal{span}\{\varphi_{j,3}^{(2)}\}^{\bot}$ is the orthogonal space to $\textnormal{span}\{\varphi_{j,3}^{(2)}\}$ in $\mathcal{H}(\p B)$ (Lemma \ref{lem-Kstar_properties max} ), and 
$\mathcal{P}_{\textnormal{span}\{\varphi_{j,3}^{(2)}\}}$ 
(resp. $\mathcal{P}_{\textnormal{span}\{\varphi_{j,3}^{(2)}\}^{\bot} }$ 
is the orthogonal (in $\mathcal{H}(\p B)$) projection on $\textnormal{span}\{\varphi_{j,3}^{(2)}\}$  (resp. $\textnormal{span}\{\varphi_{j,3}^{(2)}\}^{\bot}$).

We remark that the function $\psi_{j,3,g}$ is determined by the following equation
\beas
\widetilde{\mathcal{M}}_B[\psi_{j,3,g}] = \lambda_{j,3}\psi_{j,3,g} + \psi_{j,3}.
\eeas
Consequently,  the following result holds. 
\begin{prop} \label{lem-spectMtilde max}
The spectrum $\sigma(\widetilde{\mathcal{M}}_B) = \sigma_1 \cup \sigma_2 \cup \sigma_3 = \sigma(-\mathcal{K}^*_B)\cup\sigma(\mathcal{K}^*_B)\backslash \{-\df{1}{2}\}$ in $H(\p B)$. Moreover, under Condition \ref{cond-eigKstar max}, $\widetilde{\mathcal{M}}_B$ has eigenfunctions $\psi_{j,i}$ associated to the eigenvalues $\lambda_{j,i}\in \sigma_i$ for $j = 1,2,\ldots$ and $ i=1,2,3,$ and generalized eigenfunctions of order one $\psi_{j,3,g}$ associated to $\lambda_{j,3}\in \sigma_3$, all of which form a non-orthogonal basis of ${H}(\p B)$ (defined by (\ref{hb})).
\end{prop}
\begin{proof}
It is clear that $\lambda-\widetilde{\mathcal{M}}_B$ is bijective if and only if $\lambda \notin \sigma(-\mathcal{K}^*_B)\cup\sigma(\mathcal{K}^*_B) \setminus \{ - \frac{1}{2}\}$.\\
It is only left to show that $\psi_{j,1},\psi_{j,2},\psi_{j,3},\psi_{j,3,g}$, $j = 1,2,\ldots$ form a non-orthogonal basis of $H(\p B)$.\\
Indeed, let
\beas
\psi = \left(\begin{array}{c}
\psi^{(1)}\\
\psi^{(2)}
\end{array}\right)\in H(\p B).
\eeas
Since $\psi_{j,1}^{(1)} \cup \psi_{j,3,g}^{(1)}$, $j = 1,2,\ldots$ form an orthogonal basis of $\mathcal{H}_0^*(\p B)$, which is equivalent to $H^{-\f{1}{2}}_0(\p B)$, there exist $\alpha_{\kappa}, \kappa \in I_1 := \{(j,1)\cup (j,3,g): j=1,2,\dots \}$ such that
\beas
\psi^{(1)} = \sum_{\kappa \in I_1}\alpha_{\kappa}\Delta_{\p B}^{-1}\psi_{\kappa}^{(1)},
\eeas
and
\beas
\sum_{\kappa\in I_1}|\alpha_{\kappa}|^2\leq \infty.
\eeas
It is clear that $\|\psi_{\kappa}^{(2)}\|_{\mathcal{L}( H^{\f{1}{2}}(\p B))}$ is uniformly bounded with respect to $\kappa \in I_1$. Then
\beas
h:=\sum_{\kappa \in I_1}\alpha_{\kappa}\psi_{\kappa}^{(2)} \in H^{\f{1}{2}}(\p B).
\eeas
Since $\psi_{j,2}^{(2)} \cup \psi_{j,3}^{(2)}$, $j = 1,2,\ldots$ form an orthogonal basis of $\mathcal{H}(\p B)$, which is equivalent to $H^{\f{1}{2}}(\p B)$, there exist $\alpha_{\kappa}, \kappa \in I_2 := \{(j,2)\cup (j,3): j=1,2,\dots \}$ such that
\beas
\psi^{(2)}-h = \sum_{\kappa \in I_2}\alpha_{\kappa}\psi_{\kappa}^{(2)},
\eeas
and
\beas
\sum_{\kappa\in I_2}|\alpha_{\kappa}|^2\leq \infty.
\eeas
Hence, there exist $\alpha_{\kappa}, \kappa \in I_1 \cup I_2$ such that
\beas
\psi = \sum_{\kappa\in I_1 \cup I_2}\alpha_{\kappa}\psi_{\kappa},
\eeas
and
\beas
\sum_{\kappa\in I_1 \cup I_2}|\alpha_{\kappa}|^2\leq \infty.
\eeas
\end{proof}
To have the compactness of $\widetilde{\mathcal{M}}_B$ we need the following condition
\begin{cond} \label{cond-Mbar compact}
$\sigma_3$ is finite.
\end{cond}
Indeed, if $\sigma_3$ is not finite we have $\widetilde{\mathcal{M}}_B(\{\psi_{j,3,g};\ \j\geq1\})=\{\lambda_{j,3}\psi_{j,g,3}+\psi_{j,3};\ \ j\geq1\}$ whose adherence is not compact. However, if $\sigma_3$ is finite, using Proposition \ref{lem-spectMtilde max} we can approximate $\widetilde{\mathcal{M}}_B$ by a sequence of finite-rank operators.
\\
Throughout this paper, we assume that Condition \ref{cond-Mbar compact} holds, even though an analysis can still be done for the case where $\sigma_3$ is infinite; see section \ref{sectshell}. 

\begin{definition} \label{def-proj_basis}
Let $\mathcal{B}$ be the basis of $H(\p B)$ formed by the eigenfunctions and generalized eigenfunctions of $\widetilde{\mathcal{M}}_B$ as stated in Lemma \ref{lem-spectMtilde max}. For $\psi \in H(\p B)$, we denote by $\alpha(\psi,\psi_{\kappa})$ the projection of $\psi$ into $\psi_{\kappa}\in\mathcal{B}$ such that
\beas
\psi = \sum_{\kappa}\alpha(\psi,\psi_{\kappa})\psi_{\kappa}.
\eeas
\end{definition}
The following lemma follows from the Fredholm alternative.
\begin{lem} \label{lem-alpha max} 
Let
\beas
\psi = \left(\begin{array}{c}
\psi^{(1)}\\
\psi^{(2)}
\end{array}\right)\in H(\p B).
\eeas
Then,
\beas
\alpha(\psi,\psi_{\kappa}) = \left\{\begin{array}{cc}
\df{(\psi,\widetilde{\psi}_{\kappa})_{H(\p B)}}{(\psi_{\kappa},\widetilde{\psi}_{\kappa})_{H(\p B)}} & \kappa = (j,i),\, i=1,2,\\
\df{(\psi,\widetilde{\psi}_{\kappa'})_{H(\p B)}}{(\psi_{\kappa},\widetilde{\psi}_{\kappa'})_{H(\p B)}} & \kappa = (j,3,g),\kappa' = (j,3),\\
\df{(\psi,\widetilde{\psi}_{\kappa_g})_{H(\p B)}-\alpha(\psi,\psi_{\kappa_g})(\psi_{\kappa_g},\widetilde{\psi}_{\kappa_g})_{H(\p B)}}{(\psi_{\kappa},\widetilde{\psi}_{\kappa_g})_{H(\p B)}} & \kappa = (j,3),\kappa_g = (j,3,g),
\end{array}\right. 
\eeas
where $\widetilde{\psi}_{\kappa}\in \textnormal{Ker} (\bar{\lambda}_{\kappa}-\mathcal{M}^*_{B})$ for $\kappa = (j,i),\, i=1,2,3$; $\widetilde{\psi}_{\kappa}\in \textnormal{Ker} (\bar{\lambda}_{\kappa}-\mathcal{M}^*_{B})^2$ for $\kappa = (j,3,g)$ and $\mathcal{M}^*_{B}$ is the $H(\p B)$-adjoint of $\mathcal{M}_{B}$.
\end{lem}
The following remark is in order. 
\begin{rmk} \label{rmk-alpha max}
Note that, since $\varphi_{j,1}$ and $\varphi_{j,3}^{(1)}$ form an orthogonal basis of $\mathcal{H}^*_0(\p B)$, equivalent to $H^{-\f{1}{2}}_0(\p B)$, we also have
\beas
\alpha(\psi,\psi_{\kappa}) = \left\{\begin{array}{cc}
(\Delta_{\p B}\psi^{(1)},\varphi_{j,1})_{\mathcal{H}^*} & \kappa = (j,1),\\
\f{1}{c}(\Delta_{\p B}\psi^{(1)},\varphi_{j,3}^{(1)})_{\mathcal{H}^*} & \kappa = (j,3,g),
\end{array}\right. 
\eeas
where $c$ is defined in \eqref{eq-constant c max}.
\end{rmk}
\begin{rmk} \label{rmk-inv(L-M) max} For $i = 1,2,3,$ and $j=1,2,\ldots,$ 
\beas
(\lambda Id - \widetilde{\mathcal{M}}_B)^{-1}[\psi_{j,i}] &=& \f{\psi_{j,i}}{\lambda - \lambda_{j,i}}, \\
(\lambda Id - \widetilde{\mathcal{M}}_B)^{-1}[\psi_{j,3,g}] &=& \f{\psi_{j,3,g}}{\lambda - \lambda_{j,3}} + \f{\psi_{j,3}}{(\lambda - \lambda_{j,3})^2}. 
\eeas
\end{rmk}

Now we turn to the original equation \eqref{eq-Maxwell_System}. The following result holds.
\begin{lem}
The system of equations \eqref{eq-Maxwell_System} is equivalent to
\bea \label{eq-Maxwell_System_delta2}
W_B(\delta)\left(\begin{array}{c}
\eta(\psi)^{(1)}\\
\eta(\psi)^{(2)}\\
\om\eta(\phi)^{(1)}\\
\om\eta(\phi)^{(2)}
\end{array}\right) =
\left(\begin{array}{c}
\df{\eta(\nu\times E^i)^{(1)}}{\mu_m-\mu_c}\\
\df{\eta(\nu\times E^i)^{(2)}}{\mu_m-\mu_c} \\
\df{\eta(i\nu\times H^i)^{(1)}}{\eps_m-\eps_c}\\
\df{\eta(i\nu\times H^i)^{(2)}}{\eps_m-\eps_c}
\end{array}\right)\Bigg\vert_{\p B} , 
\eea
where
\beas
W_B(\delta) &=& W_{B,0} + \delta W_{B,1} + \delta^2 W_{B,2} + O(\delta^3)
\eeas
with
\beas \label{eq-W max2}
W_{B,0} &=& \left(\begin{array}{cc}
\lambda_{\mu} Id - \widetilde{\mathcal{M}}_B & O\\
O & \lambda_{\eps} Id - \widetilde{\mathcal{M}}_B
\end{array}\right),\\
W_{B,1} &=& \left(\begin{array}{cc}
O & \df{1}{\mu_m-\mu_c}\widetilde{\mathcal{L}}_{B,1}\\
\df{1}{\eps_m-\eps_c}\widetilde{\mathcal{L}}_{B,1} & O
\end{array}\right),\\
W_{B,2} &=& \left(\begin{array}{cc}
\df{1}{\mu_m-\mu_c}\widetilde{\mathcal{M}}_{B,2}^{\mu} & \df{1}{\mu_m-\mu_c}\widetilde{\mathcal{L}}_{B,2}\\
\df{1}{\eps_m-\eps_c}\widetilde{\mathcal{L}}_{B,2} & \df{1}{\eps_m-\eps_c}\widetilde{\mathcal{M}}_{B,2}^{\eps}
\end{array}\right)
\eeas
and
\beas
\widetilde{\mathcal{M}}_B &=& \left( \begin{array}{cc}
-\Delta_{\p B}^{-1}\mathcal{K}_B^*\Delta_{\p B} & 0\\
\mathcal{R}_B &  \mathcal{K}_B 
\end{array} \right), \\
\widetilde{\mathcal{M}}_{B,2}^{\mu} &=& \left( \begin{array}{cc}
\Delta_{\p B}^{-1}\nabla_{\p B}\cdot(\mu_m\mathcal{M}_{B,2}^{k_m}-\mu_c\mathcal{M}_{B,2}^{k_c})\nabla_{\p B} & \Delta_{\p B}^{-1}\nabla_{\p B}\cdot(\mu_m\mathcal{M}_{B,2}^{k_m}-\mu_c\mathcal{M}_{B,2}^{k_c})\vec{\textnormal{curl}}_{\p B}\\
-\Delta_{\p B}^{-1}\textnormal{curl}_{\p B}(\mu_m\mathcal{M}_{B,2}^{k_m}-\mu_c\mathcal{M}_{B,2}^{k_c})\nabla_{\p B} & -\Delta_{\p B}^{-1}\textnormal{curl}_{\p B}(\mu_m\mathcal{M}_{B,2}^{k_m}-\mu_c\mathcal{M}_{B,2}^{k_c})\vec{\textnormal{curl}}_{\p B}
\end{array} \right),\\
\widetilde{\mathcal{M}}_{B,2}^{\eps} &=& \left( \begin{array}{cc}
\Delta_{\p B}^{-1}\nabla_{\p B}\cdot(\eps_m\mathcal{M}_{B,2}^{k_m}-\eps_c\mathcal{M}_{B,2}^{k_c})\nabla_{\p B} & \Delta_{\p B}^{-1}\nabla_{\p B}\cdot(\eps_m\mathcal{M}_{B,2}^{k_m}-\eps_c\mathcal{M}_{B,2}^{k_c})\vec{\textnormal{curl}}_{\p B}\\
-\Delta_{\p B}^{-1}\textnormal{curl}_{\p B}(\eps_m\mathcal{M}_{B,2}^{k_m}-\eps_c\mathcal{M}_{B,2}^{k_c})\nabla_{\p B} & -\Delta_{\p B}^{-1}\textnormal{curl}_{\p B}(\eps_m\mathcal{M}_{B,2}^{k_m}-\eps_c\mathcal{M}_{B,2}^{k_c})\vec{\textnormal{curl}}_{\p B}
\end{array} \right),\\
\widetilde{\mathcal{L}}_{B,s} &=& \left( \begin{array}{cc}
\Delta_{\p B}^{-1}\nabla_{\p B}\cdot\mathcal{L}_{B,s}\nabla_{\p B} & \Delta_{\p B}^{-1}\nabla_{\p B}\cdot\mathcal{L}_{B,s}\vec{\textnormal{curl}}_{\p B}\\
-\Delta_{\p B}^{-1}\textnormal{curl}_{\p B}\mathcal{L}_{B,s}\nabla_{\p B} & -\Delta_{\p B}^{-1}\textnormal{curl}_{\p B}\mathcal{L}_{B,s}\vec{\textnormal{curl}}_{\p B}\end{array} \right),
\eeas
for $s=1,2$.\\
Moreover,  the eigenfunctions of $W_{B,0}$ in $H(\p B)^2$ are given by
\beas
\Psi_{1,j,i} &=& \left(\begin{array}{c}
\psi_{j,i}\\
O
\end{array}\right)\quad j=0,1,2,\dots; i=1,2,3, \\
\Psi_{2,j,i} &=& \left(\begin{array}{c}
O\\
\psi_{j,i}
\end{array}\right)\quad j=0,1,2,\dots; i=1,2,3,
\eeas
associated to the eigenvalues $\lambda_{\mu}-\lambda_{j,i}$ and $\lambda_{\eps}-\lambda_{j,i}$, respectively,
and generalized eigenfunctions of order one
\beas
\Psi_{1,j,3,g} &=& \left(\begin{array}{c}
\psi_{j,3,g}\\
O
\end{array}\right), \\
\Psi_{2,j,3,g} &=& \left(\begin{array}{c}
O\\
\psi_{j,3,g}
\end{array}\right),
\eeas
associated to eigenvalues $\lambda_{\mu}-\lambda_{j,3}$ and $\lambda_{\eps}-\lambda_{j,3}$, respectively, all of which form a non-orthogonal basis of $H(\p B)^2$.
\end{lem}
\begin{proof} The proof follows directly from Lemmas \ref{eq-equivalence0 max} and \ref{lem-spectMtilde max}.
\end{proof}
We regard the operator $W_B(\delta)$ as a perturbation of the operator $W_{B,0}$ for small $\delta$. Using perturbation theory, we can derive the perturbed eigenvalues and their associated eigenfunctions in $H(\p B)^2$.\\
We denote by $\Gamma = \big\{(k,j,i): k=1,2 ; j=1,2,\dots ; i =1,2,3 \big\}$ the set of indices for the eigenfunctions of $W_{B,0}$ and  
by $\Gamma_g = \big\{(k,j,3,g) : k=1,2 ; j=1,2,\dots \big\}$ the set of indices for the generalized eigenfunctions. We denote by $\gamma_g$ the generalized eigenfunction index corresponding to eigenfunction index $\gamma$ and vice-versa. We also denote by
\be \label{eq-valpropW0 max}
\tau_{\gamma} =\left\{\begin{array}{cc}
\lambda_{\mu} - \lambda_{j,i} & k=1,\\
\lambda_{\eps} - \lambda_{j,i} & k=2.
\end{array}\right. 
\ee
\begin{cond} \label{cond-lambda-mu dist lambda-eps}
$\lambda_{\mu}\neq\lambda_{\eps}$.
\end{cond}
In the following we will only consider $\gamma \in \Gamma$ with which there is no generalized eigenfunction index associated. In other words, we only consider $\gamma = (k,i,j)\in \Gamma$ such that $\lambda_{j,i}\in \sigma_1 \cup \sigma_2$ (see \eqref{def-partition spectrum max} for the definitons). We call this subset $\Gamma_{\textnormal{sim}}$.\\
Note that Conditions \ref{cond-eigKstar max} and \ref{cond-lambda-mu dist lambda-eps}  imply that the eigenvalues of $W_{B,0}$ indexed by $\gamma \in \Gamma_{\textnormal{sim}}$ are simple.
As $\delta$ goes to zero, the perturbed eigenvalues and eigenfunctions indexed by $\gamma \in \Gamma_{\textnormal{sim}}$ have the following asymptotic expansions: 
\begin{eqnarray} \label{eq-perturbation max}
\tau_{\gamma}(\delta) &=& \tau_{\gamma} + \delta \tau_{\gamma,1} + \delta^2 \tau_{\gamma,2} + O(\delta^3),\\
\Psi_{\gamma}(\delta) &=& \Psi_{\gamma} + \delta \Psi_{\gamma,1} + O(\delta^2),\nonumber
\end{eqnarray}
where
\begin{eqnarray} 
\tau_{\gamma,1} &=& \f{(W_{B,1}\Psi_{\gamma},\widetilde{\Psi}_{\gamma})_{H(\p B)^2}}{(\Psi_{\gamma},\widetilde{\Psi}_{\gamma})_{H(\p B)^2}} = 0, \nonumber \\
\tau_{\gamma,2} &=& \f{(W_{B,2}\Psi_{\gamma},\widetilde{\Psi}_{\gamma})_{H(\p B)^2}-(W_{B,1}\Psi_{\gamma,1},\widetilde{\Psi}_{\gamma})_{H(\p B)^2}}{(\Psi_{\gamma},\widetilde{\Psi}_{\gamma})_{H(\p B)^2}}, \label{eq-perturbation_term max}\\
(\tau_{\gamma}-W_{B,0})\Psi_{\gamma,1} &=& -W_{B,1}\Psi_{\gamma}. \nonumber
\end{eqnarray}
Here, $\widetilde{\Psi}_{\gamma'}\in \textnormal{Ker} (\bar{\tau}_{\gamma'}-W^*_{B,0})$ and $W^*_{B,0}$ is the $H(\p B)^2$ adjoint of $W_{B,0}$.\\
Using Lemma \ref{lem-alpha max} and Remark \ref{rmk-inv(L-M) max} we can solve $\Psi_{\gamma,1}$. Indeed, 
\beas
\Psi_{\gamma,1} = \sum_{\gamma'\in \Gamma \atop \gamma'\neq\gamma}\f{\alpha(-W_{B,1}\Psi_{\gamma},\Psi_{\gamma'})\Psi_{\gamma'}}{\tau_{\gamma}-\tau_{\gamma'}} + \sum_{\gamma_g'\in \Gamma_g \atop \gamma'\neq\gamma}\alpha(-W_{B,1}\Psi_{\gamma},\Psi_{\gamma_g'})\left(\f{\Psi_{\gamma_g'}}{\tau_{\gamma}-\tau_{\gamma'}}+\f{\Psi_{\gamma'}}{(\tau_{\gamma}-\tau_{\gamma'})^2}\right) \\
+ \ \ \alpha(-W_{B,1}\Psi_{\gamma},\Psi_{\gamma})\Psi_{\gamma}.
\eeas
By abuse of notation, 
\bea \label{eq-alpha2}
\alpha(x,\Psi_{\gamma}) =  \left\{\begin{array}{cc}
\alpha(x_1, \psi_{\kappa}) & \gamma = (1,j,i),\,\kappa = (j,i),\\
\alpha(x_2, \psi_{\kappa}) & \gamma = (2,j,i),\,\kappa = (j,i),
\end{array}\right. 
\eea
for
\beas
x = \left(\begin{array}{cc}
x_1\\
x_2
\end{array}\right)\in H(\p B)^2,
\eeas
and $\alpha$ introduced in Definition \ref{def-proj_basis}.\\

Consider now the degenerate case $\gamma \in \Gamma \backslash \Gamma_{\textnormal{sim}} =: \Gamma_{\textnormal{deg}} = \{ \gamma = (k,i,j)\in \Gamma$ s.t $\lambda_{j,i}\in \sigma_3 \}$. It is clear that, for $\gamma\in \Gamma_{\textnormal{deg}}$, the algebraic multiplicity of the eigenvalue $\tau_{\gamma}$ is 2 while the geometric multiplicity is 1.\\
In this case every eigenvalue $\tau_{\gamma}$ and associeted eigenfunction $\Psi_{\gamma}$ will slipt into two branches, as $\delta$ goes to zero, represented by a convergent Puiseux series as \cite{triki2004}:
\begin{eqnarray} \label{eq-perturbation Puiseux max}
\tau_{\gamma,h}(\delta) &=& \tau_{\gamma} + (-1)^h\delta^{1/2}\tau_{\gamma,1} + (-1)^{2h}\delta^{2/2}\tau_{\gamma,2} + O(\delta^{3/2}), \quad h=0,1,\\
\Psi_{\gamma,h}(\delta) &=& \Psi_{\gamma} + (-1)^h\delta^{1/2}\Psi_{\gamma,1} + (-1)^{2h}\delta^{2/2}\Psi_{\gamma,2} + O(\delta^{3/2}), \quad h=0,1,\nonumber
\end{eqnarray}
where $\tau_{\gamma,j}$ and $\Psi_{\gamma,j}$ can be recovered by recurrence formulas. For simplicity we refer to \cite{kato} for more details.

\subsection{First-order correction to plasmonic resonances and field behavior at the plasmonic resonances}
 Recall that the electric and magnetic parameters, $\eps_c$ and $\mu_c$, depend on the frequency of the incident field, $\om$, following the Drude model \cite{pierre}. Therefore, the eigenvalues of the operator $W_{B,0}$ and perturbation in the eigenvalues depend on the frequency as well, that is
\beas
\tau_{\gamma}(\delta,\om) &=& \tau_{\gamma}(\om) + \delta^2 \tau_{\gamma,2}(\om) + O(\delta^3) \quad \gamma\in \Gamma_{\textnormal{sim}},\\
\tau_{\gamma,h}(\delta,\om) &=& \tau_{\gamma} + \delta^{1/2} (-1)^h\tau_{\gamma,1}(\om) + \delta^{2/2} (-1)^{2h}\tau_{\gamma,2}(\om) + O(\delta^{3/2}), \quad \gamma\in \Gamma_{\textnormal{deg}},\quad h=0,1.
\eeas
In the sequel, we will omit frequency dependence to simplify the notation. However, we will keep in mind that all these quantities are frequency dependent.\\
We first recall different notions of plasmonic resonance  \cite{matias}. 
\begin{definition} \label{def-plasmonicFreq}
\begin{itemize}
\item[(i)] We say that $\om$ is a plasmonic resonance if $|\tau_{\gamma}(\delta)| \ll 1$ and is locally minimized for some $ \gamma \in \Gamma_{\textnormal{sim}}$ or $|\tau_{\gamma,h}(\delta)| \ll 1$ and is locally minimized for some $\gamma\in \Gamma_{\textnormal{deg}}$, $h=0,1$.
\item[(ii)] We say that $\om$ is a quasi-static plasmonic resonance if
$|\tau_{\gamma}| \ll 1 $ and is locally minimized for some $\gamma \in \Gamma$. Here, $\tau_{\gamma}$ is 
defined by \eqref{eq-valpropW0 max}. 
\item[(iii)] We say that $\om$ is a first-order corrected quasi-static plasmonic resonance if $|\tau_{\gamma} + \delta^2 \tau_{\gamma,2}| \ll 1$ and is locally minimized for some $\gamma \in \Gamma_{\textnormal{sim}}$ or $|\tau_{\gamma} + \delta^{1/2} (-1)^h\tau_{\gamma,1}|\ll 1$ and is locally minimized for some $\gamma \in \Gamma_{\textnormal{deg}}$, $h=0,1$. Here, the correction terms $\tau_{\gamma,2}$ and $\tau_{\gamma,1}$ are 
defined by \eqref{eq-perturbation_term max} and \eqref{eq-perturbation Puiseux max}. 
\end{itemize}
\end{definition}

Note that quasi-static resonance is size independent and is therefore a zero-order approximation of the plasmonic resonance in terms of the particle size while the first-order corrected quasi-static plasmonic resonance depends on the size of the nanoparticle.

We are interested in solving equation \eqref{eq-Maxwell_System_delta2}
\beas 
W_B(\delta)\Psi = f,
\eeas
where
\beas
\Psi = \left(\begin{array}{c}
\eta(\psi)^{(1)}\\
\eta(\psi)^{(2)}\\
\om\eta(\phi)^{(1)}\\
\om\eta(\phi)^{(2)}
\end{array}\right),
f = \left(\begin{array}{c}
\df{\eta(\nu\times E^i)^{(1)}}{\mu_m-\mu_c}\\
\df{\eta(\nu\times E^i)^{(2)}}{\mu_m-\mu_c} \\
\df{\eta(i\nu\times H^i)^{(1)}}{\eps_m-\eps_c}\\
\df{\eta(i\nu\times H^i)^{(2)}}{\eps_m-\eps_c}
\end{array}\right)\Bigg\vert_{\p B}
\eeas
for $\om$ close to the resonance frequencies, i.e., when $\tau_{\gamma}(\delta)$ is very small for some $\gamma$'s $\in \Gamma_{\textnormal{sim}}$ or $\tau_{\gamma,h}(\delta)$ is very small for some $\gamma$'s $\in \Gamma_{\textnormal{deg}}$, $h=0,1$. In this case, the major part of the solution would be the contributions of the excited resonance modes $\Psi_{\gamma}(\delta)$ and $\Psi_{\gamma,h}(\delta)$.\\
It is important to remark that problem \eqref{eq-Maxwell_System} could be ill-posed if either $\Re(\eps_c)\leq 0$ or $\Re(\mu_c)\leq 0$ (the imaginary part being very small), and this are precisely the cases for which we will find the resonances described above. In fact, what we do is to solve the problem for the cases $\Re(\eps_c)>0$ or $\Re(\mu_c)>0$ and then, analytically continue the solution to the general case. The resonances are the values of $\om$ for which this analytical continuation "almost" cease to be valid.\\
We introduce the following definition.
\begin{definition} \label{def-j}
We call $J \subset \Gamma$ index set of resonances if $\tau_{\gamma}$'s are close to zero when $\gamma \in \Gamma$ and are bounded from below when $\gamma \in \Gamma^c$. More precisely, we choose a threshold number $\e_0 >0$ independent of $\om$ such that
$$
 | \tau_{\gamma} | \geq \e_0 >0 \quad \mbox{for }\, \gamma \in J^c.
$$
\end{definition}


From now on, we shall use $J$ as our index set of resonances.
For simplicity, we assume throughout this paper that the following condition holds.
\begin{cond} \label{condition1add max}
We assume that $\lambda_{\mu} \neq 0$, $\lambda_{\eps} \neq 0$ or equivalently, $\mu_c \neq - \mu_m$, $\eps_c \neq - \eps_m$. 
\end{cond}
It follows that the set $J$ is finite.\\
Consider the space $\mathcal{E}_J=\mbox{span}\{\Psi_{\gamma}(\delta),\Psi_{\gamma,h}(\delta);\; \gamma \in J,\,h=0,1\}$. Note that, under Condition \ref{condition1add max}, $\mathcal{E}_J$ is finite dimensional. Similarly, we define $\mathcal{E}_{J^c}$ as the spanned by $\Psi_{\gamma}(\delta),\Psi_{\gamma,h}(\delta);\; \gamma \in J^c,\,h=0,1$ and eventually other vectors to complete the base. We have $H(\p B)^2 = \mathcal{E}_J\oplus\mathcal{E}_{J^c}$.

We define $P_J (\delta)$ and $P_{J^c} (\delta)$ as the projection into the finite-dimensional space $\mathcal{E}_J$ and infinite-dimensional space $\mathcal{E}_{J^c}$, respectively. It is clear that, for any $f \in H(\p B)^2$
\beas
f = P_J(\delta)[f] + P_{J^c}(\delta)[f].
\eeas
Moreover, we have an explicit representation for $P_J(\delta)$
\be
P_J(\delta)[f] = \sum_{\gamma\in J\cap\Gamma_{\textnormal{sim}}}\alpha_{\delta}(f,\Psi_{\gamma}(\delta))\Psi_{\gamma}(\delta)+\sum_{\gamma\in J\cap\Gamma_{\textnormal{deg}} \atop h=0,1}\alpha_{\delta}(f,\Psi_{\gamma,h}(\delta))\Psi_{\gamma,h}(\delta).
\ee
Here, as in Lemma \ref{lem-alpha max}, 
\beas
\alpha_{\delta}(f,\Psi_{\gamma}(\delta)) &=& \f{(f,\tilde{\Psi}_{\gamma}(\delta))_{H(\p B)^2}}{(\Psi_{\gamma}(\delta),\tilde{\Psi}_{\gamma}(\delta))_{H(\p B)^2}}, \quad \gamma \in J\cap\Gamma_{\textnormal{sim}},\\
\alpha_{\delta}(f,\Psi_{\gamma,h}(\delta)) &=& \f{(f,\tilde{\Psi}_{\gamma,h}(\delta))_{H(\p B)^2}}{(\Psi_{\gamma,h}(\delta),\tilde{\Psi}_{\gamma,h}(\delta))_{H(\p B)^2}}, \quad \gamma \in J\cap\Gamma_{\textnormal{deg}},\, h=0,1.
\eeas
where $\widetilde{\Psi}_{\gamma}\in \textnormal{Ker} (\bar{\tau}_{\gamma,h}(\delta)-W^*_{B}(\delta))$, $\widetilde{\Psi}_{\gamma,h}\in \textnormal{Ker} (\bar{\tau}_{\gamma,h}(\delta)-W^*_{B}(\delta))$ and $W^*_{B}(\delta)$ is the $H(\p B)^2$-adjoint of $W_{B}(\delta)$.

We are now ready to solve the equation $W_B(\delta)\Psi = f$. In view of Remark \ref{rmk-inv(L-M) max},
\be
\Psi= {W}_{B}^{-1}(\delta)[f] = 
\sum_{\gamma\in J\cap\Gamma_{\textnormal{sim}}}\f{\alpha_{\delta}(f,\Psi_{\gamma}(\delta))\Psi_{\gamma}(\delta)}{\tau_{\gamma}(\delta)}+\sum_{\gamma\in J\cap\Gamma_{\textnormal{deg}} \atop h=0,1}\f{\alpha_{\delta}(f,\Psi_{\gamma,h}(\delta))\Psi_{\gamma,h}(\delta)}{\tau_{\gamma,h}(\delta)}
+ {W}_{B}^{-1}(\delta)P_{J^c}(\delta) [f].
\ee
The following lemma holds.
\begin{lem} \label{lem-residu} The norm 
 $\| {W}_{B}^{-1}(\delta)P_{J^c}(\delta)\|_{
 \mathcal{L}(H(\p B)^2, H(\p B)^2 )}$ is uniformly bounded in $\om$ and $\delta$.
\end{lem}
\begin{proof}
Consider the operator
$$
W_{B}(\delta)|_{J^c}: P_{J^c}(\delta)H(\p B)^2 \rightarrow  P_{J^c}(\delta)H(\p B)^2.
$$
We can show that for every $\om$ and $\delta$, $\mathrm{dist}(\sigma ( W_{B}(\delta)|_{J^c}), 0) \geq \f{\e_0}{2}$, where $\sigma (W_{B}(\delta)|_{J^c})$ is the discrete spectrum of $W_{B}(\delta)|_{J^c}$.
Here and throughout the paper, $\mathrm{dist}$ denotes the distance. 
Then, it follows that
$$
\| {W}_{B}^{-1}(\delta)P_{J^c}(\delta)[f] \| = \| W_{B}^{-1}(\delta)|_{J^c}  P_{J^c}(\delta) [f] \| \lesssim \f{1}{\e_0} \exp(\f{C_1}{\e_0^2}) \| P_{J^c}(\delta) [f]\| \lesssim  \f{1}{\e_0} \exp(\f{C_1}{\e_0^2}) \| f\|,
$$
where the notation $A \lesssim B$ means that $A \leq C B$ for some constant $C$ independent of $A$ and $B$.
\end{proof}

Finally, we are ready to state our main result in this section.

\begin{thm} \label{thm1}
Let $\eta$ be defined by (\ref{defeta}). Under Conditions \ref{cond-eigKstar max}, \ref{cond-Mbar compact}, \ref{cond-lambda-mu dist lambda-eps} and \ref{condition1add max},  the scattered field $E^s=E-E^i$ due to a single plasmonic particle has the following representation:
\beas
E^s = \mu_m\nabla\times\vec{\mathcal{S}}_{D}^{k_m}[\psi](x) + \nabla\times\nabla\times\vec{\mathcal{S}}_{D}^{k_m}[\phi](x) & \quad x\in \mathbb{R}^3\backslash\bar{D},
\eeas
where
\beas
\psi &=& \eta^{-1}\big(\nabla_{\p B}\widetilde{\psi}^{(1)} + \vec{\textnormal{curl}}_{\p B}\widetilde{\psi}^{(2)}\big) ,\\
\phi &=& \f{1}{\om}\eta^{-1}\big(\nabla_{\p B}\widetilde{\phi}^{(1)} + \vec{\textnormal{curl}}_{\p B}\widetilde{\phi}^{(2)}\big),
\eeas
\beas
\Psi = \left(\begin{array}{c}
\widetilde{\psi}^{(1)}\\
\widetilde{\psi}^{(2)}\\
\widetilde{\phi}^{(1)}\\
\widetilde{\phi}^{(2)}
\end{array}\right) =
\sum_{\gamma\in J\cap\Gamma_{\textnormal{sim}}}\f{\alpha(f,\Psi_{\gamma})\Psi_{\gamma}+O(\delta)}{\tau_{\gamma}(\delta)}+\sum_{\gamma\in J\cap\Gamma_{\textnormal{deg}}}\f{\zeta_1(f)\Psi_{\gamma} + \zeta_2(f)\Psi_{\gamma,1}+O(\delta^{1/2})}{\tau_{\gamma,0}(\delta)\tau_{\gamma,1}(\delta)}
+ O(1),
\eeas
and
\beas
\zeta_1(f) &=& \f{(f,\tilde{\Psi}_{\gamma,1})_{H(\p B)^2}\tau_{\gamma}-(f,\tilde{\Psi}_{\gamma})_{H(\p B)^2}(\tau_{\gamma,1}+\tau_{\gamma}\f{a_2}{a_1})}{a_1},\\
\zeta_2(f) &=& \f{(f,\tilde{\Psi}_{\gamma})_{H(\p B)^2}}{a_1},\\
a_1 &=& (\Psi_{\gamma},\tilde{\Psi}_{\gamma,1})_{H(\p B)^2} + (\Psi_{\gamma,1},\tilde{\Psi}_{\gamma})_{H(\p B)^2},\\
a_2 &=& (\Psi_{\gamma},\tilde{\Psi}_{\gamma,2})_{H(\p B)^2} + (\Psi_{\gamma,2},\tilde{\Psi}_{\gamma})_{H(\p B)^2} +  (\Psi_{\gamma,1},\tilde{\Psi}_{\gamma,1})_{H(\p B)^2}.
\eeas
\end{thm}
\begin{proof}
Recall that
\beas
\Psi =
\sum_{\gamma\in J\cap\Gamma_{\textnormal{sim}}}\f{\alpha_{\delta}(f,\Psi_{\gamma}(\delta))\Psi_{\gamma}(\delta)}{\tau_{\gamma}(\delta)}+\sum_{\gamma\in J\cap\Gamma_{\textnormal{deg}} \atop h=0,1}\f{\alpha_{\delta}(f,\Psi_{\gamma,h}(\delta))\Psi_{\gamma,h}(\delta)}{\tau_{\gamma,h}(\delta)}
+ {W}_{B}^{-1}(\delta)P_{J^c}(\delta) [f].
\eeas
By Lemma \ref{lem-residu}, we have ${W}_{B}^{-1}(\delta)P_{J^c}(\delta) [f]=O(1)$.\\
If $\gamma\in J\cap\Gamma_{\textnormal{sim}}$, an asymptotic expansion on $\delta$ yields
\beas
\alpha_{\delta}(f,\Psi_{\gamma}(\delta))\Psi_{\gamma}(\delta) = \alpha(f,\Psi_{\gamma})\Psi_{\gamma}+O(\delta).
\eeas
If $\gamma\in J\cap\Gamma_{\textnormal{deg}}$ then $(\Psi_{\gamma},\tilde{\Psi}_{\gamma})_{H(\p B)^2}=0$. Therefore, an asymptotic expansion on $\delta$ yields
\beas
\alpha_{\delta}(f,\Psi_{\gamma,h}(\delta))\Psi_{\gamma,h}(\delta) &=& \f{(-1)^{h}(f,\tilde{\Psi}_{\gamma})_{H(\p B)^2}\Psi_{\gamma}}{\delta^{-1/2}a_1} +\\ &&\f{1}{a_1}\left(\big((f,\tilde{\Psi}_{\gamma,1})_{H(\p B)^2}-(f,\tilde{\Psi}_{\gamma})_{H(\p B)^2}\df{a_2}{a_1}\big)\Psi_{\gamma}+(f,\tilde{\Psi}_{\gamma})_{H(\p B)^2}\Psi_{\gamma,1}\right) \\ \nm && +  O(\delta^{1/2})
\eeas
with
\beas
a_1 &=& (\Psi_{\gamma},\tilde{\Psi}_{\gamma,1})_{H(\p B)^2} + (\Psi_{\gamma,1},\tilde{\Psi}_{\gamma})_{H(\p B)^2},\\
a_2 &=& (\Psi_{\gamma},\tilde{\Psi}_{\gamma,2})_{H(\p B)^2} + (\Psi_{\gamma,2},\tilde{\Psi}_{\gamma})_{H(\p B)^2} +  (\Psi_{\gamma,1},\tilde{\Psi}_{\gamma,1})_{H(\p B)^2}.
\eeas
Since $\tau_{\gamma,h}(\delta)=\tau_{\gamma} + \delta^{1/2} (-1)^h\tau_{\gamma,1} + O(\delta)$, the result follows by adding the terms $$\df{\alpha_{\delta}(f,\Psi_{\gamma,0}(\delta))\Psi_{\gamma,0}(\delta)}{\tau_{\gamma,0}(\delta)} \quad \mbox{and } \df{\alpha_{\delta}(f,\Psi_{\gamma,1}(\delta))\Psi_{\gamma,1}(\delta)}{\tau_{\gamma,1}(\delta)}.$$
The proof is then complete. 
\end{proof}

\begin{cor}
Assume the same conditions as in Theorem \ref{thm1}. Under the additional condition that
\begin{equation}
\min_{\gamma\in J\cap\Gamma_{\textnormal{sim}}} |\tau_{\gamma}(\delta)| \gg \delta^3, \ \ \min_{\gamma\in J\cap\Gamma_{\textnormal{deg}}} |\tau_{\gamma}(\delta)| \gg \delta,
\end{equation}
we have
\beas
\Psi = \sum_{\gamma\in J\cap\Gamma_{\textnormal{sim}}}\f{\alpha(f,\Psi_{\gamma})\Psi_{\gamma}+O(\delta)}{\tau_{\gamma} + \delta^2 \tau_{\gamma,2}}+\sum_{\gamma\in J\cap\Gamma_{\textnormal{deg}}}\f{\zeta_1(f)\Psi_{\gamma} + \zeta_2(f)\Psi_{\gamma,1}+O(\delta^{1/2})}{\tau_{\gamma}^2 - \delta \tau_{\gamma,1}^2}
+ O(1).
\eeas
\end{cor}

\begin{cor} \label{cor thm1 order zero max}
Assume the same conditions as in Theorem \ref{thm1}. Under the additional condition that
\begin{equation}
\min_{\gamma\in J\cap\Gamma_{\textnormal{sim}}} |\tau_{\gamma}(\delta)| \gg \delta^2, \ \ \min_{\gamma\in J\cap\Gamma_{\textnormal{deg}}} |\tau_{\gamma}(\delta)| \gg \delta^{1/2},
\end{equation}
we have
\beas
\Psi = \sum_{\gamma\in J\cap\Gamma_{\textnormal{sim}}}\f{\alpha(f,\Psi_{\gamma})\Psi_{\gamma}+O(\delta)}{\tau_{\gamma}}+\sum_{\gamma\in J\cap\Gamma_{\textnormal{deg}}}\f{\alpha(f,\Psi_{\gamma})\Psi_{\gamma}}{\tau_{\gamma}} + \alpha(f,\Psi_{\gamma,g})\left(\f{\Psi_{\gamma,g}}{\tau_{\gamma}}+\f{\Psi_{\gamma}}{\tau_{\gamma}^2}\right)
+ O(1).
\eeas
\end{cor}
\begin{proof} We have
\beas
\lim_{\delta\rightarrow 0}W_B^{-1}(\delta)P_{\textnormal{span}{\{\Psi_{\gamma,0}(\delta),\Psi_{\gamma,1}(\delta)\}}}[f] &=& \lim_{\delta\rightarrow 0}\df{\alpha_{\delta}(f,\Psi_{\gamma,0}(\delta))\Psi_{\gamma,0}(\delta)}{\tau_{\gamma,0}(\delta)} +  \df{\alpha_{\delta}(f,\Psi_{\gamma,1}(\delta))\Psi_{\gamma,1}(\delta)}{\tau_{\gamma,1}(\delta)} \\
&=& W_{B,0}^{-1}(\delta)P_{\textnormal{span}{\{\Psi_{\gamma},\Psi_{\gamma_g}\}}}[f]\\
&=& \f{\alpha(f,\Psi_{\gamma})\Psi_{\gamma}}{\tau_{\gamma}} + \alpha(f,\Psi_{\gamma,g})\left(\f{\Psi_{\gamma,g}}{\tau_{\gamma}}+\f{\Psi_{\gamma}}{\tau_{\gamma}^2}\right),
\eeas
where $\gamma \in J\cap\Gamma_{\textnormal{deg}}$, $f\in H(\p B)^2$ and $)P_{\textnormal{span}{E}}$ is the projection into the linear space generated by the elements in the set  $E$. 
\end{proof}

\begin{rmk} \label{rmk-remark tau_y plasmonics max}
 Note that for $\gamma\in J$,
\beas
\tau_{\gamma} \approx  \min\left\{\textnormal{dist}\big(\lambda_{\mu},\sigma(\mathcal{K}^*_B)\cup-\sigma(\mathcal{K}^*_B)\big),\textnormal{dist}\big(\lambda_{\eps},\sigma(\mathcal{K}^*_B)\cup-\sigma(\mathcal{K}^*_B)\big)\right\}.
\eeas
\end{rmk}
It is clear, from Remark \ref{rmk-remark tau_y plasmonics max}, that resonances can occur when exciting the spectrum of $\mathcal{K}_B^*$ or/and that of $-\mathcal{K}_B^*$. We substantiate in the following that only the spectrum of $\mathcal{K}^*_B$ can be excited to create the plasmonic resonances in the quasi-static regime.\\
Recall that 
\beas
f = \left(\begin{array}{c}
\df{\eta(\nu\times E^i)^{(1)}}{\mu_m-\mu_c}\\
\df{\eta(\nu\times E^i)^{(2)}}{\mu_m-\mu_c} \\
\df{\eta(i\nu\times H^i)^{(1)}}{\eps_m-\eps_c}\\
\df{\eta(i\nu\times H^i)^{(2)}}{\eps_m-\eps_c}
\end{array}\right)\Bigg\vert_{\p B},
\eeas
and therefore, 
\beas
f_1 := \df{\eta(\nu\times E^i)^{(1)}}{\mu_m-\mu_c} =\df{\Delta_{\p B}^{-1}\nabla_{\p B}\cdot\eta(\nu\times E^i)}{\mu_m-\mu_c}.
\eeas
Now, suppose $\gamma=(1,j,1)\in J$ (recall that $J$ is the index set of resonances). Then $\tau_{\gamma}=\lambda_{\mu}-\lambda_{1,j}$, where $\lambda_{1,j}\in \sigma_1 = \sigma(-\mathcal{K}_B^*)\backslash\sigma(\mathcal{K}_B^*)$. From Remark \ref{rmk-alpha max}, 
\beas
\alpha(f,\Psi_{\gamma}) = (\Delta_{\p B}f_1,\varphi_{j,1})_{\mathcal{H}^*} = \alpha(f,\Psi_{\gamma}) = \f{1}{\mu_m-\mu_c}(\nabla_{\p B}\cdot\eta(\nu\times E^i),\varphi_{j,1})_{\mathcal{H}^*},
\eeas 
where $\varphi_{j,1}\in \mathcal{H}_0^*(\p B)$ is a normalized eigenfunction of $\mathcal{K}_B^*(\p B)$.

A Taylor expansion of $E^i$ gives, for $x\in \p D$, 
\beas
E^{i}(x) = \sum_{\beta\in \N^3}^{\infty}\f{ (x-z)^{\beta} \p ^{\beta} E^{i}(z)}{|\beta|!}.
\eeas
Thus
\beas
\eta(\nu\times E^i)(\tilde{x}) = \eta(\nu)(\tilde{x})\times E^{i}(z) + O(\delta),
\eeas
and
\beas
\nabla_{\p B}\cdot\eta(\nu\times E^i)(\tilde{x}) &=& -\eta(\nu)(\tilde{x}) \cdot \nabla \times E^i(z) + O(\delta)\\
&=& O(\delta).
\eeas
Therefore, the zeroth-order term of the expansion of $\nabla_{\p B}\cdot\eta(\nu\times E^i)$ in $\delta$ is zero. Hence,
\beas
\alpha(f,\Psi_{\gamma}) = 0.
\eeas 
In the same way, we have 
\beas
\alpha(f,\Psi_{\gamma}) &=& 0,\\
\alpha(f,\Psi_{\gamma_g}) &=& 0
\eeas
for $\gamma=(2,j,1)\in J$ and $\gamma_g$ such that $\gamma\in J$.

As a result we see that the spectrum of $-\mathcal{K}^*_{B}$ is not excited in the zeroth-order term. However, we note that $\sigma(-\mathcal{K}_B^*)$ can be excited in higher-order terms.

\section{The quasi-static limit and the extinction 
cross-section} \label{sec-quastatic max}
\subsection{The quasi-static limit}
In this subsection we recall the quasi-static limit of the electromagnetic field at plasmonic resonances. The formula was first obtained in \cite{pierre}, but it can be derived by pursing further computations in Corollary \ref{cor thm1 order zero max}.

We first recall the definition of the polarization tensor 
\begin{equation} \label{defm max}
 M(\lambda,D ) = \int_{\partial D}  (\lambda Id - \mathcal{K}_D^*)^{-1} [\nu](x) x\, d\sigma(x),
\end{equation}
where $\lambda\in \mathbb{C}\backslash(-1/2,1/2)$. The polarization tensor  is a key ingredient of the quasi-static limit, or zeroth-order approximation, of the far-field. \\
\begin{thm} \label{thm-pierre}
Let $d_{\sigma} = \min\left\{\textnormal{dist}\big(\lambda_{\mu},\sigma(\mathcal{K}^*_D)\cup-\sigma(\mathcal{K}^*_D)\big),\textnormal{dist}\big(\lambda_{\eps},\sigma(\mathcal{K}^*_D)\cup-\sigma(\mathcal{K}^*_D)\big)\right\}$. Then, for $D = z+\delta B\Subset\mathbb{R}^3$ of class $\mathcal{C}^{1,\alpha}$ for $0<\alpha<1$, the following uniform far-field expansion holds
\beas
E^s = -\f{i\om\mu_m}{\eps_m}\nabla\times G_d(x,z,k_m)M(\lambda_{\mu},D)H^i(z)-\om^2\mu_m G_d(x,z,k_m)M(\lambda_{\eps},D)E^i(z) + O(\f{\delta^4}{d_{\sigma}}),
\eeas
where
\beas
G_d(x,z,k_m) = \eps_m\big(G(x,z,k_m)Id + \f{1}{k_m^2}D_x^2G(x,z,k_m)\big)
\eeas
is the Dyadic Green (matrix valued) function for the full Maxwell equations.
\end{thm}

\subsection{The far-field expansion}
The following lemma deals with the far-field behavior of the electromagnetic fields. We first recall the representation for the 
scattering amplitude. 
\begin{lem} The solution $(E,H)$ to the system \eqref{eq-Maxwell} has the following far-field expansion:
\beas
E^s(x) = - \frac{e^{ik_m |x| }}{4\pi|x|}  
A_{\infty}(\hat{x}) + O\left(\frac{1}{\vert x \vert^2}\right)
\eeas
as $|x|\rightarrow +\infty$, where $\hat{x} = \f{x}{|x|}$,
\beas
A_{\infty}(\hat{x}) = -i\mu_mk_m \hat{x}\times \int_{\p D}e^{-ik_m\hat{x}\cdot y}\psi(y)d\sigma(y)-k_m^2\hat{x}\times\hat{x}\times \int_{\p D}e^{-ik_m\hat{x}\cdot y}\phi(y)d\sigma(y),\\
\eeas
and
\beas
H^s(x) = - \frac{e^{ik_m |x| }}{4\pi|x|} \hat{x}\times A_{\infty}(\hat{x}) + O\left(\frac{1}{|x|^2}\right).
\eeas
\end{lem}

The following result is known as the optical cross-section theorem for the scattering of electromagnetic waves \cite{born1999principles}.
\begin{thm} Assume that the incident fields are plane waves given by
\beas
E^{i}(x) &=& p e^{ik_m d\cdot x},\\
H^{i}(x) &=& d\times p e^{ik_m d\cdot x},
\eeas
where $p\in\R^3$ and $d\in \R^3$  with $|d|=1$ are such that $p\cdot d=0$. Then, the extinction  cross-section is given by
\beas
Q^{ext} = \f{4\pi}{k_m}\Im\left[\f{p\cdot A_{\infty}(d)}{|p|^2}\right],
\eeas
where $A_{\infty}$ is the scattering amplitude.  
\end{thm}

Doing Taylor expansions on the formula of Theorem \ref{thm-pierre} gives the following proposition, which allow us to compute the extinction cross-section in terms of the polarization tensor.
\begin{prop}  Let $\hat{x}= x/|x|$. The following far-field asymptotic expansion holds:
$$\begin{array}{lll}
E^s &=& \ds - \frac{e^{ik_m |x| }}{4\pi|x|} \left(\om\mu_m k_m e^{ik_m (d-\hat{x})\cdot z}\big(\hat{x}\times Id\big) M(\lambda_{\mu},D)(d\times p)-k_m^2 e^{ik_m (d-\hat{x})\cdot z}\big(Id-\hat{x}\hat{x}^t\big) M(\lambda_{\eps},D)p\right) \\ \nm && \ds + O(\f{1}{|x|^2}) +  O(\f{\delta^4}{d_{\sigma}}).
\end{array} $$
Then we have, up to an error term of the order $O(\f{\delta^4}{d_{\sigma}})$, 
\beas
A_{\infty}(\hat{x})=\om\mu_m k_m e^{ik_m (d-\hat{x})\cdot z}\big(\hat{x}\times Id\big) M(\lambda_{\mu},D)(d\times p)-k_m^2 e^{ik_m (d-\hat{x})\cdot z}\big(Id-\hat{x}\hat{x}^t\big) M(\lambda_{\eps},D)p.
\eeas
In particular, 
\beas
A_{\infty}(d)=\om\mu_m k_m \big(d\times Id\big) M(\lambda_{\mu},D)(d\times p)-k_m^2\big(Id-dd^t\big) M(\lambda_{\eps},D)p,
\eeas
where $M(\lambda_{\mu},D)$ and $M(\lambda_{\eps},D)$ are the polarization tensors associated with $D$ and $\lambda=\lambda_\mu$ and 
$\lambda=\lambda_\eps$, respectively.
\end{prop}

\section{Explicit computations for a spherical nanoparticle} \label{sec-explicit shpere max}
\subsection{Vector spherical harmonics}
Let $\hat{x}=\frac{x}{|x|}$.
For $m=-n,...,n$ and $n=1,2,...$, set $Y_n^m$ to be the spherical harmonics defined on the unit sphere $S= \{ x \in \R^3, |x|=1 \}$. For a wave number $k>0$, the function
$$
v_{n,m}(k;x) = h^{(1)}_n(k|x|)Y_n^m(\hat{x})
$$
satisfies the Helmholtz equation $\Delta v +k^2 v =0$ in $\mathbb{R}^3\setminus \{0\}$ together with the Sommerfeld radiation condition
$$
\lim_{|x|\rightarrow\infty}\Big(\frac{\p v_{n,m}}{\p |x|}(k;x)-ik v_{n,m}(k;x) \Big) =0.
$$
Similarly, let $\widetilde{v}_{n,m}(x)$ be defined by
$$
\widetilde{v}_{n,m}(x) = j_n(k |x|) Y_n^m(\hat{x}),
$$
where $j_n$ is the spherical Bessel function of the first kind. Then the function $\widetilde{v}_{n,m}$ satisfies the Helmholtz equation in $\mathbb{R}^3$.

Next, define the vector spherical harmonics by
$$
U_{n,m} = \frac{1}{\sqrt{n(n+1)}}\nabla_S Y_n^m(\hat{x}) \quad \mbox{and} \quad V_{n,m}=\hat{x}\times U_{n,m}
$$
for $m=-n,...,n$ and $n=1,2,...$. Here, $\hat{x}\in S$ and $\nabla_S$ denote the surface gradient on the unit sphere $S$. The vector spherical harmonics form a complete  orthogonal basis for $L_T^2(S)$.

Using the vectorial spherical harmonics, we can separate the solutions of Maxwell's equations into multipole solutions;  see  \cite[Section 5.3]{N}.
Define the  exterior transverse electric multipoles, i.e., $E\cdot x =0$, as
\be
 \ \left \{
 \begin{array}{l}
 \ds E_{n,m}^{TE}(x)= -\sqrt{n(n+1)}h_n^{(1)}(k|x|)V_{n,m}(\hat{x}),\\
 \ds H_{n,m}^{TE}(x) = -\frac{i}{\omega\mu}\nabla\times\Bigr(-\sqrt{n(n+1)}h_n^{(1)}(k|x|)V_{n,m}(\hat{x})\Bigr),
 \end{array}
 \right .
 \ee
and the exterior transverse  magnetic multipoles, i.e., $H\cdot x =0$,  as
\be
 \ \left \{
 \begin{array}{l}
\ds E_{n,m}^{TM}(x)= \frac{i}{\omega\epsilon}\nabla\times\Bigr(-\sqrt{n(n+1)}h_n^{(1)}(k|x|)V_{n,m}(\hat{x})\Bigr),\\[1em]
\ds H_{n,m}^{TM}(x) =-\sqrt{n(n+1)}h_n^{(1)}(k|x|)V_{n,m}(\hat{x}). \end{array}
 \right .
 \ee
The exterior electric and magnetic multipoles satisfy the Sommerfeld radiation condition. 
In the same manner, one defines the interior multipoles $(\widetilde{E}_{n,m}^{TE},\widetilde{H}_{n,m}^{TE})$ and  $(\widetilde{E}_{n,m}^{TM},\widetilde{H}_{n,m}^{TM})$ with $h_n^{(1)}$ replaced by $j_n$, {i.e.},
\be
 \ \left \{
 \begin{array}{l}
 \ds \widetilde E_{n,m}^{TE}(x)= -\sqrt{n(n+1)}j_n(k|x|)V_{n,m}(\hat{x}),\\
 \ds \widetilde H_{n,m}^{TE}(x) = -\frac{i}{\omega\mu}\nabla\times\widetilde E_{n,m}^{TE}(x),
 \end{array}
 \right .
 \ee
and
\be
 \ \left \{
 \begin{array}{l}
\ds \widetilde H_{n,m}^{TM}(x) =-\sqrt{n(n+1)}j_n(k|x|)V_{n,m}(\hat{x}),
\\
\ds \widetilde E_{n,m}^{TM}(x)= \frac{i}{\omega\epsilon}\nabla\times\widetilde H_{n,m}^{TM}(x).
\end{array}
 \right .
 \ee
 Note that one has
 \be\label{curl_TE_formula}
 \nabla\times E_{n,m}^{TE}(k;x)
 =\frac{\sqrt{n(n+1)}}{|x|}\mathcal{H}_n(k|x|) U_{n,m}(\hat{x})+\frac{n(n+1)}{|x|}h_n^{(1)}(k|x|)Y_{n}^m(\hat{x})\hat{x}
 \ee
 and
 \be\label{curl_TM_formula}
  \nabla\times \widetilde E_{n,m}^{TE}(k;x)
 =\frac{\sqrt{n(n+1)}}{|x|}\mathcal{J}_n(k|x|) U_{n,m}(\hat{x})+\frac{n(n+1)}{|x|}j_n(k|x|)Y_{n}^m(\hat{x})\hat{x},
   \ee
where
$$
\mathcal{J}_n(t)=j_n(t) + t j_n'(t),\qquad
\mathcal{H}_n(t)=h_n^{(1)}(t) + t (h_n^{(1)})'(t).
$$

For $|x|>|y|$, the following addition formula holds:
\begin{align}
G(x,y,k) I =&
 -\sum_{n=1}^{\infty}\frac{i{k}}{n(n+1)} \frac{\epsilon}{\mu} \sum_{m=-n}^n  E_{n,m}^{TM}(x)
 \overline{\widetilde E_{n,m}^{TM}(y)}^T 
 \nonumber
 \\
 &- \sum_{n=1}^{\infty}\frac{i{k}}{n(n+1)} \sum_{m=-n}^n  E_{n,m}^{TE}(x)
 \overline{\widetilde E_{n,m}^{TE}(y)}^T 
 \nonumber
 \\
&-\frac{i}{k}  \sum_{n=1}^{\infty} \sum_{m=-n}^n \nabla v_{n,m} (x) \overline{\nabla \widetilde{v}_{n,m} (y)}^T.
\label{Green_addition_formula}
\end{align}
Alternatively, for $|x|<|y|$, we have
\begin{align}
G(x,y,k) I =&
 -\sum_{n=1}^{\infty}\frac{i{k}}{n(n+1)} \frac{\epsilon}{\mu} \sum_{m=-n}^n   \overline{\widetilde E_{n,m}^{TM}(x)} E_{n,m}^{TM}(y)
^T \nonumber
 \\
 &- \sum_{n=1}^{\infty}\frac{i{k}}{n(n+1)} \sum_{m=-n}^n  
 \overline{\widetilde E_{n,m}^{TE}(x)} E_{n,m}^{TE}(y)^T
 \nonumber
 \\
&-\frac{i}{k}  \sum_{n=1}^{\infty} \sum_{m=-n}^n\overline{\nabla \widetilde{v}_{n,m} (x)}  \nabla v_{n,m} (y)^T.
\label{Green_addition_formula2}
\end{align}
\subsection{Explicit representations of boundary integral operators}
Let $D$ be a sphere of radius $r>0$.
We have the following results. 

\begin{lem} \label{lem_SLayerpotential_sphere_formula}
Let $\partial D=\{|x|=r\}$. Then, for $r'>r$, we have 
\begin{align}
\nu \times \nabla\times \vec{S}^k_D[U_{n,m}]\big|^+_{|x|=r'} &=(-ikr) h_n^{(1)}(k r') \mathcal{J}_n(kr)U_{n,m},
\label{curl_Single_sphere}
\\[0.5em]
\nu \times \nabla\times \vec{S}^k_D[V_{n,m}]\big|^+_{|x|=r'} &=ik \frac{r^2}{r'} j_n(k r) \mathcal{H}_n(kr')V_{n,m},
\\[0.5em]
\nu \times \nabla\times\nabla\times \vec{S}^k_D[U_{n,m}]\big|^+_{|x|=r'} &=-ik \frac{r}{r'}\mathcal{J}_n(kr) \mathcal{H}_n(kr')V_{n,m},
\\[0.5em]
\nu \times \nabla\times\nabla\times \vec{S}^k_D[V_{n,m}]\big|^+_{|x|=r'} &=ik (kr)^2 j_n(kr) h_n^{(1)}(kr') U_{n,m}.
\end{align}

For $r'<r$, 
\begin{align}
\nu \times \nabla\times \vec{S}^k_D[U_{n,m}]\big|^+_{|x|=r'} &=(-ikr) j_n(k r') \mathcal{H}_n(kr)U_{n,m},
\\[0.5em]
\nu \times \nabla\times \vec{S}^k_D[V_{n,m}]\big|^+_{|x|=r'} &=ik\frac{r^2}{r'}\mathcal{J}_n(kr')h_n^{(1)}(k r) V_{n,m},
\\[0.5em]
\nu \times \nabla\times\nabla\times \vec{S}^k_D[U_{n,m}]\big|^+_{|x|=r'} &=-ik\frac{r}{r'}\mathcal{J}_n(kr') \mathcal{H}_n(kr)V_{n,m},
\\[0.5em]
\nu \times \nabla\times\nabla\times \vec{S}^k_D[V_{n,m}]\big|^+_{|x|=r'} &=ik (kr)^2 j_n(kr') h_n^{(1)}(kr) U_{n,m}.
\end{align}

\end{lem}
\proof
We only consider \eqref{curl_Single_sphere}. The other formulas can be proved in a similar way.

From \eqref{curl_TE_formula}, \eqref{curl_TM_formula}, and the definitions of ${E}_{n,m}^{TE},{E}_{n,m}^{TM}, \widetilde{E}_{n,m}^{TE}$ and $\widetilde{E}_{n,m}^{TM}$, we have
\begin{align}
\nabla_x &\times G(x,y,k) U_{n,m}(\hat y) 
\nonumber
\\=&
 -\sum_{n=1}^{\infty}\frac{i{k}}{n(n+1)} \frac{\epsilon}{\mu} \sum_{m=-n}^n  \nabla \times E_{n,m}^{TM}(x)
 \overline{\widetilde E_{n,m}^{TM}(y)}\cdot U_{p,q}(\hat{y}) 
 \nonumber
 \\
 &+ \sum_{n=1}^{\infty}\frac{i{k}}{n(n+1)} \sum_{m=-n}^n  \nabla\times E_{n,m}^{TE}(x)
 \overline{\widetilde E_{n,m}^{TE}(y)}\cdot U_{p,q}(\hat{y})
\nonumber
\\=&
 -\sum_{n=1}^{\infty}\frac{i{k}}{\sqrt{n(n+1)}} \frac{\epsilon}{\mu} \sum_{m=-n}^n   \nabla \times E_{n,m}^{TM}(x)
 \frac{-i}{\omega\eps} \frac{1}{r} \mathcal{J}_n(k r) U_{n,m}(\hat{y})\cdot U_{p,q} (\hat{y})
 \nonumber
 \\
 &+ \sum_{n=1}^{\infty}\frac{i{k}}{\sqrt{n(n+1)}} \sum_{m=-n}^n  \nabla\times E_{n,m}^{TE}(x)
 (-1)j_n(kr)V_{n,m}(\hat{y})\cdot U_{p,q}(\hat{y})
 \nonumber
\end{align}
for $|y|=r$ and $|x|>|y|$.
Therefore, we get on $|x|=r$
\begin{align}
\nabla\times \vec{S}^k_D[U_{n,m}]\big|_+ &= 
 \nabla_x \times \int_{|y|=r} G(x,y,k) U_{n,m}(\hat y) 
 \nonumber
 \\&=\frac{{kr}}{\sqrt{n(n+1)}} \frac{1}{\omega\mu}  \mathcal{J}_n(k r) (\nabla \times E_{n,m}^{TM}(x))|_{|x|=r} .
    \end{align}
Since 
$$
\nabla\times E^{TM}_{p,q} = \frac{i}{\omega\eps} \nabla\times\nabla\times E_{p,q}^{TE}=
\frac{i}{\omega\eps} k^2 E_{p,q}^{TE},
$$
we obtain
\begin{align}
\hat{x}\times\nabla\times \vec{S}^k_D[U_{n,m}]\big|_+
&=
\frac{{ikr}}{\sqrt{n(n+1)}}   \mathcal{J}_n(k r) 
 (\hat{x}\times E_{n,m}^{TE}(x))|_{|x|=r} 
 \nonumber
 \\
 &=  (-ikr) h_n^{(1)}(k r) \mathcal{J}_n(kr)U_{n,m} \quad \mbox{on } 
 {|x|=r},
 \nonumber
\end{align}
which completes the proof. 
\qed

Note that 
$$
\nu \times \nabla\times \vec{\mathcal{S}}_{D}^{k}[\phi] \big|_{\pm}=(\mp\frac{1}{2}I+\mathcal{M}_D^k)[\phi] \quad \mbox{on } \partial D,  
$$
and recall the following identity, which was proved in \cite{Torres}, 
$$
\nu \times \nabla\times\nabla\times \vec{\mathcal{S}}_{D}^{k}[\phi] =\mathcal{L}_D^k[\phi] \quad \mbox{on } \partial D.  
$$
For $m=-n,\ldots,n$ and $n=1,2,3, \ldots$, let $H_{n,m}(\p D)$ be the subspace of $H(\p D)$ defined by
$$
H_{n,m}(\p D) = \mbox{span} \{U_{n,m}, V_{n,m}\}.
$$
Let us represent the operators $\mathcal{M}_D^k$ and $\mathcal{L}_D^k$ explicitly on the subspace $H_{m,n}(\p D)$.
Using $U_{n,m},V_{n,m}$ as basis vectors, we obtain the following matrix representations for  $\mathcal{M}_D^k$ and $\mathcal{L}_D^k$ on the subspace $H_{n,m}(\p D)$:
\be
\mathcal{M}_D^k=
\begin{pmatrix}
\ds\frac{1}{2}-ikr h_n^{(1)}(k r) \mathcal{J}_n(kr)
&& 0
\\
0 &&
\ds \frac{1}{2}+ikr j_n(k r) \mathcal{H}_n(kr)
\end{pmatrix},
\ee
and
\be
\mathcal{L}_D^k=
\begin{pmatrix}
0
&& 
ik (kr)^2 j_n(kr) h_n^{(1)}(kr) 
\\
-ik \mathcal{J}_n(kr) \mathcal{H}_n(kr) 
&&
0
\end{pmatrix}.
\ee

\subsection{Asymptotic behavior of the spectrum of $\mathcal{W}_{B}(r)$}
Now we consider the asymptotic expansions of the operator $\mathcal{W}_{B}(r)$ and its spectrum when $r \ll 1$.

It is well-known that, as $t\rightarrow 0$, 
\begin{align}
j_n(t) &= \frac{t^n}{(2n+1)!!}\Big(1-\frac{1}{2(2n+3)}t^2+O(t^4)\Big),
\nonumber
\\
h_n^{(1)}(t) &= -i((2n-1)!!)t^{-n-1}\Big(1+ \frac{1}{2(2n-1)}t^2+O(t^4)\Big).
\label{spherical_bessel_asymp}
\end{align}
By making use of these asymptotics of the spherical Bessel functions, we obtain that
\begin{align}
 i\mathcal{J}_n(t)h_n^{(1)}(\tilde{t}) &= \frac{n+1}{2n+1} \Big(\df{t}{\tilde{t}}\Big)^n\frac{1}{\tilde{t}} +
 \frac{n+1}{2(2n-1)(2n+1)}\Big(\df{t}{\tilde{t}}\Big)^n \tilde{t} - \frac{n+3}{2(2n+1)(2n+3)}\Big(\df{t}{\tilde{t}}\Big)^{n+1} t
 + O(t^3),
\nonumber
\\
ij_n(t)\mathcal{H}_n(\tilde{t})
&=\frac{-n}{2n+1} \Big(\df{t}{\tilde{t}}\Big)^n\frac{1}{\tilde{t}}
+\frac{-n+2}{2(2n-1)(2n+1)}\Big(\df{t}{\tilde{t}}\Big)^n \tilde{t}
 +\frac{n}{2(2n+1)(2n+3)}\Big(\df{t}{\tilde{t}}\Big)^{n+1} t  + O(t^3),
\nonumber
\\
ij_n(t)h_n^{(1)}(\tilde{t})
&=\frac{1}{2n+1} \Big(\df{t}{\tilde{t}}\Big)^n\frac{1}{\tilde{t}} 
+ \frac{1}{2(2n-1)(2n+1)}\Big(\df{t}{\tilde{t}}\Big)^n \tilde{t}
 -\frac{1}{2(2n+1)(2n+3)}\Big(\df{t}{\tilde{t}}\Big)^{n+1} t   + O(t^3),
\nonumber
\\
i\mathcal{J}_n(t)\mathcal{H}_n(\tilde{t})
&=\frac{(-n)(n+1)}{2n+1} \Big(\df{t}{\tilde{t}}\Big)^n\frac{1}{\tilde{t}} 
+ \frac{(n+1)(-n+2)}{2(2n-1)(2n+1)}\Big(\df{t}{\tilde{t}}\Big)^n \tilde{t}
+ \frac{n(n+3)}{2(2n+1)(2n+3)}\Big(\df{t}{\tilde{t}}\Big)^{n+1} t   + O(t^3),
\label{eqn_bessel_product_asymptotic}
\end{align}
for small $t,\tilde{t} \ll 1$ with $t\approx \tilde{t}$.

So, we have
\be
\mathcal{M}_D^k=
\begin{pmatrix}
\ds\frac{(-1)}{2(2n+1)} +  (kr)^2 r_n
& 0
\\
0 &
\ds\frac{1}{2(2n+1)}+  (kr)^2 s_n
\end{pmatrix}+O(r^4),
\ee
and
\be
\mathcal{L}_D^k=
\begin{pmatrix}
0
&& 
\ds k^2 r p_n
\\
\ds\frac{n(n+1)}{2n+1}\frac{1}{r} + k^2 r q_n
&&
0
\end{pmatrix}+O(r^3),
\ee
where
\begin{align}
 p_n &= \frac{1}{2n+1},
 \nonumber
\\
 q_n &= \frac{(n+1)(n-2)}{2(2n-1)(2n+1)} - \frac{n(n+3)}{2(2n+1)(2n+3)},
  \nonumber
\\
r_n &= -\frac{n+1}{2(2n-1)(2n+1)}+\frac{(n+3)}{2(2n+1)(2n+3)},
 \nonumber
\\
s_n &= -\frac{n-2}{2(2n-1)(2n+1)}+\frac{n}{2(2n+1)(2n+3)}.
 \label{eqn_pqrs_def}
\end{align}

Therefore, we can obtain 
$$
\mathcal{W}_B(r) = \mathcal{W}_{B,0} + r \mathcal{W}_{B,1} + r^2 \mathcal{W}_{B,2}+O(r^3),
$$
where
\begin{align}
\mathcal{W}_{B,0}&=
\begin{pmatrix}
\ds \lambda_\mu-\frac{(-1)}{2(2n+1)}
& 
0
&
0
&
0
\\
\ds 0 & \ds\lambda_\mu-\frac{1}{2(2n+1)} & 
0
 & 0
\\
\ds 0 &0 & \ds \lambda_{\eps}-\frac{(-1)}{2(2n+1)}& 0
\\
0 & 0 & 0 & \ds\lambda_\eps-\frac{1}{2(2n+1)} 
\end{pmatrix},
\end{align}
\begin{align}
\mathcal{W}_{B,1}&=
\begin{pmatrix}
0
& 
0
&
0
&
 \omega  C_{\mu} p_n
\\
\ds 0 & 0 & 
 \omega C_{\mu}q_n
 & 0
\\
 0 &  \omega C_{\eps} p_n & 0 & 0
\\
\omega C_{\eps}q_n & 0 & 0 & 0
\end{pmatrix},
\\
\mathcal{W}_{B,2}&=
\begin{pmatrix}
\omega^2 D_\mu r_n
& 
0
&
0
&
0
\\
\ds 0 & 
\omega^2 D_\mu s_n
 & 
0
 & 0
\\
\ds 0 & 0 & 
\omega^2 D_\eps r_n
 & 0
\\
0 & 0 & 0 & 
\omega^2 D_\eps s_n
\end{pmatrix},
\end{align}
and
\begin{align}
C_{\mu}&=\frac{\mu_c\eps_c-\mu_m\eps_m}{\mu_m-\mu_c},
\quad
C_{\eps}=\frac{\mu_c\eps_c-\mu_m\eps_m}{\eps_m-\eps_c},
\\
D_{\mu} &= \frac{\eps_c\mu_c^2-\eps_m \mu_m^2}{\mu_m-\mu_c}, \quad D_\eps=\frac{\eps_c^2\mu_c-\eps_m^2 \mu_m}{\eps_m-\eps_c}.
\end{align}
By applying the standard perturbation theory, the asymptotics of eigenvalues of $\mathcal{W}_{B}(r)$ are obtained as follows: up to  an error term of the order $O(r^3)$,
\begin{align*}
\ds \lambda_\mu-\frac{(-1)}{2(2n+1)} +(r\omega)^2 \bigg[ C_\eps C_\mu \frac{p_n q_n}{\lambda_\mu-\lambda_\eps+p_n} + D_\mu r_n\bigg] +O(r^3),
\\
\ds \lambda_\mu-\frac{1}{2(2n+1)} + (r\omega)^2 \bigg[  
C_\eps C_\mu \frac{p_n q_n}{\lambda_\mu-\lambda_\eps-p_n} + D_\mu s_n
  \bigg] +O(r^3),
\\
 \lambda_\eps-\frac{(-1)}{2(2n+1)} + (r\omega)^2 \bigg[  
C_\eps C_\mu \frac{p_n q_n}{\lambda_\eps-\lambda_\mu+p_n} + D_\eps r_n
  \bigg]+O(r^3),
\\
 \lambda_\eps-\frac{1}{2(2n+1)} + (r\omega)^2 \bigg[  C_\eps C_\mu \frac{ p_n q_n}{\lambda_\eps-\lambda_\mu-p_n} + D_\eps s_n
  \bigg]+O(r^3),
\end{align*}
and the asymptotics of the associated eigenfunction are given by
\begin{align*}
&[1,0,0,0]^T+r \omega  \frac{C_\eps q_n}{\lambda_\mu-\lambda_\eps+p_n} [0,0,0, 1]^T +O(r^2),
\\[0.5em]
&[0,1,0,0]^T+r \omega  \frac{C_\eps}{2n+1} \frac{1}{\lambda_\mu-\lambda_\eps-p_n} [0,0,1,0]^T +O(r^2),
\\[0.5em]
&[0,0,1,0]^T+r \omega  \frac{C_\mu q_n}{\lambda_\eps-\lambda_\mu+p_n}[0,1,0,0]^T +O(r^2),
\\[0.5em]
&[0,0,0,1]^T+r \omega  \frac{C_\mu}{2n+1} \frac{1}{\lambda_\eps-\lambda_\mu-p_n} [1,0,0,0]^T +O(r^2).
\end{align*}

\subsection{Extinction cross-section}
In this subsection, we compute the extinction cross-section $Q^{ext}$.
We need the following lemma.

\begin{lem} Let $D$ be a sphere with radius $r>0$ and  
suppose that $E^i$ is given by
$$
E^i(x)=
\sum_{n=1}^{\infty} \sum_{l=-n}^n \alpha^{TE}_{nl} \widetilde E_{n,l}^{TE}(x;k_m) + \alpha^{TM}_{nl} \widetilde E_{n,l}^{TM}(x;k_m),
$$
for some coefficients $\alpha^{TE}_{nl},\alpha^{TM}_{nl}$. Then the scattered wave can be represented as follows: for $|x|>r$,
$$
E^s(x) = \sum_{n=1}^{\infty} \sum_{l=-n}^n \alpha^{TE}_{nl} S_n^{TE} E_{n,l}^{TE}(x;k_m) + \alpha^{TM}_{nl} S_n^{TM} E_{n,l}^{TM}(x;k_m),
$$
where $S_n^{TE}$ and $S_n^{TM}$ are given by
\begin{align*}
S_n^{TE}&=\frac{\mu_c j_n(k_c r)\mathcal{J}_n(k_m r)-\mu_m j_n(k_m r) \mathcal{J}_n(k_c r)}{\mu_m \mathcal{J}_n(k_c r) h_n^{(1)}(k_m r) - \mu_c j_n(k_c r) \mathcal{H}(k_m r)},
\\[0.5em]
S_n^{TM}&=\frac{\eps_c j_n(k_c r)\mathcal{J}_n(k_m r)-\eps_m j_n(k_m r) \mathcal{J}_n(k_c r)}{\eps_m \mathcal{J}_n(k_c r) h_n^{(1)}(k_m r) - \eps_c j_n(k_c r) \mathcal{H}(k_m r)}.
\end{align*}

\end{lem}
\proof
Let $E^i = \widetilde{E}_{n,l}^{TE}(x;k_m)$.
We look for a solution of the following form:
$$
E=\begin{cases}
a \,\widetilde{E}_{n,l}^{TE}(x;k_c), &\quad |x|<r
\\
\widetilde{E}_{n,l}^{TE}(x;k_m) + b \,{E}_{n,l}^{TE}(x;k_m), &\quad |x|>r.
\end{cases}
$$
Then, from the boundary condition on $\p D$, we easily see that
\begin{align} 
\begin{pmatrix}
j_n(k_{m}r)\\  \frac{1}{\mu_{m}}  \mathcal{J}_n (k_{m} r)
\end{pmatrix} 
&=
\begin{pmatrix} j_n(k_c r )&-h_n^{(1)}(k_m r)\\\frac{1}{\mu_{c}}  \mathcal{J}_n (k_c r) &
-\frac{1} {\mu_{m}}\mathcal{H}_n (k_m r)\end{pmatrix}
\begin{pmatrix}
a\\ b
\end{pmatrix} .
\end{align}
Therefore, the coefficient $a$ and $b$ can be obtained as follows:
\begin{align} 
\begin{pmatrix}
1/a\\ b/a
\end{pmatrix}
&=
\begin{pmatrix} j_n(k_m r )&h_n^{(1)}(k_m r)\\\frac{1}{\mu_{m}}  \mathcal{J}_n (k_m r) &
\frac{1} {\mu_{m}}\mathcal{H}_n (k_m r)\end{pmatrix}^{-1}
\begin{pmatrix}
j_n(k_{c}r)&h_n^{(1)}(k_{c}r)\\
\frac{1}{\mu_{c}}  \mathcal{J}_n (k_{c} r) &
\frac{1} {\mu_{c}}\mathcal{H}_n (k_{c} r)
\end{pmatrix}
\begin{pmatrix}
1 \\ 0
\end{pmatrix}, \nonumber 
\\
&= \frac{ \mu_m k_m r}{i}\begin{pmatrix}\frac{1} {\mu_{m}}\mathcal{H}_n (k_m r) &-h_n^{(1)}(k_m r)\\-\frac{1}{\mu_{m}}  \mathcal{J}_n (k_m r) &
j_n(k_m r )
\end{pmatrix}
\begin{pmatrix}
j_n(k_{c}r)\\
\frac{1}{\mu_{c}}  \mathcal{J}_n (k_{c} r) 
\end{pmatrix}, \nonumber 
\\
&= -i k_m r 
\begin{pmatrix}
\ds \mathcal{H}_n(k_m r) j_n(k_{c}r)-\frac{\mu_m}{\mu_c}h_n^{(1)}(k_m r) \mathcal{J}_n(k_c r)
\\[1em]
\ds -  \mathcal{J}_n (k_{m} r) j_n(k_c r)+\frac{\mu_m}{\mu_c}j_n(k_m r) \mathcal{J}_n(k_c r)
\end{pmatrix},
\end{align}
where we have used the following Wronskian identity for the spherical Bessel function:
$$
j_n(t) \mathcal{H}_n(t) -h_n^{(1)}(t) \mathcal{J}_n(t) = t\Big(j_n(t)(h_n^{(1)})'(t)- j_n'(t)h_n^{(1)}(t)\Big)=\frac{i}{t}.
$$
Therefore, we immediately see that
$$
b=\frac{\mu_c j_n(k_c r)\mathcal{J}_n(k_m r)-\mu_m j_n(k_m r) \mathcal{J}_n(k_c r)}{\mu_m \mathcal{J}_n(k_c r) h_n^{(1)}(k_m r) - \mu_c j_n(k_c r) \mathcal{H}(k_m r)}.
$$

Now suppose that $E^i = \widetilde{E}_{n,l}^{TM}(x;k_m)$. We look for a solution in the following form:
$$
E=\begin{cases}
c \,\widetilde{E}_{n,l}^{TM}(x;k_c), &\quad |x|<r,
\\
\widetilde{E}_{n,l}^{TM}(x;k_m) + d \,{E}_{n,l}^{TM}(x;k_m), &\quad |x|>r.
\end{cases}
$$
Then, from the boundary conditions on $|x|=r$, we obtain
\begin{align} 
\label{thiseq}
&\begin{pmatrix}\frac{1}{\eps_{c}} \mathcal{J}_n (k_c r)&
\frac{1} {\eps_{c}}  \mathcal{H}_n (k_{c} r) 
\\j_n(k_{c}r)&h_n^{(1)}(k_{c}r)\end{pmatrix}
\begin{pmatrix}
c \\ 0
\end{pmatrix}
=
\begin{pmatrix}
\frac{1}{\eps_{m}} \mathcal{J}_n (k_{m} r)&
\frac{1} {\eps_{m}} \mathcal{H}_n (k_{m} r)\\\
j_n(k_{m}r)&h_n^{(1)}(k_{m}r)
\end{pmatrix}
\begin{pmatrix}
1 \\ d
\end{pmatrix}.
\end{align}
By solving (\ref{thiseq}), we get
$$
d=\frac{\eps_c j_n(k_c r)\mathcal{J}_n(k_m r)-\eps_m j_n(k_m r) \mathcal{J}_n(k_c r)}{\eps_m \mathcal{J}_n(k_c r) h_n^{(1)}(k_m r) - \eps_c j_n(k_c r) \mathcal{H}(k_m r)}.
$$
By the principle of superposition, the conclusion immediately follows.

\qed

We also need the following lemma concerning the scattering amplitude $A_\infty$.

\begin{lem}
Suppose that the scattered electric field $E^s$ is given by
$$
E^s(x)=
\sum_{n=1}^{\infty} \sum_{l=-n}^n \beta^{TE}_{nl}  E_{n,l}^{TE}(x;k_m) + \beta^{TM}_{nl}  E_{n,l}^{TM}(x;k_m)
$$
for $\mathbb{R}^3\setminus \overline{D}$.
Then the scattering amplitude $A_\infty$ can be represented as follows:
$$
A_\infty(\hat{x}) = 
\sum_{n=1}^{\infty} \sum_{l=-n}^n\frac{4\pi(-i)^{n}}{i k_m}\sqrt{n(n+1)} \left(\beta^{TE}_{nl}  V_{n,l}(\hat{x}) + \sqrt{\frac{\mu_m}{\eps_m}}\beta^{TM}_{nl}  U_{n,l}\right).
$$
\end{lem}
\proof
It is well-known that
$$
h_n^{(1)} (t) \sim \frac{1}{t} e^{it} e^{-i\frac{n+1}{2} \pi} \quad \mbox{as } t \rightarrow \infty,
$$
and
$$
(h_n^{(1)})'(t) \sim \frac{1}{t} e^{it} e^{-i\frac{n}{2} \pi}\quad \mbox{as } t \rightarrow \infty.
$$
Then one can easily see that as $|x| \rightarrow \infty$,
$$ 
E_{n,m}^{TE}(x;k_m) \sim -\frac{e^{i k_m |x|}}{k_m |x|} e^{-i\frac{n+1}{2}\pi}  \sqrt{n(n+1)}V_{n,l}(\hat{x}) 
$$
and
$$
E_{n,m}^{TM}(x;k_m) \sim -\frac{e^{i k_m |x|}}{k_m |x|} \sqrt{\frac{\mu_m}{\eps_m}}e^{-i\frac{n+1}{2} \pi} \sqrt{n(n+1)} U_{n,l}(\hat{x}).
$$
By applying these asymptotics to the series expansion of $E^s$, the conclusion follows.

\qed

A plane wave can be represented as a series expansion. The following lemma is proved in \cite{Klinken}.

\begin{lem}
Let $E^i$ be a plane wave, that is,
$E^{i}(x) = p \,e^{ik_m d\cdot x}$ with $d\in S$ and $p\cdot d=0$. 
Then we have the following series representation for a plane wave as follows:
$$
E^i(x)=
\sum_{n=1}^{\infty} \sum_{l=-n}^n \alpha^{pw,TE}_{nl} \widetilde E_{n,l}^{TE}(x;k_m) + \alpha^{pw,TM}_{nl} \widetilde E_{n,l}^{TM}(x;k_m),
$$
where
$$
\begin{cases}
\ds\alpha^{pw,TE}_{nl} = \frac{(-1)4\pi i^n}{\sqrt{n(n+1)}}  i\big( V_{n,l}(d) \cdot p\big),
\\[1em]
\ds\alpha^{pw,TM}_{nl} = \frac{(-1)4\pi i^n}{\sqrt{n(n+1)}}  \sqrt{\frac{\eps_m}{\mu_m}}\big(U_{n,l}(d)\cdot p\big).
\end{cases}
$$

\end{lem}

\smallskip
Now we are ready to compute the extinction cross-section $Q^{ext}$.

\begin{thm} Assume that $E^{i}(x) = p \,e^{ik_m d\cdot x}$  with $d\in S$ and $p\cdot d=0$. Let $D$ be a sphere with radius $r$.
Then the extinction cross-section is given by
\begin{align*}
Q^{ext}  
&= \sum_{n=1}^{\infty} \sum_{l=-n}^n \frac{(4\pi)^3}{k_m^2|p|^2} 
\Im\left( (-1)S_n^{TE} (V_{n,l}(d)\cdot p )^2+ i S_n^{TM}(U_{n,l}(d)\cdot p)^2 \right). 
\end{align*}
Moreover, for small $r>0$, we have
\begin{align*}
Q^{ext} 
&=  \sum_{l=-1}^1 \frac{(-1)(4\pi k_m r)^3}{k_m^2|p|^2} 
\Im\left( i\frac{2}{3}\frac{\mu_c-\mu_m}{2\mu_m+\mu_c} (V_{1,l}(d)\cdot p )^2+ \frac{2}{3}\frac{\eps_c-\eps_m}{2\eps_m+\eps_c}(U_{1,l}(d)\cdot p)^2 \right)
\\ &\quad +O((k_m r)^4). 
\end{align*}

\end{thm}

\proof
Let us first compute the scattering amplitude $A_\infty$ when $E^i$ is a plane wave. From 
\begin{align*}
A_\infty(\hat{x}) &= 
\sum_{n=1}^{\infty} \sum_{l=-n}^n\frac{4\pi(-i)^{n}}{i k_m}\sqrt{n(n+1)} 
\\
&\qquad\times
\left(\alpha^{pw,TE}_{nl} S^{TE}_{n}  V_{n,l}(\hat{x}) + \sqrt{\frac{\mu_m}{\eps_m}}\alpha^{pw,TM}_{nl} S^{TM}_{n}  U_{n,l}\right)
\\
&=
\sum_{n=1}^{\infty} \sum_{l=-n}^n \frac{(4\pi)^2}{k_m} 
\left( (-1)S_n^{TE} (V_{n,l}(d)\cdot p ) V_{n,l}+i S_n^{TM}(U_{n,l}(d)\cdot p) U_{n,l} \right).
\end{align*}
Therefore, we have
\begin{align*}
Q^{ext} &= \f{4\pi}{k_m}\Im\left[\f{p\cdot A_{\infty}(d)}{|p|^2}\right]
\\
&= \sum_{n=1}^{\infty} \sum_{l=-n}^n \frac{(4\pi)^3}{k_m^2|p|^2} 
\Im\left( (-1)S_n^{TE} (V_{n,l}(d)\cdot p )^2+ i S_n^{TM}(U_{n,l}(d)\cdot p)^2 \right). 
\end{align*}

Now we assume that $r \ll 1$. By applying \eqref{spherical_bessel_asymp}, one can easily see that
\begin{align*}
S_1^{TE} &= i\frac{2}{3}\frac{(\mu_c-\mu_m)(k_m r)^3}{2\mu_m+\mu_c} +O(r^4),\\
S_1^{TM} &= i\frac{2}{3}\frac{(\eps_c-\eps_m)(k_m r)^3}{2\eps_m+\eps_c} +O(r^4),\\[0.5em]
S_n^{TE},&\,S_n^{TM} = O(r^4), \quad \mbox{for } n\geq2.
\end{align*}
Therefore, we obtain, up to an error term of the order $O(r^4)$,
\begin{align*}
Q^{ext} 
&=  \sum_{l=-1}^1 \frac{(-1)(4\pi)^3}{k_m^2|p|^2} 
\Im\left( i\frac{2}{3}\frac{(\mu_c-\mu_m)(k_m r)^3}{2\mu_m+\mu_c} (V_{1,l}(d)\cdot p )^2+ \frac{2}{3}\frac{(\eps_c-\eps_m)(k_m r)^3}{2\eps_m+\eps_c}(U_{1,l}(d)\cdot p)^2 \right). 
\end{align*}
The proof is complete.

\qed

\section{Explicit computations for a spherical shell} \label{sectshell}

\subsection{Explicit representation of boundary integral operators}
Let $D_s$ and $D_c$ be a spherical shell with radius $r_s$ and $r_c$ with $r_s>r_c>0$.
Let
$$
(\eps,\mu)=\begin{cases}
(\eps_m,\mu_m) &\quad \mbox{in }D_c,
\\
(\eps_s,\mu_s) &\quad  \mbox{in }D_s\setminus\bar{D_c},
\\
(\eps_m,\mu_m) &\quad  \mbox{in }\mathbb{R}^3\setminus \bar{D_s}.
\end{cases}
$$
Let
$$
\rho=\frac{r_c}{r_s}.
$$
The solution to the transmission problem can be represented as follows
\be 
E(x) = 
\left\{ \begin{array}{ll}
\mu_c\nabla\times\vec{\mathcal{S}}_{D_s}^{k_c}[\psi_s](x) + \nabla\times\nabla\times\vec{\mathcal{S}}_{D_s}^{k_c}[\phi_s](x)
\\
\qquad\qquad
+\mu_c\nabla\times\vec{\mathcal{S}}_{D_c}^{k_c}[\psi_c](x) + \nabla\times\nabla\times\vec{\mathcal{S}}_{D_c}^{k_c}[\phi_c](x)
 & \quad x\in D_c,
\\[0.5em]
\mu_s\nabla\times\vec{\mathcal{S}}_{D_s}^{k_s}[\psi_s](x) + \nabla\times\nabla\times\vec{\mathcal{S}}_{D_s}^{k_s}[\phi_s](x) 
\\
\qquad\qquad
+\mu_s\nabla\times\vec{\mathcal{S}}_{D_c}^{k_s}[\psi_c](x) + \nabla\times\nabla\times\vec{\mathcal{S}}_{D_c}^{k_s}[\phi_c](x)
& \quad x\in D_s \setminus \bar{D_c},
\\[0.5em]
E^{i}+\mu_m\nabla\times\vec{\mathcal{S}}_{D_s}^{k_m}[\psi_s](x) + \nabla\times\nabla\times\vec{\mathcal{S}}_{D_s}^{k_m}[\phi_s](x) 
\\
\qquad \qquad +\mu_m\nabla\times\vec{\mathcal{S}}_{D_c}^{k_m}[\psi_c](x) + \nabla\times\nabla\times\vec{\mathcal{S}}_{D_c}^{k_m}[\phi_c](x) 
& \quad x\in \mathbb{R}^3\backslash\bar{D_s},
\end{array}
\right.
\ee
and
\be
H(x) = -\f{i}{\om\mu_D}(\nabla\times E)(x) \quad x\in \mathbb{R}^3\backslash \p D,
\ee
where the pair $(\psi_s,\phi_s,\psi_c,\phi_c)\in \big(H^{-\f{1}{2}}_T(\textnormal{div}, \p D_s)\big)^2 \times \big(H^{-\f{1}{2}}_T(\textnormal{div}, \p D_c)\big)^2$ is the unique solution to
$$W^{sh}
\left(\begin{array}{c}
\psi_s\\
\phi_s\\
\psi_c\\
\phi_c
\end{array}\right):=
\left( \begin{array}{cc}
W^{sh}_{11} & W^{sh}_{12} \\
W^{sh}_{21} & W^{sh}_{22}
\end{array}\right)
\left(\begin{array}{c}
\psi_s\\
\phi_s\\
\psi_c\\
\phi_c
\end{array}\right)
=
\left(\begin{array}{c}
\nu\times E^i \\
i\omega\nu\times H^i\\
0\\
0
\end{array}\right)
$$
with
\be 
W^{sh}_{11}=\left( \begin{array}{cc}
\df{\mu_s + \mu_m}{2}Id + \mu_s\mathcal{M}_{D_s}^{k_s}-\mu_m\mathcal{M}_{D_s}^{k_m} & \mathcal{L}_{D_s}^{k_s}-\mathcal{L}_{D_s}^{k_m}\\
\mathcal{L}_{D_s}^{k_s}-\mathcal{L}_{D_s}^{k_m} & \left(\df{k_s^2}{2\mu_s}+\df{k_m^2}{2\mu_m}\right)Id + \df{k_s^2}{\mu_s}\mathcal{M}_{D_s}^{k_s}-\df{k_m^2}{\mu_m}\mathcal{M}_{D_s}^{k_m}  
\end{array} \right),
\ee
\be 
W^{sh}_{12}=\left( \begin{array}{cc}
\mu_s \nu \times\nabla\times \vec{\mathcal{S}}_{D_c}^{k_s}
-\mu_m \nu \times\nabla\times \vec{\mathcal{S}}_{D_c}^{k_m}
&
\nu \times\nabla\times\nabla\times \vec{\mathcal{S}}_{D_c}^{k_s}
-\nu \times\nabla\times\nabla\times \vec{\mathcal{S}}_{D_c}^{k_m}
\\
\nu \times\nabla\times\nabla\times \vec{\mathcal{S}}_{D_c}^{k_s}
-\nu \times\nabla\times\nabla\times \vec{\mathcal{S}}_{D_c}^{k_m}
& 
\df{k_s^2}{\mu_s} \nu \times\nabla\times \vec{\mathcal{S}}_{D_c}^{k_s}
-\df{k_m^2}{\mu_m} \nu \times\nabla\times \vec{\mathcal{S}}_{D_c}^{k_m}
\end{array} \right)\bigg|_{\p D_s} , 
\ee
\be 
W^{sh}_{21}=\left( \begin{array}{cc}
-\mu_c \nu \times\nabla\times \vec{\mathcal{S}}_{D_s}^{k_c}+\mu_s \nu \times\nabla\times \vec{\mathcal{S}}_{D_s}^{k_s}&
-\nu \times\nabla\times\nabla\times \vec{\mathcal{S}}_{D_s}^{k_c}+\nu \times\nabla\times\nabla\times \vec{\mathcal{S}}_{D_s}^{k_s}
\\
-\nu \times\nabla\times\nabla\times \vec{\mathcal{S}}_{D_s}^{k_c}+\nu \times\nabla\times\nabla\times \vec{\mathcal{S}}_{D_s}^{k_s}
 & 
 -\df{k_c^2}{\mu_c} \nu \times\nabla\times \vec{\mathcal{S}}_{D_s}^{k_c}
 +\df{k_s^2}{\mu_s} \nu \times\nabla\times \vec{\mathcal{S}}_{D_s}^{k_s}
\end{array} \right)\bigg|_{\p D_c} , 
\ee
\be 
W^{sh}_{22}=\left( \begin{array}{cc}
-\df{\mu_c + \mu_s}{2}Id - \mu_c\mathcal{M}_{D_c}^{k_c}+\mu_s\mathcal{M}_{D_c}^{k_s} & -\mathcal{L}_{D_c}^{k_c}+\mathcal{L}_{D_c}^{k_s}\\
-\mathcal{L}_{D_c}^{k_c}+\mathcal{L}_{D_c}^{k_s} & -\left(\df{k_c^2}{2\mu_c}+\df{k_s^2}{2\mu_s}\right)Id -\df{k_c^2}{\mu_c}\mathcal{M}_{D_c}^{k_c}+\df{k_s^2}{\mu_s}\mathcal{M}_{D_c}^{k_s}  
\end{array} \right).
\ee

Note that $W^{sh}_{11}$ and $W^{sh}_{22}$ are similar to the operator in left-hand side of \eqref{eq-Maxwell_System}. In the previous section for the sphere case, we have already obtained the matrix representation of this operator and its asymptotic expansion. 

By Lemma \ref{lem_SLayerpotential_sphere_formula}, we can represent $\nu\times\nabla\times \vec{\mathcal{S}}^k_{D}|_{|x|=r'}$ and $\nu\times\nabla\times\nabla\times \vec{\mathcal{S}}^k_{D}|_{|x|=r'}$ in a matrix form as follows(using $U_{n,m},V_{n,m}$ as basis): 

\noindent(i) For $r'>r$,
\begin{align}
\nu\times\nabla\times \vec{\mathcal{S}}^k_{D}|_{|x|=r'} &= 
\left( \begin{array}{cc}
(-i k r)  \mathcal{J}_n(k r)h_n^{(1)}(k r') & 0
\\
0 & ik \frac{r^2}{r'} j_n(kr)\mathcal{H}_n(k r') 
\end{array} \right),
\\
\nu\times\nabla\times\nabla\times \vec{\mathcal{S}}^k_{D}|_{|x|=r'} &= 
\left( \begin{array}{cc}
0 &ik (kr)^2 j_n(k r) h_n^{(1)}(k r')
\\
- ik \frac{r}{r'}\mathcal{J}_n(kr)\mathcal{H}_n(k r') &0
\end{array} \right);
\end{align}
(ii) For $r'<r$,
\begin{align}
\nu\times\nabla\times \vec{\mathcal{S}}^k_{D}|_{|x|=r'} &= 
\left( \begin{array}{cc}
(-i k r)  {j}_n(k r')\mathcal{H}_n(kr) & 0
\\
0 & ik \frac{r^2}{r'} \mathcal{J}_n(kr')\mathcal{h}^{(1)}_n(k r) 
\end{array} \right),
\\
\nu\times\nabla\times\nabla\times \vec{\mathcal{S}}^k_{D}|_{|x|=r'} &= 
\left( \begin{array}{cc}
0 &ik (kr)^2 j_n(k r') h_n^{(1)}(k r)
\\
- ik \frac{r}{r'}\mathcal{J}_n(kr')\mathcal{H}_n(k r) &0
\end{array} \right).
\end{align}
Using the above formulas,
the matrix representation of the operators $W^{sh}_{12}$ and $W^{sh}_{21}$ can be easily obtained. 

We now consider scaling of $W^{sh}$.
First we need some definitions.
Let $D_s = z + r_s B_s$ where $B_s$ contains the origin and $|B_s| = O(1)$. 
Let $B_c$ be defined in a similar way.
For any $x\in \p D_s$ (or $\p D_c$), let $\widetilde{x} = \f{x-z}{r_s}\in \p B_s$ (or $\p B_c$ with $r_s$ replaced by $r_c$) and define for each function $f$ defined on $\p D_s$ (or $\p D_c$), a corresponding function defined on $B$ as follows
\be 
\eta_s(f)(\widetilde{x}) = f(z + r_s \widetilde{x}),
\quad
\eta_c(f)(\widetilde{x}) = f(z + r_c \widetilde{x}).
\ee
Then, in a similar way to the sphere case, let us write
$$W_B^{sh}(r_s)
\left(\begin{array}{c}
\eta_s(\psi_s)\\
\omega\eta_s(\phi_s)\\
\eta_c(\psi_c)\\
\omega\eta_c(\phi_c)
\end{array}\right)
=
\left(\begin{array}{c}
\frac{\eta(\nu\times E^i)}{\mu_m-\mu_s} \\
\frac{\eta(i\nu\times H^i)}{\eps_m-\eps_s}\\
0\\
0
\end{array}\right).
$$

Using $(U_{n,m},V_{n,m},U_{n,m},V_{n,m})\times(U_{n,m},V_{n,m},U_{n,m},V_{n,m})$ as basis, we can represent $W^{sh}_B(r_s)$ in a $8\times 8$ matrix form in a subspace $H_{n,m}(\p B_s) \times H_{n,m}(\p B_c)$.  
Then, by using 
\eqref{eqn_bessel_product_asymptotic}, 
their asymptotic expansion can also be obtained.

Here, the resulting asymptotics of the matrix $W^{sh}_{B}$ are given as follows. Write
\be
\mathcal{W}^{sh}_B(r_s) = \mathcal{W}^{sh}_{B,0} + r_s \mathcal{W}^{sh}_{B,1} + r_s^2 \mathcal{W}^{sh}_{B,2} + O(r_s^3),
\ee
where
\begin{align}
\mathcal{W}_{B,0}^{sh}=\begin{pmatrix}
\Lambda_{\mu,\eps} & \\
& \Lambda_{\mu,\eps}
\end{pmatrix}
+
\begin{pmatrix}
  P_{0,n} & Q_{0,n} \\
R_{0,n} &-P_{0,n}
\end{pmatrix} ,
\end{align}
$$
\mathcal{W}_{B,1}^{sh} = \begin{pmatrix}
 P_{1,n} & Q_{1,n} \\
R_{1,n} & -P_{1,n}
\end{pmatrix}
, \quad
\mathcal{W}_{B,2}^{sh} = \begin{pmatrix}
 P_{2,n} & Q_{2,n} \\
R_{2,n} & -P_{2,n}
\end{pmatrix} .
$$
Here, the matrix $P_{j,n},Q_{j,n}$ and $R_{j,n}$ are given by
\begin{align*}
\Lambda_{\mu,\eps} = \begin{pmatrix}
 \lambda_\mu & & &\\ 
  & \lambda_\mu & & \\
  & & \lambda_\eps &\\
  & & & \lambda_\eps
 \end{pmatrix}
 ,
 \quad
 P_{0,n}=
 \begin{pmatrix}
 p_n & & & \\
 & -p_n & &\\
 & & p_n & \\
 & & & -p_n
 \end{pmatrix},
\end{align*}
\begin{align*}
Q_{0,n} = \rho^2\begin{pmatrix}
 g_n & & &\\ 
  & f_n & & \\
  & & g_n &\\
  & & & f_n
 \end{pmatrix}
 ,
 \quad
 R_{0,n}=\begin{pmatrix}
 f_n & & &\\ 
  & g_n & & \\
  & & f_n &\\
  & & & g_n
 \end{pmatrix},
\end{align*}
\begin{align*}
P_{1,n}&=\omega 
\begin{pmatrix}
& & & C_{\mu} p_n
\\
\ds  &  & 
 C_{\mu}q_n & 
\\
  &   C_{\eps} p_n &  & 
\\
C_{\eps}q_n &  &  & 
\end{pmatrix},
\quad
P_{2,n}=\omega^2\begin{pmatrix}
 D_\mu r_n
& 
&
&
\\
\ds  & 
 D_\mu s_n
 & 
 & 
\\
\ds  &  & 
 D_\eps r_n
 & 
\\
 &  &  & 
 D_\eps s_n
\end{pmatrix},
\end{align*}
\begin{align*}
Q_{1,n}&=
{\omega}\rho
\left(\begin{array}{cccc}
 &  & &  C_\mu \tilde p_n \\
 &  &  C_\mu \tilde q_n&  \\
 & C_\eps \tilde p_n &  &  \\
C_\eps\tilde q_n &  & &
\end{array}\right),
\quad
Q_{2,n}=\omega^2 \rho
\left(\begin{array}{cccc}
D_\mu \tilde r_n
&  &   &  \\
 & D_\mu\tilde s_n &  &  \\
 &  & D_\eps\tilde r_n &  \\
 &  & & D_\eps \tilde s_n
\end{array}\right),
\end{align*}
\begin{align*}
R_{1,n}=(-1)\omega\rho^{-1}\left(\begin{array}{cccc}
 &  &  &  C_\mu \tilde p_n \\
 &  &  C_\mu \tilde q_n&  \\
 & C_\eps \tilde p_n &  &  \\
C_\eps\tilde q_n &  & &
\end{array}\right), \quad
R_{2,n}=\omega^2 \rho^{-1}
\left(\begin{array}{cccc}
D_\mu
\tilde s_n&  &   &  \\
 & -D_\mu\tilde r_n &  &  \\
 &  & D_\eps\tilde s_n &  \\
 &  & & -D_\eps\tilde r_n
\end{array}\right).
\end{align*}

Here, $p_n,q_n,r_n,s_n$ are defined as \eqref{eqn_pqrs_def} and 
$\tilde p_n,\tilde q_n,\tilde r_n,\tilde s_n,D_\mu$ and $D_\eps$ are defined as follows:
\begin{align}
f_n &= \rho^{n}\frac{n}{2n+1}, \quad g_n = \rho^{n-1}\frac{n+1}{2n+1},
\\
\tilde p_n &= \frac{1}{2n+1}\rho^{n+1},
\\
\tilde q_n &= \frac{(n+1)(n-2)}{2(2n-1)(2n+1)}\rho^{n} - \frac{n(n+3)}{2(2n+1)(2n+3)}\rho^{n+2},
\\
\tilde r_n &= -\frac{n+1}{2(2n-1)(2n+1)}\rho^n+\frac{(n+3)}{2(2n+1)(2n+3)}\rho^{n+2},
\\
\tilde s_n &= -\frac{n-2}{2(2n-1)(2n+1)}\rho^{n+1}+\frac{n}{2(2n+1)(2n+3)}\rho^{n+3},
\end{align}
and 
\begin{align}
D_{\mu} &= \frac{\eps_s\mu_s^2-\eps_m \mu_m^2}{\mu_m-\mu_s}, \quad D_\eps=\frac{\eps_s^2\mu_s-\eps_m^2 \mu_m}{\eps_m-\eps_s}.
\end{align}

\subsection{Asymptotic behavior of the spectrum of $\mathcal{W}_B^{sh}(r_s)$}
Let us define
$$
\lambda^{sh}_n = \frac{1}{2(2n+1)}\sqrt{1+4n(n+1)\rho^{2n+1}}.
$$
Note that $\pm\lambda^{sh}_n$ are eigenvalues of the Neumann-Poincar\'e operator on the shell.

It turns out that the eigenvalues of $W_{B,0}^{sh}$ are as follows
$$
\lambda_\mu + \lambda^{sh}_n, \quad
\lambda_\mu - \lambda^{sh}_n, \quad
\lambda_\eps + \lambda^{sh}_n, \quad
\lambda_\eps - \lambda^{sh}_n,
$$ for 
$n=0,1,2,...$,
and their multiplicities is $2$. Their associated eigenfunctions are as follows:
\begin{align*}
\lambda_\mu + \lambda^{sh}_n \quad &\longrightarrow \quad E_{1}^0:=(\lambda^{sh}_n + p_n)\mathbf{e}_1 + f_n \mathbf{e}_5, \quad E_2^0:=(\lambda^{sh}_n - p_n)\mathbf{e}_2 + g_n \mathbf{e}_6,
\\
\lambda_\mu - \lambda^{sh}_n \quad &\longrightarrow \quad E_3^0:=(-\lambda^{sh}_n + p_n)\mathbf{e}_1 + f_n \mathbf{e}_5, \quad E_4^0:=(-\lambda^{sh}_n -p_n) \mathbf{e}_2 +g_n\mathbf{e}_6,
\\
\lambda_\eps + \lambda^{sh}_n \quad &\longrightarrow \quad E_5^0:=(\lambda^{sh}_n + p_n) \mathbf{e}_3+ f_n\mathbf{e}_7, 
\quad
E_6^0:=(\lambda^{sh}_n - p_n) \mathbf{e}_4+ g_n\mathbf{e}_8,
\\
\lambda_\eps - \lambda^{sh}_n \quad &\longrightarrow \quad E_7^0:=(-\lambda^{sh}_n + p_n) \mathbf{e}_3+ f_n\mathbf{e}_7,
\quad
E_8^0:=
(-\lambda^{sh}_n - p_n) \mathbf{e}_4+ g_n\mathbf{e}_8,
\end{align*}
where $\{\mathbf{e}_i\}_{i=1}^8$ is standard unit basis in $\mathbb{R}^8$.

To derive asymptotic expansions  of the eigenvalues, we apply 
degenerate eigenvalue perturbation theory (since the multiplicity of each of these eigenvalues is $2$). To state the result, we need some definitions.
Let 
\begin{align*}
T_{16,n} &= C_\eps\frac{(L-p_n)a_{1,n}-b_{1,n}}{|E^0_1| |E^0_6|}
,\quad
T_{18,n} = C_\eps\frac{(-L-p_n)a_{1,n}-b_{1,n}}{|E^0_1| |E^0_8|},
\\
T_{25,n} &= C_\eps\frac{(L+p_n)a_{2,n}-b_{2,n}}{|E^0_2| |E^0_5|}
,\quad
T_{27,n} = C_\eps\frac{(-L+p_n)a_{2,n}-b_{2,n}}{|E^0_2| |E^0_7|},
\\
T_{36,n} &= C_\eps\frac{(L-p_n)a_{3,n}-b_{3,n}}{|E^0_3| |E^0_6|},
\quad
T_{38,n} = C_\eps\frac{(-L-p_n)a_{3,n}-b_{3,n}}{|E^0_3| |E^0_8|},
\\
T_{45,n} &= C_\eps\frac{(L+p_n)a_{4,n}-b_{4,n}}{|E^0_4| |E^0_5|}
,\quad
T_{47,n} = C_\eps\frac{(-L+p_n)a_{4,n}-b_{4,n}}{|E^0_4| |E^0_7|},
\\
T_{52,n} &= \frac{C_\mu}{C_\eps}T_{16,n},
\quad
T_{54,n} = \frac{C_\mu}{C_\eps}T_{18,n},
\quad
T_{61,n}=\frac{C_\mu}{C_\eps}T_{25,n}, \quad T_{63,n}=\frac{C_\mu}{C_\eps}T_{27,n}, 
\\
T_{72,n}&=\frac{C_\mu}{C_\eps}T_{36,n}, \quad T_{74,n}=\frac{C_\mu}{C_\eps}T_{38,n}, 
\quad
T_{81,n}=\frac{C_\mu}{C_\eps}T_{45,n}, \quad T_{83,n}=\frac{C_\mu}{C_\eps}T_{47,n},
\end{align*}

where
\begin{align*}
a_{1,n} &= (\lambda_n^{sh}+p_n)q_n + \rho f_n \tilde q_n, 
\\
a_{2,n} &= (\lambda_n^{sh}-p_n)p_n + \rho g_n \tilde p_n, 
\\
a_{3,n}&= (-\lambda_n^{sh}+p_n)q_n + \rho f_n \tilde q_n, 
\\
a_{4,n}&= (-\lambda_n^{sh}-p_n)p_n + \rho g_n \tilde p_n, 
\end{align*}
and
\begin{align*}
b_{1,n} &= f_n g_n q_n + \rho^{-1}(\lambda^{sh}_n + p_n) 
g_n\tilde q_n,
\\
b_{2,n} &=f_n g_n p_n + \rho^{-1}(\lambda^{sh}_n - p_n) f_n\tilde  p_n,
\\
b_{3,n}&= f_n g_n q_n + \rho^{-1}(-\lambda^{sh}_n + p_n) g_n \tilde q_n ,
\\
b_{4,n}&= f_n g_n p_n + \rho^{-1}(-\lambda^{sh}_n - p_n) f_n \tilde p_n.
\end{align*}
We also define
\begin{align*}
K_{1,n} &= D_\mu\frac{(\lambda_n^{sh}+p_n)((\lambda_n^{sh}+p_n)r_n+\rho f_n \tilde r_n) + f_n((\lambda_n^{sh}+p_n)\rho^{-1} \tilde s_n -f_n r_n)}{|E^0_1|^2},
\\
K_{2,n} &=D_\mu\frac{g_n((-\lambda_n^{sh}+p_n)\rho^{-1} \tilde r_n-g_n s_n )+(\lambda_n^{sh}-p_n)((\lambda_n^{sh}-p_n)s_n + \rho g_n\tilde s_n)}{|E^0_2|^2},
\\
K_{3,n} &=D_\mu\frac{(-\lambda_n^{sh}+p_n)((-\lambda_n^{sh}+p_n)r_n+\rho f_n \tilde r_n) + f_n((-\lambda_n^{sh}+p_n)\rho^{-1} \tilde s_n -f_n r_n)}{|E^0_3|^2},
\\
K_{4,n} &=D_\mu\frac{g_n((\lambda_n^{sh}+p_n)\rho^{-1} \tilde r_n-g_n s_n )+(-\lambda_n^{sh}-p_n)((-\lambda_n^{sh}-p_n)s_n + \rho g_n\tilde s_n)}{|E^0_4|^2},
\\
K_{5,n}&=\frac{D_\eps}{D_\mu}K_{1,n},
\quad
K_{6,n}=\frac{D_\eps}{D_\mu}K_{2,n},
\quad
K_{7,n}=\frac{D_\eps}{D_\mu}K_{3,n},
\quad
K_{8,n}=\frac{D_\eps}{D_\mu}K_{4,n}.
\end{align*}
Now we are ready to state the result. The followings are asymptotics of eigenvalues of $W_{B}^{sh}(r_s)$
\begin{align*}
&\lambda_\mu + \lambda_\eps+ (r_s \omega)^2 \left( \frac{T_{16,n}T_{61,n}}{\lambda_\mu-\lambda_\eps}+\frac{T_{18,n}T_{81,n}}{\lambda_\mu-\lambda_\eps+2\lambda_n^{sh}} + K_{1,n}\right) +O(r_s^3),
\\
&\lambda_\mu + \lambda_\eps+ (r_s \omega)^2 \left( \frac{T_{16,n}T_{61,n}}{\lambda_\mu-\lambda_\eps}+\frac{T_{18,n}T_{81,n}}{\lambda_\mu-\lambda_\eps+2\lambda_n^{sh}} + K_{2,n}\right)+O(r_s^3),
\\
&
\lambda_\mu - \lambda_\eps+ (r_s \omega)^2 \left( \frac{T_{36,n}T_{63,n}}{\lambda_\mu-\lambda_\eps-2\lambda_n^{sh}}+\frac{T_{38,n}T_{83,n}}{\lambda_\mu-\lambda_\eps} + K_{3,n}\right)+O(r_s^3),
\\
&\lambda_\mu - \lambda_\eps+ (r_s \omega)^2 \left( \frac{T_{36,n}T_{63,n}}{\lambda_\mu-\lambda_\eps-2\lambda_n^{sh}}+\frac{T_{38,n}T_{83,n}}{\lambda_\mu-\lambda_\eps} + K_{4,n}\right)+O(r_s^3),
\\
&\lambda_\eps + \lambda_\mu+ (r_s \omega)^2 \left( \frac{T_{52,n}T_{25,n}}{\lambda_\eps-\lambda_\mu}+\frac{T_{54,n}T_{45,n}}{\lambda_\eps-\lambda_\mu+2\lambda_n^{sh}} + K_{5,n}\right)+O(r_s^3),
\\
&\lambda_\eps + \lambda_\mu+ (r_s \omega)^2 \left( \frac{T_{52,n}T_{25,n}}{\lambda_\eps-\lambda_\mu}+\frac{T_{54,n}T_{45,n}}{\lambda_\eps-\lambda_\mu+2\lambda_n^{sh}} + K_{6,n}\right)+O(r_s^3),
\\
&\lambda_\eps - \lambda_\mu+ (r_s \omega)^2 \left( \frac{T_{72,n}T_{27,n}}{\lambda_\eps-\lambda_\mu-2\lambda_n^{sh}}+\frac{T_{74,n}T_{47,n}}{\lambda_\eps-\lambda_\mu} + K_{7,n}\right)+O(r_s^3),
\\
&\lambda_\eps - \lambda_\mu+ (r_s \omega)^2 \left( \frac{T_{72,n}T_{27,n}}{\lambda_\eps-\lambda_\mu-2\lambda_n^{sh}}+\frac{T_{74,n}T_{47,n}}{\lambda_\eps-\lambda_\mu} + K_{8,n}\right)+O(r_s^3).
\end{align*}
We also have the following asymptotic expansions of the eigenfunctions:
\begin{align*}
&E^0_1+ r_s \omega\left( \frac{T_{16,n}}{\lambda_\mu-\lambda_\eps}E_6^0+\frac{T_{18,n}}{\lambda_\mu-\lambda_\eps+2\lambda_n^{sh}}E_8^0 \right) +O(r_s^2),
\\
&E^0_2+ r_s \omega \left( \frac{T_{25,n}}{\lambda_\mu-\lambda_\eps}E^0_5+\frac{T_{27,n}}{\lambda_\mu-\lambda_\eps+2\lambda_n^{sh}}E^0_7\right)+O(r_s^2),
\\
&
E^0_3+ r_s \omega \left( \frac{T_{36,n}}{\lambda_\mu-\lambda_\eps-2\lambda_n^{sh}}E^0_6+\frac{T_{38,n}}{\lambda_\mu-\lambda_\eps}E^0_8\right)+O(r_s^2),
\\
&E^0_4+ r_s \omega \left( \frac{T_{45,n}}{\lambda_\mu-\lambda_\eps-2\lambda_n^{sh}}E^0_5+\frac{T_{47,n}}{\lambda_\mu-\lambda_\eps} E^0_7\right)+O(r_s^2),
\\
&E^0_5+ r_s \omega\left( \frac{T_{52,n}}{\lambda_\mu-\lambda_\eps}E_2^0+\frac{T_{54,n}}{\lambda_\mu-\lambda_\eps+2\lambda_n^{sh}}E_4^0 \right) +O(r_s^2),
\\
&E^0_6+ r_s \omega \left( \frac{T_{61,n}}{\lambda_\mu-\lambda_\eps}E^0_1+\frac{T_{63,n}}{\lambda_\mu-\lambda_\eps+2\lambda_n^{sh}}E^0_3\right)+O(r_s^2),
\\
&
E^0_7+ r_s \omega \left( \frac{T_{72,n}}{\lambda_\mu-\lambda_\eps-2\lambda_n^{sh}}E^0_2+\frac{T_{74,n}}{\lambda_\mu-\lambda_\eps}E^0_4\right)+O(r_s^2),
\\
&E^0_8+ r_s \omega \left( \frac{T_{81,n}}{\lambda_\mu-\lambda_\eps-2\lambda_n^{sh}}E^0_1+\frac{T_{83,n}}{\lambda_\mu-\lambda_\eps} E^0_3\right)+O(r_s^2).
\end{align*}

Interestingly, the first-order term (of order $\delta$) is still zero in the asypmtotic expansions of the eigenvalues. This is due to the fact that degenerate eigenfunctions does not interact with each other.

\section{Plasmonic resonances for the anisotropic problem}
\label{sec-Anisitrop max}
In this section, we consider the scattering problem of a time-harmonic wave $u^i$, incident on a plasmonic anisotropic nanoparticle. The homogeneous medium is characterized by electric permittivity $\varepsilon_m$, while the particle occupying a bounded and simply connected domain $\Omega\Subset\mathbb{R}^3$ of class $\mathcal{C}^{1,\alpha}$ for $0<\alpha<1$ is characterized by electric anisotropic permittivity $A$. We consider $A$ to be a positive-definite symmetric matrix.

In the quasi-static regime the problem can be modeled as follows
\be \label{eq-Anisitropic_Cond_Eq max}
\begin{array}{c}
\nabla\cdot\big(\eps_m Id \chi(\R^3 \backslash \bar{\Omega}) + A \chi(\Omega)\big)\nabla u = 0,\\
\nm
|u-u^i| = O(|x|^{-2}), \quad |x| \rightarrow + \infty, 
\end{array}
\ee
where $\chi$ denotes the characteristic function and $u^i$ is a harmonic function in $\R^3$.\\
We are interested in finding the plasmonic resonances for problem \eqref{eq-Anisitropic_Cond_Eq max}.\\
First, introduce the fundamental solution to the operator $\nabla\cdot A \nabla$ in dimension three
\beas
G^A(x)= - \f{1}{4\pi \sqrt{\textnormal{det}(A)} |A_*x |}
\eeas
with $A_* = \sqrt{A^{-1}}$. From now on we will note $G^A(x,y):=G^A(x-y)$.\\
The single-layer potential associated with $A$ is
\beas \label{eq-SLP_A}
\mathcal{S}_{\Omega}^{A} [\varphi]: H^{-\f{1}{2}}(\p \Omega) &\longrightarrow & H^{\f{1}{2}}(\p \Omega) \\
\varphi &\longmapsto & \mathcal{S}_{\Omega}^{A} [\varphi](x) = \int_{\p \Omega} G^A(x, y) \varphi(y) d\sigma(y),  \quad x \in \R^3.\nonumber
\eeas
We can represent the unique solution \cite{book2}  to \eqref{eq-Anisitropic_Cond_Eq max} in the following form:
\beas
u(x) = \left\{
\begin{array}{lr}
u^i + \mathcal{S}_{\Omega} [\psi], & \quad x \in \R^3 \backslash \bar{\Omega},\\
\mathcal{S}_{\Omega}^{A} [\phi] ,  & \quad x \in {\Omega},
\end{array}\right.
\eeas
where $(\psi,\phi)\in \big(H^{-\f{1}{2}}(\p \Omega)\big)^2$ is the unique solution to the following system of integral equations on $\partial \Omega$:
\be \label{Anisotropic-syst max}
\left\{
\begin{array}{lr}
\mathcal{S}_{\Omega} [\psi] - \mathcal{S}_{\Omega}^{A} [\phi] &= -u^i,  \\
\nm
\eps_m\df{\p \mathcal{S}_{\Omega}[\psi]}{\p \nu}\Big\vert_{+} - \nu\cdot A\nabla\mathcal{S}^A_{\Omega}[\phi]\Big\vert_{-} &= -\eps_m\df{\p u^i}{\p \nu}.
\end{array} \right.
\ee

\begin{lem} \label{lem-anisotropic_S_A max}
The operator $\mathcal{S}_{\Omega}^A: H^{-\f{1}{2}}(\p \Omega) \rightarrow H^{\f{1}{2}}(\p \Omega)$ is invertible. Moreover, we have the jump formula
\beas
\nu\cdot A\nabla\mathcal{S}^A_{\Omega}\Big\vert_{\pm} = \pm\f{1}{2}Id + (\mathcal{K}_{\Omega}^A)^*,
\eeas
with
\beas
(\mathcal{K}^A_{\Omega})^*[\varphi](x) = \int_{\p {\Omega}}-\f{\big(x-y,\nu(x)\big)}{4\pi \sqrt{\textnormal{det}(A)} |A_*(x-y)|^3}\varphi(y)d\sigma(y).
\eeas
\end{lem}

\begin{proof}
Let $\mathcal{T}_{A_*} \in \mathcal{L}(H^s(\p \widetilde{\Omega}),H^s(\p {\Omega}))$ be such that $\mathcal{T}_{A_*}[\varphi](x) = \varphi(A_*x)$ for $\varphi \in H^s(\p \widetilde{\Omega})$ and $\widetilde{\Omega} = A_* {\Omega}$. Let $r_{\nu} \in \mathcal{L}(H^s(\p {\Omega}),H^s(\p {\Omega}))$ be such that $r_{\nu}[\varphi](x) =  |A_*^{-1}\nu(x) |\varphi(x)$. It follows by the change of variables $\widetilde{y} = A_*y$ that $d\sigma(\widetilde{y}) = \textnormal{det}\sqrt{A_*} |A_*^{-1}\nu(y)| d\sigma(y)$. Thus,
\beas
\mathcal{S}_{\Omega}^A = \mathcal{T}_{A_*}\mathcal{S}_{\widetilde{\Omega}}\mathcal{T}_{A_*}^{-1}r_{\nu}^{-1},
\eeas
and in particular $\mathcal{S}_{\Omega}^A$ is invertible and its inverse  $(\mathcal{S}_{\Omega}^A)^{-1} =  r_{\nu}\mathcal{T}_{A_*}\mathcal{S}_{\widetilde{\Omega}}^{-1}\mathcal{T}_{A_*}^{-1}$.

Note that, for $x \in \p {\Omega}$,
\beas
\widetilde{\nu}(\widetilde{x}) = \f{A_*^{-1}\nu(x)}{|A_*^{-1}\nu(x)|},
\eeas
where $\widetilde{\nu}(\widetilde{x})$ is the outward normal to $\p \widetilde{\Omega}$ at $\widetilde{x}=A_*x$.
We have
\bea
\nu\cdot A\nabla\mathcal{S}^A_{\Omega}\Big\vert_{\pm} &=& \nu\cdot A\nabla_x\Big(\mathcal{T}_{A_*}\mathcal{S}_{\widetilde{\Omega}}\mathcal{T}_{A_*}^{-1}r_{\nu}^{-1}\Big)\Big\vert_{\pm}\nonumber \\ 
&=& \nu\cdot AA_*\Big(\mathcal{T}_{A_*}\nabla_{\widetilde{x}}\mathcal{S}_{\widetilde{\Omega}}\mathcal{T}_{A_*}^{-1}r_{\nu}^{-1}\Big)\Big\vert_{\pm}\nonumber \\
&=& |A_*^{-1}\nu |\widetilde{\nu}\cdot\Big(\mathcal{T}_{A_*}\nabla_{\widetilde{x}}\mathcal{S}_{\widetilde{\Omega}}\mathcal{T}_{A_*}^{-1}r_{\nu}^{-1}\Big)\Big\vert_{\pm}\nonumber \\
&=& \pm\f{1}{2}Id + (r_{\nu}\mathcal{T}_{A_*})\mathcal{K}_{\widetilde{\Omega}}^*(r_{\nu}\mathcal{T}_{A_*})^{-1}.
\eea
The result follows from a change of variables in the expression of the operator $(\mathcal{K}^A_{\Omega})^* := (r_{\nu}\mathcal{T}_{A_*})\mathcal{K}_{\widetilde{\Omega}}^*(r_{\nu}\mathcal{T}_{A_*})^{-1}$.
\end{proof}

%

\begin{lem} \label{lem-anisotrop Calderon max}
$\mathcal{S}_{\Omega}^A$ is negative definite for the duality pairing $(\cdot,\cdot)_{-\f{1}{2},\f{1}{2}}$ and we can define a new inner product
\beas
(u,v)_{\mathcal{H}^*_{A}}=-(u,\mathcal{S}_{\Omega}^A[v])_{-\f{1}{2},\f{1}{2}},
\eeas
which is equivalent to $(\cdot,\cdot)_{-\f{1}{2},\f{1}{2}}$.
\end{lem}
\begin{proof}
Let $\varphi \in H^{-\f{1}{2}}(\p {\Omega})$. Using Lemma \ref{lem-anisotropic_S_A max}, we have
\beas
\varphi = \nu\cdot A\nabla\mathcal{S}^A_{\Omega}[\varphi]\Big\vert_{+} - \nu\cdot A\nabla\mathcal{S}^A_{\Omega}[\varphi]\Big\vert_{-}.
\eeas
Thus
\beas
\int_{\p \Omega}\varphi(x)\mathcal{S}_{\Omega}^A[\varphi](x)d\sigma(x) &=& \int_{\p \Omega}\nu\cdot A\nabla\mathcal{S}^A_{\Omega}[\varphi]\Big\vert_{+}(x)\mathcal{S}_{\Omega}^A[\varphi](x)d\sigma(x) - \int_{\p \Omega}\nu\cdot A\nabla\mathcal{S}^A_{\Omega}[\varphi]\Big\vert_{-}(x)\mathcal{S}_{\Omega}^A[\varphi](x)d\sigma(x)\\
&=& -\int_{\R^3 \backslash \bar{\Omega}}\nabla\mathcal{S}_{\Omega}^A[\varphi](x)\cdot A \nabla\mathcal{S}_{\Omega}^A[\varphi](x) d\sigma(x) - \int_{\R^3 \backslash \bar{\Omega}}\mathcal{S}_{\Omega}^A[\varphi](x) \nabla \cdot A \nabla\mathcal{S}_{\Omega}^A[\varphi](x) d\sigma(x)\\
&& \quad -\int_{\Omega} \nabla\mathcal{S}_{\Omega}^A[\varphi](x)\cdot A \nabla\mathcal{S}_{\Omega}^A[\varphi](x) d\sigma(x) + \int_{\Omega}\mathcal{S}_{\Omega}^A[\varphi](x) \nabla \cdot A \nabla\mathcal{S}_{\Omega}^A[\varphi](x) d\sigma(x)\\
 &=& -\int_{\R^3}\nabla\mathcal{S}_{\Omega}^A[\varphi](x)\cdot A \nabla\mathcal{S}_{\Omega}^A[\varphi](x) d\sigma(x) \leq 0,
\eeas
where the equality is achieved if and only if $\varphi = 0$. Here we have used an integration by parts, the fact that $\mathcal{S}_{\Omega}^A[\varphi](x)=O(|x|^{-1})$ as $|x|\rightarrow \infty$, $\nabla \cdot A \nabla\mathcal{S}_{\Omega}^A[\varphi](x) = 0$ for $x\in \R^3 \backslash \p \Omega$ and that $A$ is positive-definite.\\
In the same manner, it is known that
\beas
\|\varphi\|^2_{\mathcal{H}^*} = \int_{\p \Omega}\varphi(x)\mathcal{S}_{\Omega}^A[\varphi](x)d\sigma(x) = -\int_{\R^3}\nabla\mathcal{S}_{\Omega}|[\varphi](x)|^2 d\sigma(x).
\eeas
Since $A$ is positive-definite we have
\beas
c\|\varphi\|^2_{\mathcal{H}^*}\leq \int_{\p \Omega}\varphi(x)\mathcal{S}_{\Omega}^A[\varphi](x)d\sigma(x) \leq C\|\varphi\|^2_{\mathcal{H}^*},
\eeas
for some constants $c$ and $C$.\\
Using the fact that $(\cdot,\cdot)_{\mathcal{H}^*}$ is equivalent to $(\cdot,\cdot)_{-\f{1}{2},\f{1}{2}}$, we get the desired result. 
\end{proof}

From \eqref{Anisotropic-syst max} we have $\phi = (\mathcal{S}_{\Omega}^A)^{-1}(\mathcal{S}_{\Omega}[\psi] + u^i)$, whereas, by Lemma \ref{lem-anisotropic_S_A max}, the following equation holds for $\psi$:
\be \label{eq-anisotropic_reduced_eq max}
\mathcal{Q}_A[\psi] = F
\ee
with
\be \label{eq-anisotropic_reduced_eq_explicit max}
\mathcal{Q}_A = \f{1}{2}\big(\eps_mId+(\mathcal{S}_{\Omega}^A)^{-1}\mathcal{S}_{\Omega}\big) + \big(\eps_m\mathcal{K}_{\Omega}^*-(\mathcal{K}_{\Omega}^A)^*(\mathcal{S}_{\Omega}^A)^{-1}\mathcal{S}_{\Omega}\big),
\ee
and
\beas
F = -\eps_m\df{\p u^i}{\p \nu} + \nu\cdot A\nabla\mathcal{S}^A_{\Omega}[(\mathcal{S}_{\Omega}^A)^{-1}u^i]\Big\vert_{-} .
\eeas

\begin{thm} \label{thm-Qa Fredholm max}
$\mathcal{Q}_A$ has a countable number of eigenvalues.
\end{thm}
\begin{proof}
It is clear that $(\mathcal{K}_{\Omega}^A)^*: H^{-\f{1}{2}}(\p \Omega) \rightarrow H^{-\f{1}{2}}(\p \Omega)$ is a compact operator. Hence, $\eps_m\mathcal{K}_{\Omega}^*-(\mathcal{K}_{\Omega}^A)^*(\mathcal{S}_{\Omega}^A)^{-1}\mathcal{S}_{\Omega}$ is compact as well. Therefore, only the invertibility of $\f{1}{2}\big(\eps_mId+(\mathcal{S}_{\Omega}^A)^{-1}\mathcal{S}_{\Omega}\big)$ needs to be proven.\\
Since $\mathcal{S}_{\Omega}^A$ is invertible, the invertibility of $\f{1}{2}\big(\eps_mId+(\mathcal{S}_{\Omega}^A)^{-1}\mathcal{S}_{\Omega}\big)$ is equivalent to that of $\eps_m\mathcal{S}_{\Omega}^A+\mathcal{S}_{\Omega}$.\\
Consider now, the bilinear form, for $(\varphi,\psi)\in (H^{-\f{1}{2}}(\p \Omega))^2$
\beas
B(\varphi,\psi) = -\eps_m\int_{\p \Omega}\varphi(x) \mathcal{S}_{\Omega}^A[\psi](x)d\sigma(x) - \int_{\p \Omega}\varphi(x)\mathcal{S}_{\Omega}[\psi](x) d\sigma(x).
\eeas
From Lemma \ref{lem-anisotrop Calderon max}, we have
\beas
B(\psi,\psi)\geq C\|\psi\|_{H^{-\f{1}{2}}(\p \Omega)},
\eeas
for some constant $C>0$.\\
It follows then, from the Lax-Milgram theorem that $\eps_m\mathcal{S}_{\Omega}^A+\mathcal{S}_{\Omega}$ is invertible in $H^{-\f{1}{2}}(\p \Omega)$, whence the result.
\end{proof}

Recall that the electromagnetic parameter of the problem, $A$, depends on the frequency, $\om$ of the incident field. Therefore the operator $\mathcal{Q}_A$ is frequency dependent and we should write $\mathcal{Q}_A(\om)$.\\
Following definition \ref{def-plasmonicFreq}, we say that $\om$ is a plasmonic resonance if
\beas
|\textnormal{eig}_j(\mathcal{Q}_A(\om))| \ll 1 \quad \mbox{ and is locally minimal for some $j\in \N$},
\eeas
where $\textnormal{eig}_j(\mathcal{Q}_A(\om))$ stands for the $j$-th eigenvalue of $\mathcal{Q}_A(\om)$.\\
Equivalently, we can say that $\om$ is a plamonic resonance if
\be \label{eq-plasmonic anisotropic max}
\om = \argmax_{\om}\|\mathcal{Q}^{-1}_A(\om)\|_{\mathcal{L}(\mathcal{H}^*(\p \Om))}.
\ee

From now on, we suppose that $A$ is an anisotropic perturbation of an isotropic parameter, i.e., $A = \eps_c(Id + P)$, with $P$ being a symmetric matrix and $\|P\|\ll1$.
\begin{lem}
Let $A = \eps_c(Id + \delta R)$, with $R$ being a symmetric matrix, $\|R\|=O(1)$ and $\delta \ll 1$. Let $\textnormal{Tr}$ denote the trace of a matrix. Then, as $\delta \rightarrow 0$,  we have the following asymptotic expansions:
\beas
\mathcal{S}_{\Omega}^A &=& \f{1}{\eps_c}\big(\mathcal{S}_{\Omega}+\delta\mathcal{S}_{{\Omega},1}+o(\delta)\big),\\
(\mathcal{S}_{\Omega}^A)^{-1} &=& \eps_c\big(\mathcal{S}_{\Omega}^{-1}+\delta\mathcal{B}_{{\Omega},1}+o(\delta)\big),\\
(\mathcal{K}_{\Omega}^A)^* &=& \mathcal{K}_{\Omega}^*+\delta\mathcal{K}_{{\Omega},1}^*+o(\delta)
\eeas
with
\beas
\mathcal{S}_{{\Omega},1}[\varphi](x) &=& -\f{1}{2}\textnormal{Tr}(R)\mathcal{S}_{{\Omega}}[\varphi](x)-\f{1}{2} \int_{\p {\Omega}}\f{\big(R(x-y),x-y\big)}{4\pi |x-y|^3}\varphi(y)d\sigma(y),\\
\mathcal{B}_{{\Omega},1} &=& -\mathcal{S}_{\Omega}^{-1}\mathcal{S}_{{\Omega},1}\mathcal{S}_{\Omega}^{-1},\\
\mathcal{K}_{{\Omega},1}^* &=& -\f{1}{2}\textnormal{Tr}(R)\mathcal{K}_{{\Omega}}^*[\varphi](x)-\f{3}{2} \int_{\p {\Omega}}\f{\big(R(x-y),x-y\big)\big(x-y,\nu(x)\big)}{4\pi |x-y|^5}\varphi(y)d\sigma(y). 
\eeas
\end{lem}
\begin{proof}
Recall that for $\delta$ small enough
\beas
\sqrt{(I+\delta R)^{-1}} &=& Id-\f{\delta}{2}R
+ O(\delta^2),\\
\textnormal{det}(I + \delta R)= &=& 1 + \delta\textnormal{Tr}(R) + o(\delta),\\
(1+\delta x + o(\delta))^{s} &=& 1 + \delta sx + o(\delta), \quad s\in \R.
\eeas
The results follow then from asymptotic expansions of
$- \df{1}{4\pi \sqrt{\textnormal{det}(A)} |A_*x |^{\beta}}$, $\beta=1,3$ and the identity
\beas
(\mathcal{S}_{\Omega}^A)^{-1} = \eps_c(Id + \delta\mathcal{S}_{\Omega}^{-1}\mathcal{S}_{\Omega,1} + o(\delta))^{-1}\mathcal{S}_{\Omega}^{-1}.
\eeas
\end{proof}
Plugging the expressions above into the expression of $\mathcal{Q}_A$ we get the following result.
\begin{lem} \label{lem-Asymp Q_A max}
As $\delta \rightarrow 0$, the operator $\mathcal{Q}_A$ has the following asymptotic expansion
\beas
\mathcal{Q}_A = \mathcal{Q}_{A,0} + \delta\mathcal{Q}_{A,1} + o(\delta),
\eeas
where
\beas
\mathcal{Q}_{A,0} &=& \f{\eps_m+\eps_c}{2}Id + (\eps_m-\eps_c)\mathcal{K}_{\Omega}^*,\\
\mathcal{Q}_{A,1} &=& \eps_c\big((\f{1}{2}Id-\mathcal{K}_{\Omega}^*)\mathcal{B}_{{\Omega},1}\mathcal{S}_{\Omega}-\mathcal{K}_{{\Omega},1}^*\big).
\eeas
\end{lem}
We regard the operator $\mathcal{Q}_A$ as a perturbation of $\mathcal{Q}_{A,0}$. As in section 
\ref{sec-Layer potential fomulation max}, we use the standard perturbation theory to derive the perturbed eigenvalues and eigenvectors in $\mathcal{H}^*(\p {\Omega})$.

Let $(\lambda_j,\varphi_j) $ be the eigenvalue and normalized eigenfunction pairs of $\mathcal{K}_{\Omega}^*$ in $\mathcal{H}^*(\p {\Omega})$ and $\tau_j$ the eigenvalues of $\mathcal{Q}_{A,0}$. We have $\tau_j = \f{\eps_m+\eps_c}{2} + (\eps_m-\eps_c)\lambda_j$.\\
For simplicity, we consider the case when $\lambda_j$ is a  simple eigenvalue of the operator
$\mathcal{K}_{\Omega}^*$.
Define
\beas
P_{j,l} = (\mathcal{Q}_{A,1}[\varphi_j],\varphi_l)_{\mathcal{H}^*}.
\eeas

As $\delta \rightarrow 0$, the perturbed eigenvalue and eigenfunction have the following form:
\beas
\tau_j (\delta) &=& \tau_j + \delta \tau_{j, 1}+ o(\delta), \label{tau-single anist max} \\
\varphi_j(\delta) &=& \varphi_j + \delta \varphi_{j, 1} + o(\delta), \label{eigenfun-single anist max}
\eeas
where
\beas
\tau_{j, 1} &=& P_{j j}, \label{tau_j2 anist max}\\
\varphi_{j, 1}&=& \sum_{l\neq j} \f{P_{jl}}{ \big( \eps_m -  \eps_c \big) (\lambda_j- \lambda_l)} \varphi_l.
\eeas

\section{A Maxwell-Garnett theory for plasmonic nanoparticles}
\label{sec-MaxGar max}
In this subsection we derive effective properties of a system of plasmonic nanoparticles. To begin with, we consider a bounded and simply connected domain $\Omega\Subset\mathbb{R}^3$ of class $\mathcal{C}^{1,\alpha}$ for $0<\alpha<1$, filled with a composite material that consists of a matrix of constant electric permittivity $\eps_m$ and a set of periodically distributed plasmonic nanoparticles 
with (small) period $\eta$ and  electric permittivity $\eps_c$.\\
Let $Y=]-1/2,1/2[^3$ be the unit cell and denote $\delta = \eta^{\beta}$ for $\beta>0$. We set the (rescaled) periodic function
\beas
\gamma = \eps_m \chi(Y \backslash \bar{D}) + \eps_c\chi(D),
\eeas
where $D=\delta B$ with $B\Subset\mathbb{R}^3$ being of class $\mathcal{C}^{1,\alpha}$ and the volume of $B$, $|B|$, is assumed to be equal to $1$. Thus, the electric permittivity of the composite is given by the periodic function
\beas
\gamma_{\eta}(x) = \gamma(x/{\eta}),
\eeas
which has period $\eta$. 
Now, consider the problem
\begin{equation}
\label{poriginal}
\nabla\cdot\gamma_{\eta}\nabla u_{\eta} = 0 \quad \textnormal{in }\Omega
\end{equation}
with an appropriate boundary condition on $\partial \Omega$. 
Then, 
there exists a homogeneous, generally anisotropic, permittivity $\gamma^*$, such that the replacement, as $\eta \rightarrow 0$, of the original equation (\ref{poriginal}) by  
\beas
\nabla\cdot\gamma^*\nabla u_0 = 0 \quad \textnormal{in }\Omega
\eeas
is a valid approximation in a certain sense.  
The coefficient $\gamma^*$  is called an effective permittivity. It represents the overall macroscopic material property of the periodic composite made of plasmonic nanoparticles embedded in an isotropic matrix.

The (effective) matrix $\gamma^*=(\gamma^*_{pq})_{p,q=1,2,3}$ is defined by \cite{book2}
\beas
\gamma^*_{pq} = \int_{Y} \gamma(x) \nabla u_p(x)\cdot \nabla u_q(x)dx,
\eeas
where $u_p$, for $p=1,2,3$, is the unique solution to the cell problem
\be \label{eq-def_effective conductivity max}
\left\{
\begin{array}{lr}
\nabla\cdot\gamma\nabla u_{p} = 0 \quad \textnormal{in }Y,  \\
\nm
u_p-x_p \quad \textnormal{periodic (in each direction) with period 1} , \\
\nm
\int_Y u_p(x)dx = 0.
\end{array} \right.
\ee
Using Green's formula, we can rewrite $\gamma^*$ in the following form:
\be \label{eq-def_effective conductivity 2 max}
\gamma^*_{pq} = \eps_m \int_{\p Y} u_q(x)\f{\p u_p}{\p \nu}(x) d\sigma(x).
\ee
The matrix $\gamma^*$ depends on $\eta$ as a parameter and cannot be written explicitly.

The following lemmas are from \cite{book2}.
\begin{lem} \label{thm-solution to cell trans max}
For $p=1,2,3$, problem \eqref{eq-def_effective conductivity max} has a unique solution $u_p$ of the form
\beas
u_p(x) = x_p + C_p +\mathcal{S}_{D\sharp}(\lambda_\eps Id - \mathcal{K}_{D\sharp}^*)^{-1}[\nu_p](x) \quad \textnormal{in }Y,
\eeas
where $C_p$ is a constant, $\nu_p$ is the $p$-component of the outward unit normal to $\p D$, $\lambda_\eps$ is defined by 
(\ref{eq-lbda_eps,mu max}),  and
\beas
\mathcal{S}_{D\sharp}[\varphi](x) &=& \int_{\p D}G_{\sharp}(x,y)\varphi(y)d\sigma(y), \\
\mathcal{K}_{D\sharp}^*[\varphi](x) &=& \int_{\p D} \f{\p G_{\sharp}(x,y)}{\p \nu(x)}\varphi(y)d\sigma(y)
\eeas
with $G_{\sharp}(x,y)$ being the periodic Green function defined by 
$$
G_{\sharp}(x,y) = - \sum_{n \in \mathbb{Z}^3\setminus \{0\}} 
\frac{e^{i 2\pi n \cdot (x-y)}}{4 \pi^2 |n|^2}. 
$$
\end{lem}

\begin{lem} \label{lem-jump formula periodic max}
Let $\mathcal{S}_{D\sharp}$ and $\mathcal{K}_{D\sharp}^*$ be the operators defined as in Lemma \ref{thm-solution to cell trans max}. Then the following trace formula holds on $\p D$
\beas
(\pm\f{1}{2}Id+\mathcal{K}_{D\sharp}^*)[\varphi] = \f{\p \mathcal{S}_{D\sharp}[\varphi]}{\p \nu}\Big\vert_{\pm}.
\eeas
\end{lem}
For the sake of simplicity,  for  $p=1,2,3$, we set
\be \label{eq-phi_p anisotropic max}
\phi_p(y) = (\lambda_\eps Id - \mathcal{K}_{D\sharp}^*)^{-1}[\nu_p](y) \quad \textnormal{for $y$ in }\p D
\ee
Thus, from Lemma \ref{thm-solution to cell trans max}, we get
\beas
\gamma^*_{pq} = \eps_m \int_{\p Y}\big(y_q + C_q + \mathcal{S}_{D\sharp}[\phi_q](y)\big)\f{\p \big(y_p + \mathcal{S}_{D\sharp}[\phi_p](y)\big)}{\p \nu}d\sigma(y).
\eeas
Because of the periodicity of $\mathcal{S}_{D\sharp}[\phi_p]$, we get
\begin{equation}
\label{gams}
\gamma^*_{pq} = \eps_m\Big(\delta_{pq}+\int_{\p Y}y_q\f{\p \mathcal{S}_{D\sharp}[\phi_p]}{\p \nu}(y)d\sigma(y)\Big).
\end{equation}
In view of the periodicity of $\mathcal{S}_{D\sharp}[\phi_p]$, the divergence theorem applied on $Y\backslash\bar{D}$ and Lemma \ref{lem-jump formula periodic max} yields (see \cite{book2})
\beas
\int_{\p Y}y_q\f{\p \mathcal{S}_{D\sharp}[\phi_p]}{\p \nu}(y) = \int_{\p D}y_q\phi_p(y)d\sigma(y).
\eeas
Let
\beas
\psi_p(y) = \phi_p(\delta y) \quad \textnormal{for $y \in \p B$}.
\eeas
Then, by (\ref{gams}), we obtain
\begin{equation} \label{gf}
\gamma^* = \eps_m(Id + fP),
\end{equation}
where $f = |D| = \delta^3 (= \eta^{3\beta})$ is the volume fraction of $D$ and $P = (P_{pq})_{p,q = 1,2,3}$ is given by
\begin{equation} 
\label{defP}
P_{pq} = \int_{\p B}y_q\psi_p(y)d\sigma(y).
\end{equation}
To proceed with the computation of $P$ we will need the following Lemma \cite{book2}. 
\begin{lem} \label{lem-Green periodic max}
There exists a smooth function $R(x)$ in the unit cell $Y$ such that
\beas
G_{\sharp}(x,y) = - \f{1}{4\pi|x-y|}+R(x-y).
\eeas
Moreover, the Taylor expansion of $R(x)$ at $0$ is given by
\beas
R(x) = R(0)-\f{1}{6}(x_1^2+x_2^2+x_3^2) + O(|x|^4).
\eeas
\end{lem}
Now we can prove the main result of this section, which shows the validity of the Maxwell-Garnett theory uniformly with respect to the frequency under the assumptions that 
\be \label{fassumption}
f \ll \mathrm{dist}(\lambda_\eps(\omega), \sigma(\mathcal{K}_B^*))^{3/5} \quad \mbox{and } 
(Id-\delta^3R_{\lambda_\eps(\omega)}^{-1}T_0)^{-1} = O(1),
\ee
where $R_{\lambda_\eps(\omega)}^{-1}$ and $T_0$ are to be defined and $\mathrm{dist}(\lambda_\eps(\omega),\sigma(\mathcal{K}_D^*))$ is the 
distance between $\lambda_\eps(\omega)$ and the spectrum  of $\mathcal{K}_B^*$. 
\begin{thm} \label{effective} Assume that (\ref{fassumption}) holds. Then we have
\begin{equation} \label{MGF}
\gamma^* = \eps_m\big(Id + fM (Id - \f{f}{3}M)^{-1} \big)+O\Big(\f{f^{8/3}}{\mathrm{dist}(\lambda_\eps(\omega),\sigma(\mathcal{K}_B^*))^2}\Big)
\end{equation}
uniformly in $\omega$. Here, $M = M(\lambda_\eps(\omega),B)$ is the polarization tensor  \eqref{defm max} associated with $B$ and $\lambda_\eps(\omega)$.
\end{thm}
\begin{proof}
In view of Lemma \ref{lem-Green periodic max} and \eqref{eq-phi_p anisotropic max}, we can write, for $x\in \p D$, 
\beas
(\lambda_\eps(\omega) Id-\mathcal{K}_D^*)[\phi_p](x)-\int_{\p D}\f{\p R(x-y)}{\p \nu(x)}\phi_p(y)d\sigma(y) = \nu_p(x),
\eeas
which yields, for $x\in \p B$, 
\beas
(\lambda_\eps(\omega) Id-\mathcal{K}_B^*)[\psi_p](x)-\delta^2\int_{\p B}\f{\p R(\delta(x-y))}{\p \nu(x)}\psi_p(y)d\sigma(y) = \nu_p(x).
\eeas
By virtue of Lemma \ref{lem-Green periodic max}, we get
\beas
\nabla R(\delta(x-y)) = -\f{\delta}{3}(x-y)+O(\delta^3)
\eeas
uniformly in $x,y\in \p B$. Since $\int_{\p B}\psi_p(y)d\sigma(y)=0$, we now have
\beas
(R_{\lambda_\eps(\omega)}-\delta^3T_0+\delta^5T_1)[\psi_p](x)=\nu_p(x),
\eeas
and so
\be \label{eq-phi_p MaxGar max}
(Id-\delta^3R_{\lambda_\eps(\omega)}^{-1}T_0+\delta^5R_{\lambda_\eps(\omega)}^{-1}T_1)[\psi_p](x)=R_{\lambda_\eps(\omega)}^{-1}[\nu_p](x),
\ee
where
\beas
R_{\lambda_\eps(\omega)}[\psi_p](x) &=& (\lambda_\eps(\omega) Id-\mathcal{K}_B^*)[\psi_p](x),\\
T_0[\psi_p](x) &=& \f{\nu(x)}{3}\cdot\int_{\p B}y\psi_p(y)d\sigma(y),\\
\|T_1\|_{\mathcal{L}(\mathcal{H}^*(\p B))} &=& O(1).
\eeas
Since $\mathcal{K}_B^*$ is a compact self-adjoint operator in $\mathcal{H}^*(\p B)$ it follows that \cite{Gil}
\be \label{eq-operator estimator MaxGar max}
\|(\lambda_\eps(\omega) Id-\mathcal{K}_B^*)^{-1}\|_{\mathcal{L}(\mathcal{H}^*(\p B))} \leq \f{c}{\mathrm{dist}(\lambda_\eps(\omega),\sigma(\mathcal{K}_B^*))} 
\ee
for a constant $c$.\\
It is clear that $T_0$ is a compact operator. From the fact that
the imaginary part of $R_{\lambda_\eps(\omega)}$ is nonzero, it follows that $Id-\delta^3R_{\lambda_\eps(\omega)}^{-1}T_0$ is invertible.\\
Under the assumption
\beas
(Id-\delta^3R_{\lambda_\eps(\omega)}^{-1}T_0)^{-1} = O(1),\\
\delta^5 \ll  \mathrm{dist}(\lambda_\eps(\omega),\sigma(\mathcal{K}_B^*)).
\eeas
and \eqref{eq-phi_p MaxGar max}, \eqref{eq-operator estimator MaxGar max} we get
\beas
\psi_p(x)&=&(Id-\delta^3R_{\lambda_\eps(\omega)}^{-1}T_0+\delta^5R_{\lambda_\eps(\omega)}^{-1}T_1)^{-1}R_{\lambda_\eps(\omega)}^{-1}[\nu_p](x)\\
&=& (Id-\delta^3R_{\lambda_\eps(\omega)}^{-1})^{-1}R_{\lambda_\eps(\omega)}^{-1}[\nu_p](x) + O\Big(\f{\delta^5}{\mathrm{dist}(\lambda_\eps(\omega),\sigma(\mathcal{K}_B^*))}\Big).
\eeas
Therefore, we obtain and estimate for $\psi_p$
\beas
\psi_p = O\Big(\f{1}{\mathrm{dist}(\lambda_\eps(\omega),\sigma(\mathcal{K}_B^*))}\Big).
\eeas
Now we multiply \eqref{eq-phi_p MaxGar max} by $y_q$ and integrate over $\p B$. We can derive from the estimate of $\psi_p$ that
\beas
P(Id-\f{f}{3}M) = M + O\Big(\f{\delta^5}{\mathrm{dist}(\lambda_\eps(\omega),\sigma(\mathcal{K}_B^*))^2}\Big),
\eeas
and therefore,
\beas
P = M(Id+\f{f}{3}M)^{-1}+O\Big(\f{\delta^5}{\mathrm{dist}(\lambda_\eps(\omega),\sigma(\mathcal{K}_B^*))^2}\Big),
\eeas
with $P$ being defined by (\ref{defP}). 
Since $f= \delta^3$ and 
$$M = O\Big(\f{\delta^3}{\mathrm{dist}(\lambda_\eps(\omega),\sigma(\mathcal{K}_B^*))}\Big),$$
it follows from (\ref{gf}) that the Maxwell-Garnett formula (\ref{MGF}) holds (uniformly in the frequency $\omega$) under the assumption (\ref{fassumption}) on the volume fraction $f$. 
\end{proof}

\begin{rmk}
As a corollary of Theorem \ref{effective}, we see that in the case when $fM=O(1)$, which is equivalent to the scale $f=O\Big(\mathrm{dist}(\lambda_\eps(\omega),\sigma(\mathcal{K}_B^*))\Big)$, the matrix $fM(Id-\f{f}{3}M)^{-1}$ may have a negative-definite symmetric real part. This implies that the effective medium is plasmonic as well as anisotropic.
\end{rmk}

\begin{rmk}
It is worth emphasizing that Theorem \ref{effective} does not only prove the validity of the Maxwell-Garnett theory but it can also be used 
together with the results in section \ref{sec-Anisitrop max} in order to derive the plasmonic resonances of the effective medium made of a dilute system of arbitrary-shaped plasmonic nanoparticles, following \eqref{eq-plasmonic anisotropic max}
\beas
\om = \argmax_{\om}\|\mathcal{Q}^{-1}_{\gamma^*}(\om)\|_{\mathcal{L}(\mathcal{H}^*(\p \Om))}.
\eeas 

\end{rmk}


\begin{thebibliography} {99}


\bibitem{book1} H. Ammari. \textsl{An Introduction to Mathematics of Emerging Biomedical Imaging},  Math. \& Appl., Volume 62, Springer, Berlin, 2008.


\bibitem{yatin}
 H. Ammari, Y.T. Chow, K. Liu, and J. Zou, Optimal shape design by partial spectral data, SIAM J. Sci. Comput., to appear. 

\bibitem{scatcoef} H. Ammari, Y. Chow, and J. Zou, Super-resolution in highly contrasted media from the perspective of scattering coefficients, arXiv: 1410.1253.

\bibitem{ACKLM2} H. Ammari, G. Ciraolo, H. Kang, H. Lee, and G.W. Milton, Spectral theory
of a Neumann-Poincar\'e-type operator and analysis of anomalous localized resonance II, Contemp. Math., 615 (2014), 1--14.

\bibitem{ciraolo}  H. Ammari, G. Ciraolo, H. Kang, H. Lee, and K. Yun, Spectral analysis of the Neumann-Poincar\'e operator and characterization of the stress concentration in anti-plane elasticity, Arch. Ration. Mech. Anal.,  208  (2013),   275--304.

\bibitem{pierre}
 H. Ammari, Y. Deng, and P. Millien,
Surface plasmon resonance of nanoparticles and applications in imaging, Arch. Ration. Mech. Anal., DOI:10.1007/s00205-015-0928-0. 


\bibitem{book3}  H. Ammari, J. Garnier, W. Jing, H. Kang, M. Lim, K. S\o lna, and H. Wang,
\textsl{Mathematical and Statistical Methods for Multistatic Imaging}, Lecture Notes in Mathematics, Volume 2098, Springer, Cham, 2013.

\bibitem{iakovleva} H. Ammari, E. Iakovleva, D. Lesselier, and G. Perrusson, 
MUSIC-type electromagnetic imaging of a collection of small three-dimensional
inclusions, SIAM J. Sci. Comp.,  29 (2007), 674--709.
      
      
\bibitem{book2} {H. Ammari and H. Kang}, \textsl{Polarization and Moment
Tensors with Applications to Inverse Problems and Effective Medium
Theory}, Applied Mathematical Sciences, Vol. 162, Springer-Verlag,
New York, 2007.      
      
\bibitem{lim} H. Ammari, H. Kang, M. Lim, and H. Lee, Enhancement of near-cloaking. Part II: The Helmholtz equation, Comm. Math. Phys., 317  (2013),   485--502.

\bibitem{lim2} H. Ammari, H. Kang, H. Lee, M. Lim, and S. Yu,  Enhancement of near cloaking for the full Maxwell equations, SIAM J. Appl. Math.,  73  (2013),  2055-–2076. 

\bibitem{matias} H. Ammari, P. Millien, M. Ruiz, and H. Zhang, 
Mathematical analysis of plasmonic nanoparticles: the scalar  
case,  arXiv:1506.00866. 

\bibitem{triki2004} H. Ammari and F. Triki, Splitting of resonant and scattering frequencies under shape deformation, J. Diff. Equat.,   202  (2004), 231--255.

\bibitem{hai} H. Ammari and H. Zhang, A mathematical theory of super-resolution by using a system of sub-wavelength Helmholtz resonators. Comm. Math. Phys.,
337 (2015),  379--428.


\bibitem{hai2} H. Ammari and H. Zhang, Super-resolution in high contrast media, 
Proc. Royal Soc. A, 2015 (471), 20140946. 


\bibitem{kang1} K. Ando and H. Kang, Analysis of plasmon resonance on smooth domains using spectral properties of the Neumann-Poincar\'e operator, arXiv:1412.6250.

\bibitem{hyeonbae} K. Ando, H. Kang, and H. Liu, Plasmon resonance with finite frequencies: a validation of the quasi-static approximation for diametrically small inclusions, arXiv: 1506.03566.

\bibitem{physics1}
S. Arhab, G. Soriano, Y. Ruan, G. Maire, A. Talneau, D. Sentenac, P.C. Chaumet, K. Belkebir, and H. Giovannini, {Nanometric resolution with far-field optical profilometry},
Phys. Rev. Lett., 111 (2013), 053902.

\bibitem{baffou2010mapping} {G. Baffou, C. Girard, and R. Quidant},
 {Mapping heat origin in plasmonic structures},
 {Phys. Rev. Lett.},
  {104} (2010),
  {136805}.

\bibitem{gang1b} G. Bao and P. Li, Near-field imaging of infinite rough surfaces, SIAM J. Appl. Math., 73  (2013),  2162--2187.


\bibitem{gang2b} G. Bao and J. Lin, Near-field imaging of the surface displacement on an infinite ground plane, Inverse Probl. Imaging,  7 (2013),  377--396.

\bibitem{gang} G. Bao, J. Lin, and F. Triki, A multi-frequency inverse source problem, J. Diff. Equat.,   249  (2010), 3443--3465.


\bibitem{triki} E. Bonnetier and F. Triki, On the spectrum of the Poincar\'e variational problem for two close-to-touching inclusions in 2D, Arch. Ration. Mech. Anal.,  209  (2013),   541--567.

\bibitem{born1999principles}
M. Born and E. Wolf, \textsl{Principles of Optics: Electromagnetic Theory of Propagation, 
Interference and Diffraction of Light}, Cambridge Univ. Press, 1999.


\bibitem{buffa2002traces}
A.~Buffa, M.~Costabel, and D.~Sheen,
\newblock On traces for $H(\mathrm{curl}, \Omega)$ in lipschitz domains.
\newblock { J. Math. Anal. Appl.},
276 (2002), 845--867.

\bibitem{Gri12}
D. Grieser, The plasmonic eigenvalue problem, Rev. Math. Phys.  26  (2014),   1450005.

\bibitem{griesmaier2008asymptotic}
R. Griesmaier,
\newblock An asymptotic factorization
method for inverse electromagnetic scattering in layered media,
SIAM J. Appl. Math., 68 (2008), 1378--1403.

\bibitem{plasmon4} P.K. Jain, K.S. Lee, I.H. El-Sayed, and M.A. El-Sayed, Calculated absorption and scattering properties of gold nanoparticles of different size, shape, and composition: Applications in biomedical imaging and biomedicine, J. Phys. Chem. B, 110 (2006), 7238--7248.


\bibitem{eth} R. Giannini, C.V. Hafner, and J.F. L\"offler, 
Scaling behavior of individual nanoparticle plasmon resonances, J. Phys. Chem. C, 119 (2015), 6138--6147. 

\bibitem{Gil} M. I. Gil, Norm Estimations for Operator Valued Functions and Applications,
Vol. 192. CRC Press, 1995.

\bibitem{kang2} H. Kang, K. Kim, H. Lee, and J. Shin, Spectral properties of the Neumann-Poincar\'e operator and uniformity of estimates for the conductivity equation with complex coefficients, arXiv:1406.3873.


\bibitem{kang3} H. Kang, M. Lim, and S. Yu, Spectral resolution of the Neumann-Poincar\'e operator on intersecting disks and analysis of plamson resonance, arXiv:1501.02952.


\bibitem{kato} T. Kato, 
\textsl{Perturbation Theory for Linear Operators} (2nd ed.),
Springer-Verlag, Berlin, 1980.


\bibitem{shapiro} D. Khavinson, M. Putinar, and H.S. Shapiro, Poincar\'e's variational problem
in potential theory, Arch. Rational Mech. Anal., 185 (2007) 143--184.

\bibitem{kelly} K.L. Kelly, E. Coronado, L.L. Zhao, and G.C. Schatz, 
The optical properties of metal nanoparticles: The influence of size, shape, and dielectric environment, J. Phys. Chem. B, 107 (2003), 668--677. 

\bibitem{Klinken} L. Klinkenbusch, \textsl{Scattering of an arbitrary plane electromagnetic wave
by a finite elliptic cone}, Archiv f\"{u}r Elektrotechnik 76 (1993), 181-193.


\bibitem{link} S. Link and M.A. El-Sayed, Shape and size dependence of radiative, non-radiative and photothermal properties of gold nanocrystals, Int. Rev. Phys. Chem., 19 (2000), 409--453. 

\bibitem{plasmon1} I.D. Mayergoyz, D.R. Fredkin, and Z. Zhang, Electrostatic (plasmon) resonances in nanoparticles, Phys. Rev. B, 72 (2005), 155412.

\bibitem{plasmon3} I.D. Mayergoyz and Z. Zhang, Numerical analysis of plasmon resonances in nanoparticules, IEEE Trans. Mag., 42 (2006), 759--762.

\bibitem{miller} O.D. Miller, C.W. Hsu, M.T.H. Reid, W. Qiu, B.G. DeLacy, J.D. Joannopoulos, M. Soljaci\'c, and S. G. Johnson, Fundamental limits to extinction by metallic nanoparticles, Phys. Rev. Lett., 112 (2014), 123903.

\bibitem{N} J.-C. N\'ed\'elec, \textsl{Acoustic and Electromagnetic Equations: Integral Representations for Harmonic
Problems}, Springer-Verlag, New York, 2001.

\bibitem{novotny} S. Palomba, L. Novotny, and R.E. Palmer,
Blue-shifted plasmon resonance of individual size-selected
gold nanoparticles, Optics Commun.,  281 (2008), 480--483.


\bibitem{berry} M. Reed and B. Simon, \textsl{Methods of Modern Mathematical Physics. IV Analysis of Operators}, Academic Press, New York, 1970. 

\bibitem{SC10}
D. Sarid and W. A. Challener, \textsl{Modern Introduction to Surface Plasmons: Theory, Mathematical Modeling, and Applications},
Cambridge University Press, New York, 2010.

\bibitem{tocho} L.B. Scaffardi and J.O. Tocho,
Size dependence of refractive index of gold
nanoparticles, Nanotech., 17 (2006), 1309--1315.

\bibitem{Torres} R. H. Torres, \textsl{Maxwell's equations and dielectric obstacles
with lipschitz boundaries}, J. London Math. Soc. (2) 57 (1998), 157-169.


\end{thebibliography}
\end{document}